\documentclass[11pt,hyp,]{nyjm}
\usepackage{stix2}
\usepackage{hyperref}
\hypersetup{nesting=true,debug=true,naturalnames=true}
\usepackage{graphicx,amssymb,upref}

\let\<\langle
\let\>\rangle
\newcommand{\doi}[1]{doi:\,\href{http://dx.doi.org/#1}{#1}}
\newcommand{\mailurl}[1]{\email{\href{mailto:#1}{#1}}}
\let\uml\"

\usepackage{amsmath}
\usepackage{amsthm}
\usepackage[utf8]{inputenc}
\usepackage{bm}
\usepackage{tikz-cd}
\tikzcdset{arrow style=math font}
\usepackage{etoolbox}

\title[On the global homotopy theory of symmetric monoidal categories]{On the global homotopy theory of\\ symmetric monoidal categories}
\author{Tobias Lenz}
\address{Mathematisches Institut, Rheinische Friedrich-Wilhelms-Universit\"at Bonn, Endenicher Allee 60, 53115 Bonn, Germany \& Max-Planck-Institut für Mathematik, Vivatsgasse 7, 53111 Bonn, Germany}
\curraddr{\scshape Mathematical Institute, University of Utrecht, Budapestlaan 6, 3584 CD Utrecht, The Netherlands}
\mailurl{t.lenz@uu.nl}

\newtheorem{thm}{Theorem}[section]
\newtheorem{lemma}[thm]{Lemma}
\newtheorem{cor}[thm]{Corollary}
\newtheorem{prop}[thm]{Proposition}

\newtheorem{introthm}{Theorem}

\theoremstyle{definition}
\newtheorem{rk}[thm]{Remark}
\newtheorem{constr}[thm]{Construction}
\newtheorem{warn}[thm]{Warning}
\newtheorem{ex}[thm]{Example}
\newtheorem{nex}[thm]{Non-example}

\newtheorem*{claim*}{Claim}
\AtBeginEnvironment{claim*}{\let\oldqed=\qedsymbol%
\renewcommand\qedsymbol{\scriptsize$\triangle$}}
\AtEndEnvironment{claim*}{\let\qedsymbol=\oldqed}

\newtheorem{defi}[thm]{Definition}

\numberwithin{equation}{section}

\newcommand{\nerve}{\textup{N}}
\newcommand{\h}{\textup{h}}
\newcommand{\cat}[1]{\textbf{\textup{#1}}}

\newcommand{\blank}{{\textup{--}}}
\newcommand{\pr}{{\textup{pr}}}
\newcommand{\Hom}{{\textup{Hom}}}
\newcommand{\id}{{\textup{id}}}
\newcommand{\ev}{{\textup{ev}}}
\newcommand{\im}{\mathop{\textup{im}}}

\newcommand{\Fun}{\textup{Fun}}
\newcommand{\Inj}{\textup{Inj}}
\newcommand{\supp}{\mathop{\textup{supp}}\nolimits}
\newcommand{\Maps}{\mathord{\textup{maps}}}
\newcommand{\Ob}{\mathop{\textup{Ob}}}
\newcommand{\myh}{\mathord{\textup{`$\mskip-.1\thinmuskip h\mskip-.4\thinmuskip$'}}}
\newcommand{\sat}{{\textup{sat}}}
\newcommand{\forget}{\mathop{\textup{forget}}}

\newcommand{\bmM}{\hbox{\hfuzz=20pt\setbox0=\hbox to 0pt{$\mathcal M$}\kern-.175pt\copy0\kern.175pt\copy0\kern.175pt\copy0\kern.175pt$\mathcal M$\kern-.175pt}}
\newcommand{\bmEM}{{\bm{E}\bmM}}

\let\oldboxtimes=\boxtimes
\renewcommand{\boxtimes}{\mathchoice{\mathbin{\raise1.5pt\hbox{\scriptsize$\oldboxtimes$}}}%
{\mathbin{\raise1.5pt\hbox{\scriptsize$\oldboxtimes$}}}%
{\mathbin{\raise.75pt\hbox{\scriptsize$\scriptstyle\oldboxtimes$}}}%
{\mathbin{\raise.75pt\hbox{\scriptsize$\scriptscriptstyle\oldboxtimes$}}}}

\begin{document}
\begin{abstract}
\emph{Parsummable categories} were introduced by Schwede as input for his global algebraic $K$-theory construction. We prove that their whole homotopy theory with respect to the so-called \emph{global equivalences} can already be modelled by the more mundane symmetric monoidal categories.

In another direction, we show that the resulting homotopy theory is also equivalent to the homotopy theory of a certain simplicial analogue of parsummable categories, that we call \emph{parsummable simplicial sets}. These form a bridge to several concepts of `globally coherently commutative monoids' like ultra-commutative monoids and global $\Gamma$-spaces, that we explore in \cite{g-global}.
\end{abstract}

\maketitle
\setcounter{tocdepth}{1}
\tableofcontents

\section*{Introduction}
The \emph{algebraic $K$-theory} of rings encodes information about a wide range of phenomena in number theory, geometry, and other areas of pure mathematics. While historically the roots of the subject lie in algebra, Quillen's construction \cite{quillen-plus} of the $K$-groups of a ring $R$ is decidedly homotopy theoretic in nature: he first assigns to $R$ an infinite loop space (or, in modern interpretation, a connective spectrum) $\textbf{K}(R)$, and the $K$-groups are then only obtained in a second step as the homotopy groups of it.

Quillen's second construction \cite{quillen-theorem-A} made it clear that algebraic $K$-theory does not really depend on the ring $R$ itself, but only on its module category. Building on this observation, algebraic $K$-theory was soon extended to more general categorical inputs. In particular, May \cite{may-permutative} constructed the algebraic $K$-theory of small symmetric monoidal categories, and an equivalent construction was later given by Shimada and Shimakawa \cite{shimada-shimakawa}; more precisely, they show how symmetric monoidal categories yield \emph{special $\Gamma$-spaces} in the sense of Segal, which we can think of as `commutative monoids up to coherent systems of homotopies.' Segal's delooping machinery \cite{segal-gamma} then associates to each (special) $\Gamma$-space a connective spectrum, and together this yields the $K$-theory $\textbf{K}(\mathscr C)$ of a symmetric monoidal category $\mathscr C$. If $R$ is a ring, then applying $\textbf K$ to a skeleton of the symmetric monoidal category of finitely generated projective $R$-modules and $R$-linear isomorphisms under direct sum recovers the usual $K$-theory of $R$.

\subsection*{\texorpdfstring{$\bm K$}{K}-theory as group completion} A particularly striking structural insight on the $K$-theory of symmetric monoidal categories is Thomason's result \cite[Theorem~5.1 and Lemma~1.9.2]{thomason} that $\textbf{K}$ exhibits the homotopy category of connective spectra as localization of the category of small symmetric monoidal categories.

This result was later refined by Mandell~\cite[Theorem~1.4]{mandell}, who showed that already the intermediate passage to the homotopy category of special $\Gamma$-spaces is a localization, i.e.~symmetric monoidal categories model all `coherently commutative monoids in spaces.' Thomason's original result then follows from this via Segal's comparison between the homotopy theories of (special) $\Gamma$-spaces and connective spectra \cite[Proposition~3.4]{segal-gamma}.

More precisely, Segal shows (in modern language) that the passage from special $\Gamma$-spaces to connective spectra is a Bousfield localization and that it identifies the homotopy category of connective spectra with the one of the \emph{very special} (or \emph{grouplike}) $\Gamma$-spaces. Together with Mandell's result we can view this as a precise formulation of the slogan that $K$-theory is `higher group completion,' just like $K_0$ can be defined as an ordinary group completion.

\subsection*{Equivariant and global algebraic \texorpdfstring{$\bm K$}{K}-theory} The study of \emph{$G$-equivariant algebraic $K$-theory} for a fixed finite group $G$ already began in the 80's, but recent years have seen a renewed interest in it, for example through the work of Merling \cite{merling} and her coauthors \cite{equivariant-A, mmo}.

One possible approach to the subject goes back to Shimakawa \cite{shimakawa}, who developed \emph{$\Gamma$-$G$-spaces} as a $G$-equivariant generalization of Segal's $\Gamma$-spaces, and used this machinery to construct the equivariant $K$-theory $\textbf{K}_G(\mathscr C)$ of a small symmetric monoidal category $\mathscr C$ with a suitable $G$-action. Here $\textbf{K}_G(\mathscr C)$ is a $G$-spectrum in the sense of $G$-equivariant stable homotopy theory; we emphasize that this theory is richer than the na\"ive homotopy theory of $G$-objects in spectra, and similarly for $\Gamma$-$G$-spaces. In particular, Shimakawa's result is not simply a consequence of functoriality of the usual non-equivariant $K$-theory constructions---for example, we can extract from $\textbf K_G(\mathscr C)$ not only a single  $\mathbb N$-graded $K$-group, but in fact one for each subgroup $H\subset G$, and these graded abelian groups are connected by additional structure maps providing them with the structure of a so-called \emph{$G$-Mackey functor}.

In this sense it turns out that the $G$-equivariant algebraic $K$-theory of a symmetric monoidal category $\mathscr C$ with \emph{trivial} $G$-action already contains interesting additional information. If we fix $\mathscr C$ and vary $G$, this yields a family of equivariant $K$-theory spectra associated to $\mathscr C$, which are related by suitable change-of-group maps. A rigorous framework meant to capture the notion of such families is \emph{global stable homotopy theory} in the sense of \cite{schwede-book}, and it is therefore natural to ask whether we can collect all this equivariant information in a single \emph{global spectrum}.

A candidate for this has been recently proposed by Schwede \cite{schwede-k-theory}, who introduced \emph{global algebraic $K$-theory}. His approach differs from the other constructions discussed above in that it is not based on symmetric monoidal categories, but on so-called \emph{parsummable categories}. However, there is a specific way to assign a parsummable category to a small symmetric monoidal category, which can then be used to define its global algebraic $K$-theory.

\subsection*{New results}
The present article studies the global homotopy theory of symmetric monoidal and parsummable categories, in particular laying the groundwork for refinements of Thomason's and Mandell's results to equivariant and global algebraic $K$-theory.

More precisely, there is a notion of \emph{global weak equivalences} of parsummable categories, and global algebraic $K$-theory is invariant under them \cite[Theorem~4.16]{schwede-k-theory}. On the other hand, Schwede introduced a \emph{global model structure} on the category of small categories (modelling unstable global homotopy theory) in \cite{schwede-cat}, and we call a strong symmetric monoidal functor a global weak equivalence if its underlying functor is a weak equivalence in this model structure. As our first main result, we compare these two homotopy theories, thereby bringing Schwede's construction on an equal footing with the other approaches considered above:

\begin{introthm}[see Theorem~\ref{thm:perm-cat-vs-parsumcat}]\label{thm:perm-vs-par-sum-cat}
The passage from small symmetric monoidal categories to parsummable categories defines an equivalence of homotopy theories with respect to the global weak equivalences (i.e.~it induces an equivalence on the corresponding $\infty$-categorical localizations).
\end{introthm}

In particular, symmetric monoidal categories are just as good from the perspective of global algebraic $K$-theory as general parsummable categories. This also follows the general pattern that \emph{on the pointset level} global objects can often be modelled by ordinary non-equivariant objects, and that it is only through the notion of weak equivalence that their equivariant behavior emerges.

As our second contribution, we introduce \emph{parsummable simplicial sets} as a simplicial analogue of parsummable categories. There is again a suitable notion of global weak equivalences, and with respect to these we prove:

\begin{introthm}[see Theorem~\ref{thm:nerve-vs-c-bullet}]\label{thm:par-sum-cat-vs-sset}
The nerve defines an equivalence of homotopy theories between the categories of parsummable categories and of parsummable simplicial sets.
\end{introthm}

We are particularly interested in parsummable simplicial sets because they form a bridge to several concepts of `globally coherently commutative monoids' that we study in \cite{g-global}. In particular, we show in \emph{op.~cit.} that their homotopy theory is equivalent to a suitable global version of $\Gamma$-spaces and to Schwede's \emph{ultra-commutative monoids} \cite{schwede-book} (viewed through the eyes of finite groups). These further comparisons use techniques from homotopical algebra, and in particular they use that the global weak equivalences of parsummable simplicial sets are part of a model structure. It is not clear, whether such model structures also exist on the other categories discussed in this article.

While interesting in their own right, Theorems~\ref{thm:perm-vs-par-sum-cat} and~\ref{thm:par-sum-cat-vs-sset} above are crucial ingredients to the proof we give in \cite{g-global} that symmetric monoidal categories model all `globally coherently commutative monoids,' which can be viewed as a global refinement of Mandell's theorem and as an  `additive' version of Schwede's result that categories model all unstable global homotopy types \cite[Theorem~3.3]{schwede-cat}. Together with a global version of Segal's delooping theory, that we also develop in \cite{g-global}, this in particular allows us to refine Thomason's original result to a global comparison.

\subsection*{\texorpdfstring{$\bm G$}{G}-global homotopy theory}
In this article we will actually prove Theorems~\ref{thm:perm-vs-par-sum-cat} and~\ref{thm:par-sum-cat-vs-sset} in greater generality: namely, we allow an additional discrete, but possibly infinite group $G$ to act everywhere, and we consider the resulting categories of $G$-objects with respect to so-called \emph{$G$-global weak equivalences}, which for $G=1$ recovers the previous definitions.

However, for general $G$, the $G$-global weak equivalences are typically finer than the underlying global weak equivalences, and they are in particular fine enough to recover the usual $G$-equivariant information when $G$ is finite. More precisely, Shimakawa's equivariant $K$-theory is invariant under $G$-global weak equivalences,{\hskip0pt plus 2pt} so{\hskip0pt plus 2pt} that{\hskip0pt plus 2pt} the{\hskip0pt plus 2pt} $G$-global{\hskip0pt plus 2pt} versions{\hskip0pt plus 2pt} of{\hskip0pt plus 2pt} Theorems~\ref{thm:perm-vs-par-sum-cat}{\hskip0pt plus 2pt} and~\ref{thm:par-sum-cat-vs-sset}{\hskip0pt plus 2pt} provide structural information about the equivariant algebraic $K$-theory of symmetric monoidal categories with suitable $G$-actions. We study $G$-global homotopy theory systematically in \cite{g-global}, where we in particular use these generalizations to prove $G$-equivariant versions of Thomason's and Mandell's results.

\subsection*{Outline}
In Section~\ref{sec:reminder} we review the basic theory of tame $E\mathcal M$-categories and parsummable categories, and we define the $G$-global homotopy theory of $E\mathcal M$-$G$-categories and $G$-parsummable categories. Their simplicial counterparts are then introduced and studied in Section~\ref{sec:parsum-sset}.

Section~\ref{sec:construction} is devoted to the construction of the tame $E\mathcal M$-$G$-category $\textup{C}_X$ associated to a tame $E\mathcal M$-$G$-simplicial set $X$. We prove in Section~\ref{sec:unstable} that the resulting functor is homotopy inverse to the nerve, and that the $E\mathcal M$-$G$-categories arising this way satisfy a certain technical condition that we call \emph{weak saturatedness}. Section~\ref{sec:comp-parsum} is then devoted to lifting the results of the previous two sections to $G$-parsummable categories and $G$-parsummable simplicial sets, in particular proving Theorem~\ref{thm:par-sum-cat-vs-sset}.

Finally, we prove Theorem~\ref{thm:perm-vs-par-sum-cat} in Section~\ref{sec:sym-mon} by using the weak saturatedness established in Section~\ref{sec:unstable} to reduce it to the categorical comparison result of \cite{perm-parsum-categorical}.

\subsection*{Acknowledgements} This article was written as part of my PhD thesis at the University of Bonn, and I would like to thank my advisor Stefan Schwede for suggesting equivariant and global algebraic $K$-theory as a thesis topic, as well as for helpful remarks on a previous version of this article. I am moreover grateful to the anonymous referee for various comments that helped improve the exposition of the present article. Finally, I am indebted to Markus Hausmann, who first suggested to me that there should be a notion of $G$-global homotopy theory. This proved to be a fruitful way of thinking about many phenomena surrounding global $K$-theory, which in particular permeates \cite{g-global}.

I am grateful to the Max Planck Institute for Mathematics in Bonn for their hospitality and support during my PhD studies. At the time the first version of this article was written, I was an associate member of the Hausdorff Center for Mathematics, funded by the Deutsche Forschungsgemeinschaft (DFG, German Research Foundation) under Germany's Excellence Strategy (GZ 2047/1, project ID 390685813).

\section{A reminder on parsummable categories}\label{sec:reminder}
\subsection{\texorpdfstring{$\bmM$}{M}-sets and tameness} We begin by recalling the monoid $\mathcal M$ as well as some basic results about the combinatorics of $\mathcal M$-actions from \cite{I-vs-M-1-cat}.

\begin{defi}
If $A, B$ are sets, then we write $\Inj(A,B)$ for the set of injective maps $A\to B$. We write $\omega$ for the countably infinite set $\{1,2,\dots\}$, and we write $\mathcal M$ for $\Inj(\omega,\omega)$ considered as a monoid under composition.
\end{defi}

\begin{warn}
In \cite{I-vs-M-1-cat, schwede-k-theory}, the symbol $M$ is used for the above monoid, while \cite{schwede-k-theory} uses $\mathcal M$ for what we call $E\mathcal M$ below. The reason for this change in notation is consistency with \cite{g-global}, where actions of the above monoid are studied in their own right. In particular we prove in \cite{g-global} that simplicial sets with an action of $\Inj(\omega,\omega)$ (when equipped with a slightly subtle notion of weak equivalence) already model global homotopy theory, so we think it deserves to be notationally distinguished from a generic monoid.
\end{warn}

Next, we come to the notions of support and tameness for $\mathcal M$-sets, whose categorical and simplicial counterparts will later be central for defining parsummable categories and parsummable simplicial sets, respectively.

\begin{defi}
Let $X$ be an $\mathcal M$-set, let $x\in X$, and let $A\subset\omega$ be finite. We say that \emph{$x$ is supported on $A$} if $u.x=x$ for all $u\in\mathcal M$ fixing $A$ pointwise. We say that $x$ is \emph{finitely supported} if it is supported on some finite set. The $\mathcal M$-set $X$ is called \emph{tame} if all its elements are finitely supported.
\end{defi}

\begin{defi}
Let $X$ be an $\mathcal M$-set and let $x\in X$ be finitely supported. Then the \emph{support} $\supp(x)$ of $x$ is the intersection of all finite sets on which it is supported.
\end{defi}

\begin{lemma}\label{lemma:M-set-supported-on-supp}
In the above situation, $x$ is supported on $\supp(x)$.
\end{lemma}

Put differently, a finitely supported $x$ is supported on a unique minimal set.

\begin{proof}
This is immediate from \cite[Proposition~2.3]{I-vs-M-1-cat}, also see the discussion after Proposition~2.4 of \emph{op.~cit.}
\end{proof}

\begin{ex}\label{ex:Inj}
If $A$ is any finite set, then we can consider $\Inj(A,\omega)$ with $\mathcal M$-action given by $\mathcal M\times\Inj(A,\omega), (u,i)\mapsto u\circ i$. This is tame: an injection $i\colon A\to\omega$ is obviously supported on $i(A)$. In fact, $\supp(i)=i(A)$: namely, if $B\not\supset i(A)$, then we can pick an $a\in i(A)\smallsetminus B$ and an injection $u$ fixing $B$ pointwise with $a\notin\im u$. Then $u.i=u\circ i$ does not hit $a$, so $u.i\not=i$, i.e.~$i$ is not supported on $B$.
\end{ex}

\begin{nex}\label{nex:inj-infinite}
If $A$ is countably infinite, then $\Inj(A,\omega)$ is \emph{not} tame, in particular $\mathcal M$ with its left regular action is not tame. In fact, no element is finitely supported: if $i\colon A\to\omega$ is any injection, and $B\subset\omega$ is any finite set, then $B\not\supset i(A)$ as the right hand side is infinite. The same argument as in the previous example then shows that $i$ is not supported on $B$.
\end{nex}

We close this discussion by collecting some basic facts about the support for easy reference.

\begin{lemma}\label{lemma:supp-change}
\begin{enumerate}
\item If $f\colon X\to Y$ is an $\mathcal M$-equivariant map of $\mathcal M$-sets and $x\in X$ is finitely supported, then also $f(x)$ is finitely supported. Moreover, $\supp(f(x))\subset\supp(x)$.\label{item:sc-m-equiv}
\item If $X$ is an $\mathcal M$-set and $x\in X$ is finitely supported, then $u.x$ is finitely supported for all $u\in\mathcal M$. Moreover, $\supp(u.x)=u(\supp(x))$.
\end{enumerate}
\end{lemma}
\begin{proof}
The first statement is immediate from the definition and it also appears without proof in \cite[discussion after Proposition~2.5]{I-vs-M-1-cat}. The second statement is \cite[Proposition~2.5-(ii)]{I-vs-M-1-cat}.
\end{proof}

\begin{lemma}\label{lemma:supp-m-set-agree}
Let $X$ be an $\mathcal M$-set, let $x\in X$ be supported on the finite set $A\subset\omega$, and let $u,v\in\mathcal M$ with $u(a)=v(a)$ for all $a\in A$. Then $u.x=v.x$.
\end{lemma}
\begin{proof}
This is \cite[Proposition~2.5-(i)]{I-vs-M-1-cat}.
\end{proof}

\subsection{\texorpdfstring{$\bmEM$}{EM}-categories and parsummable categories} We recall the so-called `chaotic categories':

\begin{constr}
Let $X$ be a set. We write $EX$ for the (small) category with set of objects $\Ob(EX)=X$ and precisely one morphism $x\to y$ for any $x,y\in X$, which we denote by $(y,x)$. The composition is then uniquely determined by this; explicitly, $(z,y)(y,x)=(z,x)$ for all $x,y,z\in X$.

For any map of sets $f\colon X\to Y$ there is a unique functor $Ef\colon EX\to EY$ that is given on objects by $f$, and this way $E$ becomes a functor $\cat{Set}\to\cat{Cat}$. It is easy to check that $E$ is right adjoint to $\Ob$ with counit $\Ob EX\to X$ the identity.
\end{constr}

In particular, $E$ preserves products and terminal objects, so it sends ordinary monoids to monoids in the $1$-category $\cat{Cat}$ (i.e.~\emph{strict monoidal categories}). Central to Schwede's construction of global algebraic $K$-theory is the categorical monoid $E\mathcal M$ obtained this way.

\begin{defi}
An \emph{$E\mathcal M$-category} is a category $\mathcal C$ together with a strict action of $E\mathcal M$. A \emph{map of $E\mathcal M$-categories} is a functor $f\colon\mathcal C\to\mathcal D$ strictly commuting with the action, i.e.~such that the diagram
\begin{equation*}
\begin{tikzcd}
E\mathcal M\times\mathcal C\arrow[r, "\text{act}"]\arrow[d, "E\mathcal M\times f"'] &\mathcal C\arrow[d, "f"]\\
E\mathcal M\times\mathcal D\arrow[r, "\text{act}"'] & \mathcal D
\end{tikzcd}
\end{equation*}
commutes. We write $\cat{$\bmEM$-Cat}$ for the category of small $E\mathcal M$-categories.
\end{defi}

We will denote $E\mathcal M$-categories by the calligraphic letters $\mathcal C,\mathcal D$, and so on (distinguishing them from ordinary categories denoted by $\mathscr C$, $\mathscr D$, etc.).

If $\mathcal C$ is an $E\mathcal M$-category, then we have in particular an action of the discrete monoid $\mathcal M$ on $\mathcal C$, which then restricts to an $\mathcal M$-action on the (large) set $\Ob\mathcal C$. In addition, we are given for each $u,v\in\mathcal M$ a natural isomorphism $[u,v]\colon (v.\blank)\Rightarrow(u.\blank)$ given on $x\in\mathcal C$ by $(u,v).\id_x\colon v.x\to u.x$. Specializing to $v=1$ this yields in particular for each $x\in\mathcal C$ an isomorphism $u^x_\circ\colon x\to u.x$. From functoriality and associativity of the action one easily concludes that
\begin{equation}\label{eq:associativity-u-circ}
(uv)^x_\circ=u^{v.x}_\circ v^x_\circ
\end{equation}
for all $u,v\in\mathcal M$ and $x\in\mathcal C$. The following useful lemma shows that $E\mathcal M$-actions on categories can conversely be described by the above data, considerably simplifying their construction:

\begin{lemma}
Let $\mathcal C$ be a category, and assume we are given an $\mathcal M$-action on $\Ob\mathcal C$ together with for each $u\in\mathcal M$, $x\in\mathcal C$ an isomorphism $u^x_\circ\colon x\to u.x$ such that these data satisfy the relation $(\ref{eq:associativity-u-circ})$.

Then there exists a unique $E\mathcal M$-action on $\mathcal C$ extending the $\mathcal M$-action on $\Ob\mathcal C$ and such that $u^x_\circ=[u,1]_x$ for all $x\in\mathcal C$ and $u\in\mathcal M$.
\end{lemma}
\begin{proof}
See \cite[Proposition~2.6]{schwede-k-theory}.
\end{proof}

There is also a relative version of the lemma, that allows us to check $E\mathcal M$-equivariance of functors in terms of the above data:

\begin{lemma}
Let $\mathcal C,\mathcal D$ be $E\mathcal M$-categories and let $f\colon \mathcal C\to\mathcal D$ be a functor of their underlying categories. Then $f$ is $E\mathcal M$-equivariant if and only if $\Ob f\colon\Ob\mathcal C\to\Ob\mathcal  D$ is $\mathcal M$-equivariant and $f(u^x_\circ)=u^{f(x)}_\circ\colon f(x)\to u.(f(x))=f(u.x)$ for all $x\in \mathcal C$, $u\in\mathcal M$.
\end{lemma}
\begin{proof}
This is \cite[Corollary~1.3]{perm-parsum-categorical}.
\end{proof}

We can now define the categorical counterpart of the support:

\begin{defi}
Let $\mathcal C\in\cat{$\bmEM$-Cat}$, let $x\in\mathcal C$, and let $A\subset\omega$ be a finite set. Then we say that $x$ is \emph{supported on $A$} if $x$ is supported on $A$ as an element of the $\mathcal M$-set $\Ob(\mathcal C)$. We write $\mathcal C_{[A]}\subset\mathcal C$ for the full subcategory spanned by the objects supported on $A$.

We morever say that $x$ is \emph{finitely supported} if it is supported on some finite set, and we write $\mathcal C^\tau$ for the full subcategory of those.

Finally, we call $\mathcal C$ \emph{tame} if all its objects are finitely supported (i.e.~if $\Ob(\mathcal C)$ is tame), and we denote the full subcategory of $\cat{$\bmEM$-Cat}$ spanned by the tame $E\mathcal M$-categories by $\cat{$\bmEM$-Cat}^\tau$.
\end{defi}

Lemma~\ref{lemma:supp-change}-$(\ref{item:sc-m-equiv})$ shows that any $E\mathcal M$-equivariant functor $f\colon\mathcal C\to\mathcal D$ restricts to $\mathcal C_{[A]}\to\mathcal D_{[A]}$ for all finite $A\subset\omega$, and hence in particular to $\mathcal C^\tau\to\mathcal D^\tau$.

\begin{defi}
Let $\mathcal C\in\cat{$\bmEM$-Cat}$, and let $x\in\mathcal C$ be finitely supported. Then the \emph{support} $\supp(x)$ of $x$ is the intersection of all (finite) sets $A\subset\omega$ on which $x$ is supported, i.e.~its support as an element of the $\mathcal M$-set $\Ob(\mathcal C)$.
\end{defi}

Lemma~\ref{lemma:M-set-supported-on-supp} shows that $x$ is indeed supported on $\supp(x)$, also see~\cite[Proposition~2.13-(i)]{schwede-k-theory}.

\begin{ex}\label{ex:EInj-supp}
If $A$ is any finite set, then the $\mathcal M$-action on $\Inj(A,\omega)$ from Example~\ref{ex:Inj} induces an $E\mathcal M$-action on $E\Inj(A,\omega)$ recovering the original action on objects. In particular, this $E\mathcal M$-category is tame and the support of an object $i\colon A\to\omega$ is precisely $i(A)$.
\end{ex}

At first sight, the definition of support might seem a bit to weak (or `un-categorical') as for $u\in\mathcal M$ fixing $\supp(x)$ pointwise we only explicitly require that $x=u.x$ and not that the canonical comparison isomorphism $u^x_\circ\colon x\to u.x$ be the identity. However, it turns out that this seemingly stronger condition is in fact automatically satisfied:

\begin{lemma}\label{lemma:support-strong}
Let $\mathcal C\in\cat{$\bmEM$-Cat}$, let $x\in\mathcal C$ be finitely supported, and let $u,v\in\mathcal M$ agree on $\supp(x)$. Then $u.x=v.x$. If moreover $u',v'\in\mathcal M$ agree on $\supp(x)$, then $[u',u]_x=[v',v]_x\colon u.x=v.x\to u'.x=v'.x$.

In particular, $u^x_\circ=\id_x$ if $u$ restricts to the identity on $\supp(x)$.
\end{lemma}
\begin{proof}
This follows from \cite[Proposition~2.13-(ii)]{schwede-k-theory}.
\end{proof}

\begin{warn}
The above lemma should not be misunderstood as saying that $u.x=x$ for $u\in\mathcal M$, $x\in\mathcal C$ if and only if $u^x_\circ$ is the identity. In particular, if $H\subset\mathcal M$ is a subgroup (i.e.~a submonoid that at the same time is a group), then for $x\in\mathcal C^H$ the structure isomorphisms $h_\circ\colon x\to h.x$ are usually non-trivial, and instead they define a potentially interesting $H$-action on $x$. In fact, for so-called \emph{saturated} $E\mathcal M$-categories $\mathcal C$, that we will recall below, and `nicely' embedded subgroups $H$, all $H$-objects in $\mathcal C$ arise this way, cf.~\cite[Construction~7.4]{schwede-k-theory}.
\end{warn}

We are now ready to introduce the box product of tame $E\mathcal M$-categories \cite[Definition~2.32]{schwede-k-theory}:

\begin{defi}
Let $\mathcal C,\mathcal D\in\cat{$\bmEM$-Cat}^\tau$. Their \emph{box product} $\mathcal C\boxtimes\mathcal D$ is the full subcategory of $\mathcal C\times\mathcal D$ spanned by the pairs $(c,d)$ of \emph{disjointly supported} objects, i.e.~pairs such that $\supp(c)\cap\supp(d)=\varnothing$.
\end{defi}

By \cite[Proposition~2.13-(iii)]{schwede-k-theory} the diagonal $E\mathcal M$-action on $X\times Y$ restricts to an $E\mathcal M$-action on $X\boxtimes Y$, which is again tame by \cite[Corollary~2.34]{schwede-k-theory}. The box product then becomes a subfunctor of the cartesian product, and it is not hard to check that the usual unitality, associativity, and symmetry isomorphisms of the cartesian product restrict to corresponding isomorphisms for $\boxtimes$, making it the tensor product of a preferred symmetric monoidal structure on $\cat{$\bmEM$-Cat}^\tau$ with unit the terminal $E\mathcal M$-category \cite[Proposition~2.35]{schwede-k-theory}.

\begin{defi}
A \emph{parsummable category} is a commutative monoid for the box product on $\cat{$\bmEM$-Cat}^\tau$. We write $\cat{ParSumCat}$ for the corresponding category of commutative monoids (whose morphisms are the monoid homomorphisms).
\end{defi}

Explicitly, this means that the data of a parsummable category consist of a tame $E\mathcal M$-category $\mathcal C$ equipped with an object $0$ of empty support as well as a suitably functorial `addition' defined for any pair of disjointly supported objects $x,y\in\mathcal C$ and for any pair of morphisms $f\colon x\to x'$, $g\colon y\to y'$ such that $x,y$ and $x',y'$ are disjointly supported. The addition has to be strictly associative, unital, and commutative whenever this makes sense. In general, two given objects $x,y\in\mathcal C$ might not be summable (i.e.~$\supp(x)\cap\supp(y)\not=\varnothing$), but \cite[proof of Theorem~2.33]{schwede-k-theory} shows that we can always replace $x,y$ by an isomorphic pair of summable objects.

\begin{ex}
Any abelian monoid $A$ becomes a parsummable category when viewed as a discrete category with trivial $E\mathcal M$-action.
\end{ex}

\begin{ex}
Associated to any tame $E\mathcal M$-category $\mathcal C$ we have a \emph{free parsummable category} $\textbf{P}\mathcal C\mathrel{:=}\coprod_{n\ge0}\mathcal C^{\boxtimes n}/\Sigma_n$ (where $\Sigma_n$ denotes the symmetric group); here the additive unit is the unique object $0\in\mathcal C^{\boxtimes 0}/\Sigma_0$ and the sum is given by concatenation. This defines a left adjoint to the forgetful functor $\cat{ParSumCat}\to\cat{$\bmEM$-Cat}^\tau$, see~\cite[Example~4.6]{schwede-k-theory}.
\end{ex}

\begin{ex}\label{ex:Phi}
Let $\mathscr C$ be a permutative category, i.e.~a symmetric monoidal category in which the associativity and unitality isomorphisms are the respective identities. Then \cite[Construction~11.1]{schwede-k-theory} associates a parsummable category $\Phi(\mathscr C)$ to $\mathscr C$ as follows: an object of $\Phi(\mathscr C)$ is an infinite family $(X_1,X_2,\dots)$ of objects in $\mathscr C$ such that $X_i=\bm1$ (the unit of the tensor product) for almost all $i$. If $(Y_1,Y_2,\dots)$ is another such object, then a morphism $(X_1,\dots)\to (Y_1,\dots)$ in $\Phi(\mathscr C)$ is a morphism $\bigotimes_{i=1}^\infty X_i\to\bigotimes_{i=1}^\infty Y_i$ in $\mathscr C$; note that these infinite tensor products indeed make sense as $\otimes$ is assumed to be strictly associative and unital. The $\mathcal M$-action on $\Ob(\Phi(\mathscr C))$ is given by `permuting and extending by $\bm 1$,' i.e.~if $u\in\mathcal M$, then $u.(X_1,\dots)=(Y_1,\dots)$ with
\begin{equation*}
    Y_i=\begin{cases}
        \,X_j & \text{if $u(j)=i$}\\
        \,\bm1 & \text{if $i\notin\im(u)$};
    \end{cases}
\end{equation*}
the comparison map $u_\circ : (X_1,\dots)\to (Y_1,\dots)$ is obtained as a suitable composition of symmetry isomorphisms, see \emph{loc.~cit.} for details.

The sum on objects is given by `interlacing': if $(X_1,\dots)$ and $(Y_1,\dots)$ are disjointly supported, then we have $X_i=\bm1$ or $Y_i=\bm1$ (or both) for every $i$, and we define $(X_1,\dots)+(Y_1,\dots)=(Z_1,\dots)$ with
\begin{equation*}
    Z_i=\begin{cases}
        \,X_i & \text{if $Y_i=\bm1$}\\
        \,Y_i & \text{otherwise}.
    \end{cases}
\end{equation*}
The sum on morphisms is given by tensoring and conjugating with suitable composites of symmetry isomorphisms; we again refer to the aforementioned reference for details.
\end{ex}

\subsection{\texorpdfstring{$\bm{G}$}{G}-global homotopy theory of \texorpdfstring{$\bm{G}$}{G}-parsummable categories}
For the rest of this article, let us fix a discrete (not necessarily finite) group $G$. We will be interested in the category $\cat{$\bmEM$-$\bm G$-Cat}$ of $G$-objects in $\cat{$\bmEM$-Cat}$, whose objects we call small $E\mathcal M$-$G$-categories. They can also be alternatively described as small categories with an $E\mathcal M$-action and a $G$-action such that the two actions commute, or as small categories with an action of the categorical monoid $E\mathcal M\times G$.

Any $\mathcal C\in\cat{$\bmEM$-$\bm G$-Cat}$ has in particular an underlying $E\mathcal M$-action, from which it inherits the notions of support and tameness. The full subcategory
\begin{equation*}
\cat{$\bmEM$-$\bm G$-Cat}^\tau\subset\cat{$\bmEM$-$\bm G$-Cat}
\end{equation*}
of the tame $E\mathcal M$-$G$-categories is then canonically identified with $G$-objects in $\cat{$\bmEM$-Cat}^\tau$.

In particular, the box product of tame $E\mathcal M$-categories together with its unitality, associativity, and symmetry isomorphisms automatically lifts to a symmetric monoidal structure on $\cat{$\bmEM$-$\bm G$-Cat}^\tau$. We write $\cat{$\bm G$-ParSumCat}$ for the corresponding category of commutative monoids and call its objects \emph{$G$-parsummable categories}. The category $\cat{$\bm G$-ParSumCat}$ is then again canonically identified with the category of $G$-objects in $\cat{ParSumCat}$.

We now want to study $E\mathcal M$-$G$-categories and $G$-parsummable categories from an equivariant point of view. The crucial observation for this is that any group embeds into $\mathcal M$ in a particularly nice way; in order to describe these embeddings we need the following two notions:

\begin{defi}
Let $H$ be a finite group. A countable $H$-set $X$ is called a \emph{complete $H$-set universe} if every finite $H$-set embeds into $X$ equivariantly.
\end{defi}

\begin{defi}
A finite subgroup $H\subset\mathcal M$ is called \emph{universal} if $\omega$ equipped with the restriction of the tautological $\mathcal M$-action to $H$ is a complete $H$-set universe.
\end{defi}

\begin{rk}
The above notion of universality is analogous to \cite[Definitions~1.3 and~1.4]{schwede-L} which studies global homotopy theory with respect to compact Lie groups.
\end{rk}

\begin{lemma}\label{lemma:universal-abundant}
Any finite group $H$ admits an injective homomorphism $i\colon H\to\mathcal M$ with universal image. If $j\colon H\to\mathcal M$ is another such homomorphism, then there exists a {\rm(}non-canonical\/{\rm)} invertible element $\varphi\in\mathcal M$ with $j(h)=\varphi i(h) \varphi^{-1}$ for all $h\in H$.
\end{lemma}

This lemma appeared in a preliminary version of \cite{schwede-k-theory}. It tells us in particular that we can associate to an $E\mathcal M$-category $\mathcal C$ for any \emph{abstract} finite group $H$ an underlying $H$-category $\mathcal C_H$ well-defined up to (a priori non-canonical) isomorphism by picking an injective homomorphism $i\colon H\to\mathcal M$ with universal image and setting $\mathcal C_H=i^*\mathcal C$, which explains the global equivariant behaviour of $E\mathcal M$-categories.

\begin{proof}[\llap`Proof\kern1pt\rlap'\kern-.5pt]
This is completely analogous to \cite[proof of Proposition~1.5]{schwede-L}, and we leave the details to the interested reader.
\end{proof}

In order to define the `$G$-global homotopy theory' of small $E\mathcal M$-$G$-categories, we introduce some notation: if $H\subset\mathcal M$ is a subgroup and $\varphi\colon H\to G$ is a homomorphism, then we write $\Gamma_{\varphi}\mathrel{:=}\{(h,\varphi(h)) : h\in H\}$ for the corresponding graph subgroup of $E\mathcal M\times G$. If $\mathcal C\in\cat{$\bmEM$-$\bm G$-Cat}$, then we abbreviate $\mathcal C^\varphi\mathrel{:=}\mathcal C^{\Gamma_\varphi}$ for the category of $\Gamma_\varphi$-fixed points, and if $f\colon\mathcal C\to\mathcal D$ is a morphism in $\cat{$\bmEM$-$\bm G$-Cat}$, then we write $f^{\varphi}\mathrel{:=}f^{\Gamma_\varphi}\colon\mathcal C^\varphi\to\mathcal D^\varphi$. This yields a functor $(\blank)^\varphi\colon\cat{$\bmEM$-$\bm G$-Cat}\to\cat{Cat}$.

\begin{defi}
A morphism $f\colon\mathcal C\to\mathcal D$ in $\cat{$\bmEM$-$\bm G$-Cat}$ is called a \emph{$G$-global weak equivalence} if $f^\varphi\colon\mathcal C^\varphi\to\mathcal D^\varphi$ is a weak homotopy equivalence (i.e.~a weak homotopy equivalence on nerves) for each universal subgroup $H\subset\mathcal M$ and each homomorphism $\varphi\colon H\to G$.

A morphism of $G$-parsummable categories is called a $G$-global weak equivalence if it so as a morphism in $\cat{$\bmEM$-$\bm G$-Cat}$.
\end{defi}

For $G=1$, Schwede \cite[Definition~2.26]{schwede-k-theory} considered these under the name `global equivalence.' Here we use the term `$G$-global \emph{weak} equivalence' instead in order to emphasize that these are a refinement of the weak homotopy equivalences instead of the \emph{categorical equivalences}, i.e.~those $(E\mathcal M\times G)$-equivariant functors that are equivalences of underlying categories.

In fact, the $G$-global weak equivalences and the categorical equivalences are in general incomparable, i.e.~neither notion implies the other. However, there is at least a particular class of interesting $E\mathcal M$-$G$-categories for which the $G$-global weak equivalences are indeed weaker than the categorical equivalences, which we will now introduce.

\begin{constr}
Let $\mathcal C$ be a (small) $E\mathcal M$-$G$-category, let $H\subset\mathcal M$ be universal, and let $\varphi\colon H\to G$ be any group homomorphism. Then $H$ acts on $EH$ via its left regular action, and it acts on $\mathcal C$ via the diagonal of its actions via $\mathcal M$ and $\varphi$. We write $\mathcal C^{\myh\varphi}\mathrel{:=}\Fun(EH,\mathcal C)^\varphi$ for the fixed points of the induced conjugation action. In other words, $\Fun(EH,\mathcal C)^\varphi$ is the subcategory of those $\Phi\colon EH\to \mathcal C$ such that $\Phi\circ (h.\blank)=\big((h,\varphi(h)).\blank\big)\circ\Phi$ together with those natural transformations $\tau\colon\Phi\Rightarrow\Psi$ such that $\tau_h=(h,\varphi(h)).\tau_1$.

Restricting along $EH\to {*}$ produces a fully faithful functor $\mathcal C\to\Fun(EH,\mathcal C)$ that is equivariant with respect to the above action on the target and $H$ acting on $\mathcal C$ as before. In particular, we get an induced functor $\mathcal C^\varphi\to\Fun(EH,\mathcal C)^\varphi=\mathcal C^{\myh\varphi}$ that is again fully faithful as a limit of fully faithful functors.
\end{constr}

Here we use the notation $\mathcal C^{\myh\varphi}$ as we want to stress that the above are homotopy fixed points (commonly denoted $\mathcal C^{h\varphi}$), but \emph{with respect to the categorical equivalences} and not with respect to the weak homotopy equivalences or $G$-global weak equivalences. In particular, if $f$ is any categorical equivalence, then $f^{\myh\varphi}$ will be an equivalence of categories by \cite[Proposition~2.16]{merling}, but if $f$ is a $G$-global weak equivalence, then $f^{\myh\varphi}$ need not be a weak homotopy equivalence, as the following example for $G=1$ shows:

\begin{ex}
Let $\mathcal C$ be the groupoid $B(\mathbb Z/2)$ and let $f\colon\mathcal D\to\mathcal C$ be a weak homotopy equivalence from a category $\mathcal D$ without non-trivial isomorphisms; this exists for example by taking a cofibrant replacement in the Thomason model structure on $\cat{Cat}$ \cite[Proposition~5.7]{thomason-unstable}, or via Proposition~\ref{prop:last-vertex} below.

If we equip $\mathcal C$ and $\mathcal D$ with the trivial $E\mathcal M$-action, then $f$ becomes a global weak equivalence in $\cat{$\bmEM$-Cat}^\tau$. However, if $H\subset\mathcal M$ is a universal subgroup isomorphic to $\mathbb Z/2$, then $f^{\myh H}$ is a map from $\mathcal D^{\myh H}\simeq\mathcal D$ (here we used that $\mathcal D$ has no non-trivial isomorphisms, so that any map from $EH$ is constant) to $\mathcal C^{\myh H}\simeq\Fun(B(\mathbb Z/2),B(\mathbb Z/2))$, and this can't be a weak homotopy equivalence as the nerve of the latter has two path components (corresponding to the identity of $\mathbb Z/2$ and the zero homomorphism).
\end{ex}

\begin{defi}
A small $E\mathcal M$-$G$-category $\mathcal C$ is called \emph{saturated} if for all universal $H\subset\mathcal M$ and all $\varphi\colon H\to G$ the above functor $\mathcal C^\varphi\hookrightarrow\mathcal C^{\myh\varphi}$ is an equivalence of categories. It is called \emph{weakly saturated} if $\mathcal C^\varphi\hookrightarrow\mathcal C^{\myh\varphi}$ is a weak homotopy equivalence.

We write $\cat{$\bmEM$-$\bm G$-Cat}^{\tau,s}\subset\cat{$\bmEM$-$\bm G$-Cat}^{\tau,ws}\subset\cat{$\bmEM$-$\bm G$-Cat}^{\tau}$ for the full subcategories of saturated and weakly saturated tame $E\mathcal M$-$G$-categories, respectively.
\end{defi}

If $G=1$, the above definition of saturatedness agrees with \cite[Definition~7.3]{schwede-k-theory}. However, for the present article the weak notion will be more important.

\begin{lemma}\label{lemma:cat-between-ws}
Let $f\colon\mathcal C\to\mathcal D$ be a categorical equivalence in $\cat{$\bmEM$-$\bm G$-Cat}^{\tau,ws}$. Then $f$ is a $G$-global weak equivalence.
\end{lemma}
\begin{proof}
Let $H\subset\mathcal M$ be a universal subgroup, and let $\varphi\colon H\to G$ be any group homomorphism. Then we have a commutative diagram
\begin{equation*}
\begin{tikzcd}
\mathcal C^\varphi\arrow[d, hook]\arrow[r, "f^\varphi"] &\mathcal D^\varphi\arrow[d,hook]\\
\mathcal C^{\myh\varphi}\arrow[r, "f^{\myh\varphi}"'] &\mathcal D^{\myh\varphi}
\end{tikzcd}
\end{equation*}
with the vertical maps as above. The bottom map is an equivalence of categories because $f$ is, hence in particular a weak homotopy equivalence. Moreover, the vertical maps are weak homotopy equivalences by assumption, so that the top arrow is also a weak homotopy equivalence by $2$-out-of-$3$ as desired.
\end{proof}

\begin{rk}
The same argument shows that if $\mathcal C$ and $\mathcal D$ are actually saturated, then any categorical equivalence $f\colon\mathcal C\to\mathcal D$ even induces \emph{equivalences of categories} on all the relevant fixed points.
\end{rk}

We also recall our \emph{saturation construction} that first appeared for $G=1$ as \cite[Construction~7.20]{schwede-k-theory}:

\begin{constr}
Let $\mathcal C$ be a small $E\mathcal M$-category. Then $\Fun(E\mathcal M,\mathcal C)$ carries two commuting $E\mathcal M$-actions: one via the given action on $\mathcal C$ and one via the right $E\mathcal M$-action on itself via precomposition. We equip $\Fun(E\mathcal M,\mathcal C)$ with the diagonal of these two actions.

Now assume that $\mathcal C$ is tame. We write $\mathcal C^{\sat}\mathrel{:=}\Fun(E\mathcal M,\mathcal C)^\tau$. If $f\colon\mathcal C\to\mathcal D$ is an $E\mathcal M$-equivariant functor, then we write $f^{\sat}\mathrel{:=}\Fun(E\mathcal M,f)^\tau$; we omit the easy verification that $f^{\sat}$ is well-defined and $E\mathcal M$-equivariant, and that this way $(\blank)^{\sat}$ becomes an endofunctor of $\cat{$\bmEM$-Cat}^\tau$.

Finally, we consider the $E\mathcal M$-equivariant functor $\mathcal C\to\Fun(E\mathcal M,\mathcal C)$ induced by $E\mathcal M\to{*}$, which restricts to an $E\mathcal M$-equivariant functor $s\colon\mathcal C\to\mathcal C^{\sat}$, see~\cite[Theorem~7.24-(iv)]{schwede-k-theory}. We omit the easy verification that $s$ is natural.

Pulling through the $G$-actions everywhere, we can upgrade $(\blank)^{\sat}$ to an endo\-functor of $\cat{$\bmEM$-$\bm G$-Cat}^\tau$, and $s$ automatically defines a natural transformation from the identity to this lift.
\end{constr}

\begin{thm}\label{thm:sat-cat}
Let $\mathcal C$ be a tame $E\mathcal M$-$G$-category.
\begin{enumerate}
\item\label{item:sc-sat} $\mathcal C^{\sat}$ is saturated, so that $(\blank)^{\sat}$ restricts to a functor \begin{equation*}\cat{$\bmEM$-$\bm G$-Cat}^\tau\to\cat{$\bmEM$-$\bm G$-Cat}^{\tau,\text{s}}.\end{equation*}
\item $s\colon\mathcal C\to\mathcal C^{\sat}$ is a categorical equivalence.\label{item:sc-cat-eq}
\end{enumerate}
In particular, the inclusion $\cat{$\bmEM$-$\bm G$-Cat}^{\tau,\text{s}}\hookrightarrow\cat{$\bmEM$-$\bm G$-Cat}^\tau$ is a homotopy equivalence \emph{with respect to the categorical equivalences on both sides}, and $(\blank)^\sat$ is homotopy inverse to it.
\end{thm}

Here we call a homotopical functor $F\colon\mathscr C\to\mathscr D$ between categories with weak equivalences a \emph{homotopy equivalence} if there exists a homotopical functor $G\colon\mathscr D\to\mathscr C$ that is \emph{homotopy inverse} to it, i.e.~such that both $FG$ and $GF$ are connected by zig-zags of natural levelwise weak equivalences to the respective identities.

\begin{proof}
This is similar to the usual global situation, where this argument appeared in slightly different form as \cite[Theorem~7.24]{schwede-k-theory}.

We will show that $\Fun(E\mathcal M,\mathcal C)$ is saturated, and that the inclusion $\mathcal C^{\sat}=\Fun(E\mathcal M,\mathcal C)^\tau\hookrightarrow\Fun(E\mathcal M,\mathcal C)$ induces an \emph{equivalence of categories} on $\varphi$-fixed points for all universal $H\subset\mathcal M$ and all $\varphi\colon H\to G$. Applying the latter to $H=1$ in particular shows that $\mathcal C^\sat\hookrightarrow\Fun(E\mathcal M,\mathcal C)$ is a categorical equivalence, and since so is $\mathcal C\to\Fun(E\mathcal M,\mathcal C)$ (as $E\mathcal M\simeq{*}$), also $s\colon\mathcal C\to\mathcal C^{\sat}$ is a categorical equivalence by $2$-out-of-$3$, proving $(\ref{item:sc-cat-eq})$. On the other hand, it shows that for any $\varphi$ as above the top horizontal and right hand vertical map in the evident commutative diagram
\begin{equation*}
\begin{tikzcd}
(\mathcal C^\sat)^\varphi\arrow[r]\arrow[d,hook] & \Fun(E\mathcal M,\mathcal C)^\varphi\arrow[d,hook]\\
(\mathcal C^\sat)^{\myh\varphi}\arrow[r] & \Fun(E\mathcal M,\mathcal C)^{\myh\varphi}
\end{tikzcd}
\end{equation*}
are equivalences of categories. Moreover, we can also deduce that the lower horizontal map is an equivalence of categories (as a homotopy limit of equivalences), hence so is the left hand vertical map by $2$-out-of-$3$, which will precisely prove $(\ref{item:sc-sat})$.

It remains to prove the two claims. For the first one, we have to show that for all $H,\varphi$ as above the canonical map $\Fun(E\mathcal M,\mathcal C)^\varphi\to\Fun(EH,\Fun(E\mathcal M,\mathcal C))^\varphi$ is an equivalence. Under the identification $\Fun(EH,\Fun(E\mathcal M,\mathcal C))\cong\Fun(EH\times E\mathcal M,\mathcal C)\cong\Fun(E(H\times\mathcal M),\mathcal C)$ given by the adjunction isomorphism and the fact that $E$ preserves products, the right hand side of the above map corresponds to the fixed points with respect to the same $H$-action on $\mathcal C$ as before and the $H$-action on $E(H\times \mathcal M)$ induced by the left regular $H$-action and the $H$-action on $\mathcal M$ via $h.u=uh^{-1}$. Under this identification, the canonical map is induced by $E(\pr)\colon E(H\times\mathcal M)\to E\mathcal M$, and it will be enough to show that this is an $H$-equivariant equivalence of categories, i.e.~an equivalence in the $2$-category of $H$-categories, $H$-equivariant functors, and $H$-equivariant natural transformations.

For this we observe that both $H\times\mathcal M$ and $\mathcal M$ are free $H$-sets. Thus, there exists an $H$-equivariant map $r\colon\mathcal M\to H\times\mathcal M$. It is then easy to check that for varying $u\in\mathcal M$ the unique maps $\pr(r(u))\to u$ in $E\mathcal M$ assemble into an $H$-equivariant isomorphism $E(\pr)E(r)=E(\pr\circ r)\cong\id_{E\mathcal M}$, and similarly $E(r)E(\pr)\cong\id_{E(H\times\mathcal M)}$ equivariantly. This completes the proof of the first claim.

For the second claim we observe that $(\mathcal C^{\sat})^\varphi\to \Fun(E\mathcal M,\mathcal C)^\varphi$ is always fully faithful as a limit of fully faithful functors, so that it is enough to show that it is also essentially surjective. Moreover, \cite[Proposition~7.22]{schwede-k-theory} shows that $\Phi\colon E\mathcal M\to\mathcal C$ is supported on some finite set $A$ if (and only if) all $\Phi(u)$ are supported on $A$ and $\Phi$ factors through the restriction $E(\textup{res}_A)\colon E\mathcal M\to E\Inj(A,\omega)$.

As $H$ is a complete $H$-set universe, $\omega$ contains a finite faithful $H$-subset $S$ (for example, we could take any free $H$-orbit). By faithfulness, $\Inj(S,\omega)$ is a free $H$-set, hence there exists an $H$-equivariant map $\chi\colon\Inj(S,\omega)\to H$. As above one then argues that $E(\chi\circ\textup{res}_S)^*\colon\Fun(EH,\mathcal C)^\varphi\to\Fun(E\mathcal M,\mathcal C)^\varphi$ is an equivalence of categories. Moreover, the above characterization shows that $E(\chi\circ\textup{res}_S)^*$ lands in $\Fun(E\mathcal M,\mathcal C)^\tau=\mathcal C^{\sat}$; more precisely, $E(\chi\circ\textup{res}_S)^*(\Psi)=E(\textup{res}_S)^*E(\chi)^*(\Psi)$ is supported on $S\cup\bigcup_{h\in H}\supp(\Psi(h))$ for any $\Psi\colon EH\to\mathcal C$. We conclude that $(\mathcal C^{\sat})^\varphi\hookrightarrow\Fun(E\mathcal M,\mathcal C)^\varphi$ is essentially surjective, completing the proof of the second claim and hence of the theorem.
\end{proof}

\begin{cor}\label{cor:ws-vs-s}
The inclusion $\cat{$\bmEM$-$\bm G$-Cat}^{\tau,s}\hookrightarrow\cat{$\bmEM$-$\bm G$-Cat}^{\tau,ws}$ is a homotopy equivalence with respect to the $G$-global weak equivalences on both sides.
\end{cor}
\begin{proof}
We claim that $(\blank)^{\sat}$ restricts to the desired homotopy inverse, for which it suffices to show that $s\colon\mathcal C\to\mathcal C^{\sat}$ is a $G$-global weak equivalence for all weakly saturated $\mathcal C$. This is in turn an immediate consequence of the previous theorem together with Lemma~\ref{lemma:cat-between-ws}.
\end{proof}

Next, we turn to parsummable structures. Schwede shows in \cite[Construction~7.25]{schwede-k-theory} (as an application of \cite[Proposition~7.22]{schwede-k-theory} mentioned above) that the canonical isomorphism
$\Fun(E\mathcal M,\mathcal C)\times\Fun(E\mathcal M,\mathcal D)\to\Fun(E\mathcal M,\mathcal C\times\mathcal  D)$ restricts for any tame $E\mathcal M$-categories $\mathcal C,\mathcal D$ to a morphism $\mathcal C^{\sat}\boxtimes\mathcal D^{\sat}\to(\mathcal C\boxtimes\mathcal D)^{\sat}$ and that this together with the unique map ${*}\to {*}^{\sat}$ makes $(\blank)^{\sat}$ into a lax symmetric monoidal functor with respect to $\boxtimes$. In particular, if $\mathcal C$ is a parsummable category, then $\mathcal C^{\sat}$ again admits a natural parsummable structure: explicitly, the addition in $\mathcal C^{\sat}$ is given pointwise, and the unit for the addition is the functor constant at zero. The $E\mathcal M$-equivariant functor $s\colon\mathcal C\to\mathcal C^{\sat}$ is then in fact a morphism of parsummable categories by \cite[Theorem~7.27]{schwede-k-theory}. We therefore immediately conclude from the above:

\begin{cor}\label{cor:sat-ps-cat}
The inclusion $\cat{$\bm G$-ParSumCat}^{s}\hookrightarrow\cat{$\bm G$-ParSumCat}$ of the full subcategory of saturated parsummable categories is a homotopy equivalence with respect to the \emph{categorical equivalences} on both sides. A homotopy inverse is given by the above saturation construction.\qed
\end{cor}

\begin{cor}\label{cor:ws-vs-s-parsummable}
The inclusion $\cat{$\bm G$-ParSumCat}^{s}\hookrightarrow\cat{$\bm G$-ParSumCat}^{ws}$ is a homotopy equivalence with respect to the $G$-global weak equivalences on both sides.\qed
\end{cor}

We will prove in Theorem~\ref{thm:parsum-cat-sat-global} below that also the inclusion \begin{equation*}\cat{$\bm G$-ParSumCat}^{ws}\hookrightarrow\cat{$\bm G$-ParSumCat}\end{equation*} is a homotopy equivalence with respect to the $G$-global weak equivalences, or in other words, that $\cat{$\bm G$-ParSumCat}^{s}\hookrightarrow\cat{$\bm G$-ParSumCat}$ is a homotopy equivalence not only with respect to the categorical equivalences, but also with respect to the $G$-global weak equivalences. We emphasize again that $\mathcal C\to\mathcal C^{\sat}$ is usually \emph{not} a $G$-global weak equivalence, so that this is not a consequence of Corollary~\ref{cor:sat-ps-cat}.

\section{\texorpdfstring{\for{toc}{$E\mathcal M$}\except{toc}{$\bmEM$}}{EM}-simplicial sets and parsummable simplicial sets}\label{sec:parsum-sset}
The functor $(\blank)_0\colon\cat{SSet}\to\cat{Set}$ admits a right adjoint, that we again denote by $E$. Explicitly, $(EX)_n=\prod_{i=0}^nX$ with the evident functoriality in the two variables. By uniqueness of adjoints we can then conclude that the simplicial set $EX$ is canonically isomorphic to the nerve of the category $EX$; more precisely, there is a unique simplicial map that is the identity on vertices, and this map is an isomorphism.

As before, we see that $E\colon\cat{Set}\to\cat{SSet}$ preserves products, and hence sends monoids to simplicial monoids; in particular, we get a monoid $E\mathcal M$, which by the above is identified with $\nerve(E\mathcal M)$ with monoid structure induced by the categorical monoid $E\mathcal M$.

\begin{defi}
An $E\mathcal M$-simplicial set is a simplicial set $X$ together with an action of the simplicial monoid $E\mathcal M$, i.e.~a simplicial map $E\mathcal M\times X\to X$ that is associative and unital. We write $\cat{$\bmEM$-SSet}$ for the category of $E\mathcal M$-simplicial sets and $E\mathcal M$-equivariant simplicial maps.
\end{defi}

If $\mathcal C$ is a small $E\mathcal M$-category, then $\nerve(\mathcal C)$ inherits an $E\mathcal M$-action via
\begin{equation*}
E\mathcal M\times\nerve(\mathcal C)\cong\nerve(E\mathcal M)\times\nerve(\mathcal C)
\cong\nerve(E\mathcal M\times\mathcal C)\xrightarrow{\nerve(\text{act})}\nerve(\mathcal C)
\end{equation*}
where the first isomorphism is the one discussed above and the second one comes from the fact that $\nerve$ preserves finite products. This way, the nerve obviously lifts to a functor $\cat{$\bmEM$-Cat}\to\cat{$\bmEM$-SSet}$.

\begin{rk}\label{rk:nerve-induced-action}
Let $\mathcal C$ be a small $E\mathcal M$-category. We will now make the above $E\mathcal M$-action on $\nerve(\mathcal C)$ explicit, for which we calculate for all $u,v\in\mathcal M$ and $f\colon x\to y$ in $\mathcal C$:
\begin{equation*}
(v,u). f = \big((v,u)\circ\id_u\big).(\id_y\circ f)=
\big((v,u).\id_y\big) \circ (\id_u.f)= [v,u]_y\circ u.f;
\end{equation*}
as $u.f=u^y_\circ f (u_\circ^x)^{-1}$ we conclude from this that the diagram
\begin{equation*}
\begin{tikzcd}[column sep=large]
x\arrow[d, "u_\circ^x"']\arrow[r, "f"] & y\arrow[d, "v_\circ^y"]\\
u.x \arrow[r,"{(v,u).f}"'] & v.y
\end{tikzcd}
\end{equation*}
commutes. Since the vertical maps are isomorphisms, this in fact completely determines $(v,u).f$. We can therefore immediately conclude that the action of $(u_0,\dots,u_k)\in (E\mathcal M)_k$ on a $k$-simplex $x_0\xrightarrow{\alpha_1}x_1\to\cdots\to x_k$ is uniquely characterized by demanding that inserting it as the lower row in
\begin{equation*}
\begin{tikzcd}
x_0\arrow[r, "\alpha_1"]\arrow[d, "(u_0)_\circ^{x_0}"'] & x_1\arrow[d, "(u_1)_\circ^{x_1}"] \arrow[r] & \cdots\arrow[r] & x_k\arrow[d, "(u_k)_\circ^{x_k}"]\\
u_0.x_0\arrow[r] & u_1.x_1\arrow[r] & \cdots\arrow[r] & u_k.x_k
\end{tikzcd}
\end{equation*}
makes all the squares commute.
\end{rk}

\begin{ex}\label{ex:EInj-SSet}
Let $A$ be a (finite) set. Analogously to Example~\ref{ex:EInj-supp}, the canoncial $\mathcal M$-action on $\Inj(A,\omega)$ makes $E\Inj(A,\omega)$ into an $E\mathcal M$-simplicial set. It is then canonically isomorphic to the nerve of the category of the same name considered in the aforementioned example.
\end{ex}

\subsection{Supports and tameness} Next, we want to introduce analogues of the notions of support and tameness for $E\mathcal M$-simplicial sets.

\begin{constr}
Let $0\le k\le n$. We write $i_k\colon\mathcal M\to\mathcal M^{n+1}$ for the homomorphism sending $u\in\mathcal M$ to $(1,\dots,1,u,1\dots,1)$ where $u$ is in the $(k+1)$-th spot.

If $X$ is an $E\mathcal M$-simplicial set, we therefore get $(n+1)$ commuting $\mathcal M$-actions on $X_n$ by pulling back the action of $\mathcal M^{n+1}=(E\mathcal M)_n$ along the injections $i_0,\dots,i_n$.
\end{constr}

\begin{defi}
Let $X$ be an $E\mathcal M$-simplicial set, let $0\le k\le n$, and let $x\in X_n$. Then we say that \emph{$x$ is $k$-supported} on the finite set $A\subset\omega$ if it is supported on $A$ as an element of the $\mathcal M$-set $i_k^*X_n$. We say that $x$ is \emph{$k$-finitely supported} if it is $k$-supported on some finite set, in which case we write $\supp_k(x)$ for its support in $i_k^*X_n$.

We say that $x$ is \emph{supported} on $A$ if it is $k$-supported on $A$ for all $0\le k\le n$, and we call it \emph{finitely supported} if it is supported on some finite set, i.e.~if it is $k$-finitely supported for all $0\le k\le n$. In this case its \emph{support} $\supp(x)$ is defined as $\bigcup_{k=0}^n\supp_k(x)$.

The $E\mathcal M$-simplicial set $X$ is called \emph{tame} if all its simplices are finitely supported. We write $\cat{$\bmEM$-SSet}^\tau\subset\cat{$\bmEM$-SSet}$ for the full subcategory spanned by the tame $E\mathcal M$-simplicial sets.
\end{defi}

\begin{rk}\label{rk:support-vs-m-support}
Any $E\mathcal M$-simplicial set $X$ in particular forgets to an $\mathcal M$-simplicial set, and it is natural to ask whether we can characterize the support of $x\in X_n$ in terms of the resulting $\mathcal M$-action (i.e.~the diagonal action) on $X_n$, analogously to Lemma~\ref{lemma:support-strong}.

We prove in \cite{g-global} that this is indeed the case: $x$ is supported on the finite set $A$ in the above sense if and only if it supported on $A$ as an element of the $\mathcal M$-set $X_n$. However, the combinatorial argument for this is somewhat lengthy, and as we will not need this result for the present article, we have decided to omit it.
\end{rk}

\begin{ex}\label{ex:nerve-supported-on}\label{ex:nerve-kth-support}
Let $\mathcal C$ be a tame $E\mathcal M$-category. We claim that an $n$-simplex $\alpha_\bullet\mathrel{:=}(x_0\xrightarrow{\alpha_1}x_1\to\cdots\to x_n)$ of $\nerve(\mathcal C)$ is $k$-supported on the finite set $A\subset\omega$ if and only if $A\supset\supp(x_k)$, i.e.~if $x_k$ is supported on $A$. In particular this shows that $\alpha_\bullet$ is finitely supported and $\supp_k(\alpha_\bullet)=\supp(x_k)$, $\supp(\alpha_\bullet)=\bigcup_{k=0}^n\supp(x_k)$.

To prove the claim let us first assume that $A$ contains $\supp(x_k)$; we will show that $i_k(u).\alpha_\bullet=\alpha_\bullet$ for all $u$ fixing $A$ pointwise. Indeed, by the description of $(u_0,\dots,u_n).\alpha_\bullet$ from Remark~\ref{rk:nerve-induced-action} applied to $(u_0,\dots,u_n)=i_k(u)$ it suffices that $u_\circ^{x_k}$ is the identity, which is immediate from Lemma~\ref{lemma:support-strong}.

Conversely, assume $\alpha_\bullet$ is $k$-supported on $A$, and let $u$ be any injection fixing $A$ pointwise, so that in particular $i_k(u).\alpha_\bullet=\alpha_\bullet$. Comparing the $k$-th vertices of these $n$-simplices then shows that $u.x_k=x_k$, and letting $u$ vary we see that $x_k$ is supported on $A$ as desired.
\end{ex}

In particular, we conclude from the above example that the nerve restricts to $\cat{$\bmEM$-Cat}^\tau\to\cat{$\bmEM$-SSet}^\tau$.

\begin{ex}\label{ex:EInj-sset-supp}
Let $A$ be a finite set. Then the above together with Example~\ref{ex:EInj-supp} shows that the $E\mathcal M$-simplicial set $E\Inj(A,\omega)$ from Example~\ref{ex:EInj-SSet} is tame and that $\supp_k(i_0,\dots,i_n)=i_k(A)$.
\end{ex}

\begin{warn}
Example~\ref{ex:nerve-kth-support} shows that if $X$ is isomorphic to the nerve of a tame $E\mathcal M$-category, then the $k$-th support of an $n$-simplex $x$ agrees with the support of its $k$-th vertex. This is \emph{not} true for general tame $E\mathcal M$-simplicial sets. Even worse, the support of an $n$-simplex can be strictly larger than the union of the supports of its vertices, for which we will give an example now:

Let $A$ be a non-empty finite set, and let $X$ be obtained from $E\Inj(A,\omega)\times\Delta^1$ by collapsing both copies of $E\Inj(A,\omega)$ to a single point. This is still tame since any $n$-simplex of $X$ can be represented by some $n$-simplex of $E\Inj(A,\omega)\times\Delta^1$, on whose support it is then obviously supported. Moreover, the unique vertex of $X$ has empty support for trivial reasons.

However, the quotient does not identify any two edges $\{i\}\times\Delta^1$, $\{j\}\times\Delta^1$ for distinct injections $i,j\colon A\to\omega$. In particular, $\supp[\{i\}\times\Delta^1]=i(A)\not=\varnothing$.
\end{warn}

\begin{lemma}\label{lemma:k-supp-em-equiv}
Let $f\colon X\to Y$ be an $E\mathcal M$-equivariant map of $E\mathcal M$-simplicial sets, let $0\le k\le n$, and let $x\in X_n$ be $k$-supported on some finite set $A\subset\omega$. Then also $f(x)$ is $k$-supported on $A$.

In particular, if $x$ is supported on $A$, then so is $f(x)$.
\end{lemma}
\begin{proof}
The first statement is an instance of Lemma~\ref{lemma:supp-change}-$(\ref{item:sc-m-equiv})$, and the second one follows immediately from this.
\end{proof}

\begin{lemma}\label{lemma:k-supp-action}
Let $X$ be an $E\mathcal M$-simplicial set, and let $x\in X_n$ be $k$-finitely supported for some $0\le k\le n$. Then $(u_0,\dots,u_n).x$ is $k$-finitely supported for all $u_0,\dots,u_n\in\mathcal M$, and $\supp_k((u_0,\dots,u_n).x)\subset u_k(\supp_k(x))$.
\end{lemma}
\begin{proof}
The set map $(u_0,\dots,u_n).\blank\colon X_n\to X_n$ factors as
\begin{equation*}
((u_0,\dots,u_{k-1},1,u_{k+1},\dots,u_n).\blank)\circ(i_k(u_k).\blank).
\end{equation*}
As a self map of the $\mathcal M$-set $i_k^*X_n$, the former is $\mathcal M$-equivariant, so the claim follows from the two parts of Lemma~\ref{lemma:supp-change}.
\end{proof}

\begin{lemma}\label{lemma:k-supp-simplicial}
Let $X$ be a \emph{tame} $E\mathcal M$-simplicial set, let $x\in X_n$, and let $\varphi\colon[m]\to[n]$ be any map in $\Delta$. Then $\supp_k(\varphi^*x)\subset\supp_{\varphi(k)}(x)$ for all $0\le k\le m$.
\end{lemma}
\begin{proof}
Let $b\in\omega\smallsetminus\supp_{\varphi(k)}(x)$ and let $u$ be an injection fixing $\supp_{\varphi(k)}(x)$ pointwise with $b\notin\im(u)$. Then $i_{\varphi(k)}(u).x=x$, hence
\begin{equation*}
(i_{\varphi(k)}(u)_{\varphi(0)},\dots, i_{\varphi(k)}(u)_{\varphi(m)}).\varphi^*x=\varphi^*(i_{\varphi(k)}(u)).\varphi^*x = \varphi^*(i_{\varphi(k)}(u).x)=\varphi^*(x).
\end{equation*}
As $\varphi^*x$ is $k$-finitely supported, we conclude from the previous lemma that
\begin{align*}
\supp_k(\varphi^*x)&=\supp_k((i_{\varphi(k)}(u)_{\varphi(0)},\dots, i_{\varphi(k)}(u)_{\varphi(m)}).\varphi^*x)\\
&\subset i_{\varphi(k)}(u)_{\varphi(k)}(\supp_k\varphi^*x)=u(\supp_{k}\varphi^*x),
\end{align*}
hence in particular $b\notin\supp_k(\varphi^*x)$. The claim follows by letting $b$ vary.
\end{proof}

\begin{warn}
The above argument presupposes that $\varphi^*x$ is $k$-finitely supported, so it does \emph{not} show that the finitely supported simplices of a general $E\mathcal M$-simplicial set form a subcomplex. While this does indeed hold (see for example Remark~\ref{rk:support-vs-m-support} above), the proof is harder, and as we will only be interested in tame $E\mathcal M$-simplicial sets below, we have decided to only consider the slightly weaker version above.
\end{warn}

We can now prove the following analogue of Lemma~\ref{lemma:support-strong}:

\begin{lemma}\label{lemma:support-strong-sset}
Let $X$ be a \emph{tame} $E\mathcal M$-simplicial set, let $x\in X_n$ be supported on the finite set $A\subset\omega$, and let $\varphi\colon[m]\to[n]$ be a map in $\Delta$. Then $(u_0,\dots,u_n).\varphi^*x=(v_0,\dots,v_n).\varphi^*x$ for all $u_0,\dots,u_n,v_0,\dots,v_n\in\mathcal M$ such that $u_i(a)=v_i(a)$ for all $i=0,\dots,n$ and $a\in A$.
\end{lemma}
\begin{proof}
The previous lemma immediately implies that $\varphi^*x$ is supported on $A$, so we may assume without loss of generality that $m=n$ and $\varphi=\id$. Using Lemma~\ref{lemma:supp-m-set-agree} together with Lemma~\ref{lemma:k-supp-action}, one then easily shows by descending induction that $(1,\dots,1,u_k,\dots,u_n).x=(1,\dots,1,v_k,\dots,v_n).x$ for all $0\le k\le n+1$, which for $k=0$ is precisely what we wanted to prove.
\end{proof}

\begin{prop}\label{prop:corep}
Let $A\subset\omega$ be finite, and let $X$ be a \emph{tame} $E\mathcal M$-simplicial set. Then the simplices of $X$ supported on $A$ form a subcomplex $X_{[A]}\subset X$, and $(\blank)_{[A]}\colon\cat{$\bmEM$-SSet}^\tau\to\cat{SSet}$ is a subfunctor of the forgetful functor.

Moreover, it is corepresentable in the enriched sense by the $E\mathcal M$-simplicial set $E\Inj(A,\omega)$ from Example~\ref{ex:EInj-SSet} via evaluation at the inclusion $\iota_A\colon A\hookrightarrow\omega$.
\end{prop}
\begin{proof}
We first observe that $E\Inj(A,\omega)$ is tame and that $\iota_A$ is supported on $A$ by Example~\ref{ex:EInj-sset-supp}. We will show that the simplicial map \begin{equation*}\ev\colon\Maps_{E\mathcal M}(E\Inj(A,\omega),X)\to X\end{equation*} given by evaluation at $\iota_A$ is injective with image precisely $X_{[A]}$ for each $E\mathcal M$-simplicial set $X$. All the remaining claims will then easily follow from this.

Let us show that the evaluation is injective. Indeed, by definition a $k$-simplex of $\Maps_{E\mathcal M}(E\Inj(A,\omega),X)$ is given by an $E\mathcal M$-equivariant simplicial map \begin{equation*}E\Inj(A,\omega)\times\Delta^n\to X\end{equation*} and we have to show that any two such maps $f,g$ whose restrictions to $\{\iota_A\}\times\Delta^n${\hskip0pt plus 2pt} agree,{\hskip0pt plus 2pt} are{\hskip0pt plus 2pt} already{\hskip0pt plus 2pt} equal.{\hskip0pt plus 2pt} For{\hskip0pt plus 2pt} this,{\hskip0pt plus 2pt} we{\hskip0pt plus 2pt} consider{\hskip0pt plus 2pt} an{\hskip0pt plus 2pt} arbitrary{\hskip0pt plus 2pt} $k$-simplex{\hskip0pt plus 2pt} $((u_0,\dots,u_k),\varphi)$ of $E\Inj(A,\omega)\times\Delta^n$, and we pick an extension $\hat u_i$ of $u_i\colon A\to\omega$ to an injection $\omega\to\omega$ (i.e.~an element of $\mathcal M$) for every $i=0,\dots,k$. If we moreover write $s$ for the unique map $[k]\to[0]$ in $\Delta$, then
\begin{equation*}
    ((u_0,\dots,u_k),\varphi)=(\hat u_0,\dots,\hat u_k).((\iota_A,\dots,\iota_A),\varphi)=(\hat u_0,\dots,\hat u_k).(s^*\iota_A,\varphi)
\end{equation*}
by definition of the action, and hence
\begin{align*}
f((u_0,\dots,u_k),\varphi)&=(\hat u_0,\dots,\hat u_k).f(s^*\iota_A,\varphi)\\
&=(\hat u_0,\dots,\hat u_k).g(s^*\iota_A,\varphi)=g((u_0,\dots,u_k),\varphi)
\end{align*}
by $E\mathcal M$-equivariance and the assumption. This proves injectivity.

Lemma~\ref{lemma:k-supp-em-equiv} shows that $\ev(f)$ is supported on $A$ for all $f\colon E\Inj(A,\omega)\times\Delta^n\to X$. Conversely, let $x$ be an $n$-simplex supported on $A$. Then Lemma~\ref{lemma:support-strong-sset} shows that the assignment
\begin{align*}
f_x\colon E\Inj(A,\omega)\times\Delta^n &\to X\\
((u_0,\dots,u_k),\varphi) &\mapsto (\hat u_0,\dots,\hat u_k).\varphi^*x
\end{align*}
is independent of the chosen extensions $\hat u_i\in\mathcal M$. From this it easily follows that $f_x$ is simplicial, $E\mathcal M$-equivariant, and that $\ev f_x=x$, proving surjectivity.
\end{proof}

In particular, we see that any map $\alpha\colon K\to X_{[A]}$ admits a unique $E\mathcal M$-equivariant extension $E\Inj(A,\omega)\times K\to X$; we will usually denote this extension by $\tilde\alpha$.

\subsection{Parsummable simplicial sets} We are now ready to introduce the box product of tame $E\mathcal M$-simplicial sets:

\begin{constr}
Let $X,Y$ be tame $E\mathcal M$-simplicial sets, and let $n\ge 0$. We define $(X\boxtimes Y)_n\subset (X\times Y)_n$ to consist of precisely those pairs $(x,y)$ such that $\supp_k(x)\cap\supp_k(y)=\varnothing$ for all $0\le k\le n$.
\end{constr}

\begin{prop}
Let $X,Y$ be tame $E\mathcal M$-simplicial sets. Then the above defines an $E\mathcal M$-simplicial subset $X\boxtimes Y\subset X\times Y$, which we call the \emph{box product} of $X$ and $Y$. Both $X\boxtimes Y$ and $X\times Y$ are tame, and $\blank\boxtimes\blank$ is a subfunctor
\begin{equation*}
\cat{$\bmEM$-SSet}^\tau\times\cat{$\bmEM$-SSet}^\tau\to\cat{$\bmEM$-SSet}^\tau
\end{equation*}
of the cartesian product.
\end{prop}
\begin{proof}
Lemma~\ref{lemma:k-supp-simplicial} shows that $X\boxtimes Y$ is a subcomplex, and it is closed under the (diagonal) $E\mathcal M$-action by Lemma~\ref{lemma:k-supp-action}. Moreover, if $f\colon X\to X'$ and $g\colon Y\to Y'$ are $E\mathcal M$-equivariant, then $(f\times g)(X\boxtimes Y)\subset X'\boxtimes Y'$ by Lemma~\ref{lemma:k-supp-em-equiv}.

It only remains to show that $X\times Y$ (and hence $X\boxtimes Y$) is tame, for which it suffices to observe that $(x,y)$ is by definition $k$-supported on $\supp_k(x)\cup\supp_k(y)$ (in fact, $\supp_k(x,y)=\supp_k(x)\cup\supp_k(y)$) for all $x\in X_n$, $y\in Y_n$, and $0\le k\le n$, and hence in particular supported on $\supp(x)\cup\supp(y)$ (in fact, $\supp(x,y)=\supp(x)\cup\supp(y)$).
\end{proof}

\begin{prop}
The unitality, associativity, and symmetry isomorphisms of the cartesian product on $\cat{$\bmEM$-SSet}^\tau$ restrict to corresponding isomorphisms for $\boxtimes$. This makes $\cat{$\bmEM$-SSet}^\tau$ into a symmetric monoidal category with tensor product $\boxtimes$ and unit the terminal $E\mathcal M$-simplicial set.
\end{prop}
\begin{proof}
We will show that the associativity isomorphism $(X\times Y)\times Z\to X\times(Y\times Z)$ restricts to an isomorphism $(X\boxtimes Y)\boxtimes Z\to X\boxtimes(Y\boxtimes Z)$ for all $X,Y,Z\in\cat{$\bmEM$-SSet}^\tau$.

Indeed, we have to show that if $x\in X,y\in Y,z\in Z$, then $((x,y),z)\in (X\boxtimes Y)\boxtimes Z$ if and only if $(x,(y,z))\in X\boxtimes (Y\boxtimes Z)$. But the first condition is equivalent to demanding that $\supp_k(x)\cap\supp_k(y)=\varnothing$ and $\supp_k(x,y)\cap\supp_k(z)=\varnothing$. We have seen in the proof of the previous proposition that $\supp_k(x,y)=\supp_k(x)\cup\supp_k(y)$, so these two together are equivalent to demanding that $\supp_k(x),\supp_k(y),\supp_k(z)$ be pairwise disjoint. By a symmetric argument this is then in turn equivalent to $(x,(y,z))\in X\boxtimes(Y\boxtimes Z)$ as desired.

The arguments for the symmetry and unitality isomorphisms are similar, and we omit them. All the necessary coherence conditions of the resulting isomorphisms then follow automatically from the corresponding results for the cartesian symmetric monoidal structure.
\end{proof}

\begin{defi}
A \emph{parsummable simplicial set} is a commutative monoid for $\boxtimes$ in $\cat{$\bmEM$-SSet}^\tau$. We write $\cat{ParSumSSet}$ for the corresponding category.
\end{defi}

\begin{prop}
The canonical isomorphism $\nerve(\mathcal C)\times\nerve(\mathcal D)\to\nerve(\mathcal C\times\mathcal D)$ restricts to an isomorphism $\nerve(\mathcal C)\boxtimes\nerve(\mathcal D)\to\nerve(\mathcal C\boxtimes\mathcal D)$ for all $\mathcal C,\mathcal D\in\cat{$\bmEM$-Cat}^\tau$. Together with the unique map ${*}\to\nerve({*})$ this makes $\nerve\colon\cat{$\bmEM$-Cat}^\tau\to\cat{$\bmEM$-SSet}^\tau$ into a strong symmetric monoidal functor with respect to the box products on both sides.
\end{prop}
\begin{proof}
Let us prove the first statement, which amounts to saying that if
\begin{equation*}
x_0\xrightarrow{\alpha_1}x_1\to\cdots x_n\qquad\text{and}\qquad
y_0\xrightarrow{\beta_1}y_1\to\cdots y_n
\end{equation*}
are $n$-simplices of $\nerve(\mathcal C)$ and $\nerve(\mathcal D)$, respectively, then
\begin{equation}\label{eq:product-of-morphisms}
(x_0,y_0)\xrightarrow{(\alpha_1,\beta_1)}(x_1,y_1)\to\cdots\to (x_n,y_n)
\end{equation}
lies in the image of $\nerve(\mathcal C\boxtimes\mathcal D)\to\nerve(\mathcal C\times\mathcal D)$ if and only if $(\alpha_\bullet,\beta_\bullet)\in\nerve(\mathcal C)\boxtimes\nerve(\mathcal D)$. But indeed, the latter condition is equivalent to $\supp_k(\alpha_\bullet)\cap\supp_k(\beta_\bullet)=\varnothing$ for all $0\le k\le n$, which by Example~\ref{ex:nerve-kth-support} is further equivalent to $\supp(x_k)\cap\supp(y_k)=\varnothing$ for all $0\le k\le n$. But this is by definition equivalent to $(x_k,y_k)\in\mathcal C\boxtimes\mathcal D$ for all $0\le k\le n$, which is in turn equivalent to $(\alpha_k,\beta_k)\colon (x_{k-1},y_{k-1})\to (x_k,y_k)$ being a morphism in $\mathcal C\boxtimes\mathcal D$ for $1\le k\le n$ as $\mathcal C\boxtimes\mathcal D\subset\mathcal C\times\mathcal D$ is a full subcategory. Finally, by definition of the nerve this is further equivalent to $(\ref{eq:product-of-morphisms})$ lying in the image of $\nerve(\mathcal C\boxtimes\mathcal D)\to\nerve(\mathcal C\times\mathcal D)$, which completes the proof of the first statement.

It is clear that also ${*}\to\nerve({*})$ is an isomorphism. As all the structure isomorphisms on both $\cat{$\bmEM$-Cat}^\tau$ and $\cat{$\bmEM$-SSet}^\tau$ are defined as restrictions of the structure isomorphisms of the cartesian symmetric monoidal structures, all the necessary coherence conditions hold automatically, which completes the proof of the proposition.
\end{proof}

In particular, we see that the nerve lifts to $\cat{ParSumCat}\to\cat{ParSumSSet}$. Explicitly, this sends a parsummable category $\mathcal C$ to $\nerve(\mathcal C)$ with $E\mathcal M$-action as above. The additive unit is given by the vertex $0\in\mathcal C$, and if
\begin{equation*}
x_0\xrightarrow{\alpha_1}x_1\to\cdots\to x_n\qquad\text{and}\qquad
y_0\xrightarrow{\beta_1}y_1\to\cdots\to y_n
\end{equation*}
are summable $n$-simplices, then $\supp(x_k)\cap\supp(y_k)$ for $0\le k\le n$, and $\alpha_\bullet+\beta_\bullet$ is the $n$-simplex
\begin{equation*}
(x_0+y_0)\xrightarrow{\alpha_1+\beta_1}(x_1+y_1)\to\cdots\to (x_n+y_n).
\end{equation*}

\subsection{\texorpdfstring{$\bm G$}{G}-global homotopy theory of \texorpdfstring{$\bm G$}{G}-parsummable simplicial sets}
Recall that we fixed a discrete group $G$. As before, we can extend the box product formally to the category $\cat{$\bmEM$-$\bm G$-SSet}^\tau$ of $G$-objects in $\cat{$\bmEM$-SSet}^\tau$, i.e.~tame $E\mathcal M$-simplicial sets with a $G$-action through $E\mathcal M$-equivariant morphisms, which we can further identify with simplicial sets with an action of the simplicial monoid $E\mathcal M\times G$, so that the underlying $E\mathcal M$-simplicial set is tame.

The category $\cat{$\bm G$-ParSumSSet}$ of commutative monoids in $\cat{$\bmEM$-$\bm G$-SSet}^\tau$ is then canonically identified with the $G$-objects in $\cat{ParSumSSet}$. Moreover, the nerve lifts to a strong symmetric monoidal functor $\cat{$\bmEM$-$\bm G$-Cat}^\tau\to\cat{$\bmEM$-$\bm G$-SSet}^\tau$ inducing $\nerve\colon\cat{$\bm G$-ParSumCat}\to\cat{$\bm G$-ParSumSSet}$. We now want to consider these from a $G$-global perspective; again using the notation $(\blank)^\varphi$ for the fixed points with respect to the graph subgroup $\Gamma_{H,\varphi}$, we define for this:

\begin{defi}
A morphism $f\colon X\to Y$ in $\cat{$\bmEM$-$\bm G$-SSet}$ is called a \emph{$G$-global weak equivalence} if $f^\varphi\colon X^\varphi\to Y^\varphi$ is a weak homotopy equivalence for all universal subgroups $H\subset\mathcal M$ and all homomorphisms $\varphi\colon H\to G$.
\end{defi}

\begin{defi}
A morphism $f\colon X\to Y$ in $\cat{$\bm G$-ParSumSSet}$ is called a \emph{$G$-global weak equivalence} if its underlying morphism of $E\mathcal M$-$G$-simplicial sets is.
\end{defi}

\begin{rk}
As the nerve is a right adjoint, it commutes with taking $\varphi$-fixed points (up to canonical isomorphism). Thus $\nerve\colon\cat{$\bmEM$-$\bm G$-Cat}\to\cat{$\bmEM$-$\bm G$-SSet}$ preserves and reflects weak equvialences, and so does $\nerve\colon\cat{$\bm G$-ParSumCat}\to\cat{$\bm G$-ParSumSSet}$.
\end{rk}

\section{The \texorpdfstring{\for{toc}{$E\mathcal M$}\except{toc}{$\bmEM$}}{EM}-category associated to an \texorpdfstring{\for{toc}{$E\mathcal M$}\except{toc}{$\bmEM$}}{EM}-simplicial set}\label{sec:construction}
While $\nerve\colon\cat{Cat}\to\cat{SSet}$ is a homotopy equivalence, its left adjoint $\h$ (sending a simplicial set to its \emph{homotopy category}) is not homotopically meaningful. Instead, a possible homotopy inverse (going back to Quillen) of the nerve is the following:

\begin{defi}
Let $X$ be a simplicial set. Its \emph{category of simplices} $\Delta\downarrow X$ is the small category with objects the simplicial maps $f\colon\Delta^n\to X$ ($n\ge0$) and morphisms $\alpha\colon f\to g$ those simplicial maps $\alpha$ satisfying $f=g\circ\alpha$.
\end{defi}

If $S\subset[m]$, let us write $\Delta^S$ for the unique $(|S|-1)$-simplex of $\Delta^m$ with set of vertices $S$.

\begin{constr}
Let $X$ be simplicial set. A general $k$-simplex $\alpha_\bullet$ of $\nerve(\Delta\downarrow X)$ then corresponds to a diagram
\begin{equation*}
\begin{tikzcd}
\Delta^{n_0}\arrow[r,"\alpha_1"]\arrow[dr, bend right=10pt, "f_0"'] & \Delta^{n_1}\arrow[r, "\alpha_2"]\arrow[d, "f_1"] & \cdots \arrow[r, "\alpha_k"] & \Delta^{n_k}.\arrow[dll, bend left=10pt, "f_k"]\\
& X
\end{tikzcd}
\end{equation*}
There is a unique $k$-simplex $\sigma_{\alpha_\bullet}$ of $\Delta^{n_k}$ with $\ell$-th vertex ($0\le\ell\le k$) given by $\alpha_k\cdots\alpha_{\ell+1}(\Delta^{\{n_\ell\}})$ as $\Delta^{n_k}$ is the nerve of a poset and since $\Delta^{\{n_{\ell+1}\}}\ge\alpha_{\ell+1}(\Delta^{\{n_\ell\}})$ in $\Delta^{n_{\ell+1}}$ for all $\ell=0,\dots,k-1$.

We now define the `last vertex map' $\epsilon\colon\nerve(\Delta\downarrow X)\to X$ via $\epsilon(\alpha_\bullet)\mathrel{:=}f_k(\sigma_{\alpha_\bullet})$.
\end{constr}

One can show that $\epsilon$ is indeed a simplicial map, and that it is natural with respect to the functoriality of $\Delta\downarrow\blank$ via postcomposition. If $X$ is the nerve of a category, the above construction appears in \cite[VI.3]{illusie}, while the general version seems to originate with Thomason \cite[Proposition~4.2]{thomason}.

\begin{prop}\label{prop:last-vertex}
For any simplicial set $X$ the `last vertex map' $\epsilon\colon\nerve(\Delta\downarrow X)\to X$ is a weak homotopy equivalence.
\end{prop}
\begin{proof}
Thomason proves a topological analogue of this as \cite[Proposition~4.2]{thomason}; unfortunately, this does not immediately imply the above simplicial version because it is not clear a priori that $\nerve(\Delta\downarrow\blank)$ preserves weak homotopy equivalences.

Instead, we will use that the last vertex map is an $\infty$-categorical localization, see e.g.~\cite[Proposition~7.3.15]{cisinski}. As any $\infty$-categorical localization is in particular a weak homotopy equivalence, this immediately implies the proposition.
\end{proof}

One crucial step \cite[Proposition~4.5]{thomason} in Thomason's comparison between symmetric monoidal categories and connective spectra is a variant of the above construction yielding a functor from $E_\infty$ spaces to lax symmetric monoidal categories. Similarly, our proofs of Theorems~\ref{thm:perm-vs-par-sum-cat} and~\ref{thm:par-sum-cat-vs-sset} will rely on a parsummable refinement $\textup{C}_\bullet$ of it. The rest of this section is devoted to constructing the underlying $E\mathcal M$-category of this together with an analogue of the `last vertex map.'

\begin{constr}
Let $X$ be an $E\mathcal M$-simplicial set. We define a small category $\textup{C}_X$ as follows: an object of $\textup{C}_X$ is a quadruple $(A,S,m_\bullet,f)$ consisting of two finite subsets $A,S\subset\omega$, a family $(m_a)_{a\in A}$ of non-negative integers $m_a\ge0$, and an $E\mathcal M$-equivariant map $f\colon E\Inj(S,\omega)\times\prod_{a\in A}\Delta^{m_a}\to X$, where $E\mathcal M$ acts on $E\Inj(S,\omega)$ as in Example~\ref{ex:EInj-SSet}. A morphism $(A,S,m_\bullet,f)\to (B,T,n_\bullet,g)$ is an $E\mathcal M$-equivariant map $\alpha\colon E\Inj(S,\omega)\times\prod_{a\in A}\Delta^{m_a}\to E\Inj(T,\omega)\times\prod_{b\in B}\Delta^{n_b}$ such that $g\alpha=f$. Composition is inherited from the composition in $\cat{$\bmEM$-SSet}$; in particular, the identity of $(A,S,m_\bullet,f)$ is given by the identity of $E\Inj(S,\omega)\times\prod_{a\in A}\Delta^{m_a}$.

We now define for each $u\in\mathcal M$ and each object $(A,S,m_\bullet,f)$ of $\textup{C}_X$ the object $u.(A,S,m_\bullet,f)$ as the quadruple $(u(A),u(S),m_{u^{-1}(\bullet)},f\circ (u^*\times u^*))$ where $(m_{u^{-1}(\bullet)})_b=m_{u^{-1}(b)}$ for each $b\in u(A)$, $u^*\colon E\Inj(u(S),\omega)\to E\Inj(S,\omega)$ is restriction along $u\colon S\to u(S)$, and $u^*\colon\prod_{b\in u(A)}\Delta^{m_{u^{-1}(b)}}\to\prod_{a\in A}\Delta^{m_a}$ is the unique map with $\pr_a\circ u^*=\pr_{u(a)}$.

Finally, we define $u_\circ^{(A,S,m_\bullet,f)}\colon(A,S,m_\bullet,f)\to u.(A,S,m_\bullet,f)$ as
\begin{equation*}
(u^*\times u^*)^{-1}\colon E\Inj(S,\omega)\times\prod_{a\in A}\Delta^{m_a}\to E\Inj(u(S),\omega)\times\prod_{b\in u(A)}\Delta^{m_{u^{-1}(b)}}.
\end{equation*}
\end{constr}

\begin{warn}
We have to be slightly careful as two different morphisms in $\textup{C}_X$ might be given by the same morphism of $E\mathcal M$-simplicial sets. As a consequence, whenever we want to prove two morphisms in $\textup{C}_X$ to be equal, we first have to show that their sources and targets agree.
\end{warn}

\begin{lemma}
The above defines an $E\mathcal M$-action on $\textup{C}_X$.
\end{lemma}
\begin{proof}
It is clear that $u^*\times u^*$ is $E\mathcal M$-equivariant, so that \begin{equation*}(u(A),u(S),m_{u^{-1}(\bullet)},f\circ(u^*\times u^*))\end{equation*} is again an object of $\textup{C}_X$. Moreover, it is clearly an isomorphism, so that $(u^*\times u^*)^{-1}$ is well-defined and again $E\mathcal M$-equivariant; as it tautologically commutes with the reference maps to $X$, we see that $u^{(A,S,m_\bullet,f)}_\circ$ is indeed an isomorphism $(A,S,m_\bullet,f)\to u.(A,S,m_\bullet,f)$.

To finish the proof it suffices that the above defines an $\mathcal M$-action on $\Ob(\textup{C}_X)$ and that
\begin{equation}\label{eq:CX-EM-associativity}
u^{v.(A,S,m_\bullet,f)}_\circ v_\circ^{(A,S,m_\bullet,f)}=(uv)_\circ^{(A,S,m_\bullet,f)}
\end{equation}
for all $u,v\in\mathcal M$ and $(A,S,m_\bullet,f)\in\textup{C}_X$.

It is clear from the definition that $1.(A,S,m_\bullet,f)=(A,S,m_\bullet,f)$. Moreover, one easily checks that the diagram
\begin{equation*}
\begin{tikzcd}
E\Inj((uv)(S),\omega)\times\prod\limits_{c\in (uv)(A)}\Delta^{m_{(uv)^{-1}(c)}}\arrow[r, "(uv)^*\times(uv)^*"]\arrow[d, equal] & E\Inj(S,\omega)\times\prod\limits_{a\in A}\Delta^{m_a}\\
E\Inj(u(v(S)),\omega)\times\prod\limits_{c\in u(v(A))}\Delta^{(m_{v^{-1}(\bullet)})_{u^{-1}(c)}}\arrow[r,"u^*\times u^*"'] &
E\Inj(v(S),\omega)\times\prod\limits_{b\in v(A)}\Delta^{m_{v^{-1}(b)}}\arrow[u, "v^*\times v^*"']
\end{tikzcd}
\end{equation*}
commutes, which immediately implies the associativity of the $\mathcal M$-action. Moreover, it shows that the identity $(\ref{eq:CX-EM-associativity})$ holds as morphisms in $\cat{$\bmEM$-SSet}$; as both sides are morphisms $(A,S,m_\bullet,f)\to (uv).(A,S,m_\bullet,f)=u.(v.(A,S,m_\bullet,f))$, they then also agree as morphisms in $\textup{C}_X$, which completes the proof of the lemma.
\end{proof}

\begin{lemma}\label{lemma:C-bullet-supp}
The $E\mathcal M$-category $\textup{C}_X$ is tame. Moreover, $\supp(A,S,m_\bullet,f)=A\cup S$ for any object $(A,S,m_\bullet,f)\in\textup{C}_X$.
\end{lemma}
\begin{proof}
Let us first show that $(A,S,m_\bullet,f)$ is supported on $A\cup S$, which will in particular imply tameness of $\textup{C}_X$. If $u$ fixes $A$ and $S$ pointwise, then obviously $u(A)=A,u(S)=S$ and $m_{u^{-1}(\bullet)}=m_\bullet$. Moreover, it is clear from the definition that both $u^*\colon E\Inj(u(S),\omega)\to E\Inj(S,\omega)$ and $u^*\colon\prod_{a\in A}\Delta^{m_a}\to\prod_{b\in u(A)}\Delta^{m_{u^{-1}(b)}}$ are the respective identities, so $f\circ (u^*\times u^*)=f$, and hence altogether $u.(A,S,m_\bullet,f)=(A,S,m_\bullet,f)$ as desired.

Conversely, let $(A,S,m_\bullet,f)$ be supported on some finite set $B$; we have to show that $A\subset B$ and $S\subset B$. We will only prove the first statement (the argument for the second one being analogous), for which we argue by contradiction: if $A\not\subset B$, then we choose any $a\in A\smallsetminus B$ and an injection $u$ fixing $B$ pointwise such that $a\notin\im u$. But then $a\notin u(A)$, hence $u(A)\not=A$ and $u.(A,S,m_\bullet,f)\not=(A,S,m_\bullet,f)$ contradicting the assumption that $(A,S,m_\bullet,f)$ be supported on $B$.
\end{proof}

\begin{constr}
Let $\varphi\colon X\to Y$ be an $E\mathcal M$-equivariant map. We define $\textup{C}_\varphi\colon \textup{C}_X\to\textup{C}_Y$ as follows: an object $(A,S,m_\bullet,f)$ is sent to $(A,S,m_\bullet,\varphi\circ f)$ and a morphism $(A,S,m_\bullet,f)\to (B,T,n_\bullet,g)$ given by $\alpha\colon E\Inj(S,\omega)\times\prod_{a\in A}\Delta^{m_a}\to E\Inj(T,\omega)\times\prod_{b\in B}\Delta^{n_b}$ is sent to the morphism $(A,S,m_\bullet,\varphi\circ f)\to (B,T,n_\bullet,\varphi\circ g)$ given by the same $\alpha$.
\end{constr}

\begin{lemma}\label{lemma:c-bullet-functor}
In the above situation, $\textup{C}_\varphi$ is an $E\mathcal M$-equivariant functor. This defines a functor $\textup{C}_\bullet\colon\cat{$\bmEM$-SSet}\to\cat{$\bmEM$-Cat}^\tau$.
\end{lemma}
\begin{proof}
It is clear that $\textup{C}_\varphi$ is a well-defined functor and that it commutes with the $\mathcal M$-action on objects. To show that it is $E\mathcal M$-equivariant, it is then enough to show that $\textup{C}_\varphi(u^{(A,S,m_\bullet,f)}_\circ)=u^{(A,S,m_\bullet,\varphi\circ f)}_\circ$. As we already know that both sides are maps between the same objects, it suffices to prove this as maps in $\cat{$\bmEM$-SSet}$, where it is indeed immediate from the definition that both sides are given by $(u^*\times u^*)^{-1}\colon E\Inj(S,\omega)\times\prod_{a\in A}\Delta^{m_a}\to E\Inj(u(S),\omega)\times\prod_{b\in u(A)}\Delta^{m_{u^{-1}(b)}}$.

Finally, it is obvious from the definition that $\textup{C}_{\id}=\id$, and that $\textup{C}_\psi\textup{C}_\varphi=\textup{C}_{\psi\varphi}$ for any further $E\mathcal M$-equivariant map $\psi\colon Y\to Z$, which then completes the proof of the lemma.
\end{proof}

In order to construct the $E\mathcal M$-equivariant refinement of the `last vertex map,' we need the following easy structural insight on the $E\mathcal M$-simplicial sets appearing in the definition of $\textup{C}_\bullet$:

\begin{rk}
If $A$ is any set and $(m_a)_{a\in A}$ is an $A$-tuple of non-negative integers, then $\prod_{a\in A}\Delta^{m_a}$ is isomorphic to the nerve of the poset $\prod_{a\in A}[m_a]$. The latter has a unique terminal object (i.e.~maximum element) given by $(m_a)_{a\in A}$, and we write ${*}$ for the corresponding vertex of $\prod_{a\in A}\Delta^{m_a}$, i.e.~${*}=\prod_{a\in A}\Delta^{\{m_a\}}$.

If $S$ is any further set, then $E\Inj(S,\omega)$ is by construction the nerve of a category in which there is precisely one morphism $u\to v$ for any two objects $u,v$. It follows, that there exists for any $u\in\Inj(S,\omega)$ and any vertex $x$ of $E\Inj(S,\omega)\times\prod_{a\in A}\Delta^{m_a}$ a unique edge $x\to (u,{*})$.

Finally, again using that in $E\Inj(S,\omega)$ and $[m_a]$ there is at most one morphism $x\to y$ for any two objects $x,y$, we see that any $n$-simplex of $E\Inj(S,\omega)\times\prod_{a\in A}\Delta^{m_a}$ is completely determined by its $(n+1)$-tuple of vertices. Conversely, such an $(n+1)$-tuple $(x_0,\dots,x_n)$ comes from an $n$-simplex if and only if there exists for each $1\le i\le n$ a (necessarily unique) edge $x_{i-1}\to x_i$.
\end{rk}

\begin{constr}\label{constr:epsilon}
Let $X$ be an $E\mathcal M$-simplicial set. We define $\epsilon\colon\nerve(\textup{C}_X)\to X$ as follows: if $(A_0,S_0,m^{(0)}_\bullet,f_0)\xrightarrow{\alpha_1}(A_1,S_1,m^{(1)}_\bullet,f_1)\to\cdots\xrightarrow{\alpha_k}(A_k,S_k,m^{(k)}_\bullet,f_k)$ is a $k$-simplex of $\nerve(\textup{C}_X)$, then we denote by $\sigma_{\alpha_\bullet}$ the unique $k$-simplex of \begin{equation*}E\Inj(S_k,\omega)\times\prod_{a\in A_k}\Delta^{m^{(k)}_a}\end{equation*} whose{\hskip0pt plus 2pt} $\ell$-th{\hskip0pt plus 2pt} vertex{\hskip0pt plus 2pt} ($\ell=0,\dots,k$){\hskip0pt plus 2pt} is{\hskip0pt plus 2pt} given{\hskip0pt plus 2pt} by{\hskip0pt plus 2pt} $\alpha_k\cdots\alpha_{\ell+1}(\iota_{S_\ell},{*})$,{\hskip0pt plus 2pt} where{\hskip0pt plus 2pt} $\iota_{S_\ell}\in\allowbreak\Inj(S_\ell,\omega)$ denotes the inclusion. This is indeed well-defined as there exists an edge $\alpha_{\ell}(\iota_{S_{\ell-1}},{*})\to(\iota_{S_{\ell}},{*})$ in $E\Inj(S_\ell,\omega)\times\Delta^{n_{\ell}}$ for all $1\le\ell\le n$.

We then set $\epsilon(\alpha_\bullet)\mathrel{:=}f_k(\sigma_{\alpha_\bullet})\in X_k$.
\end{constr}

\begin{prop}
The above defines a natural transformation $\epsilon\colon\nerve\circ\textup{C}_\bullet\Rightarrow\id$ of endofunctors of $\cat{$\bmEM$-SSet}$.
\end{prop}
\begin{proof}
Let us first show that $\epsilon_X$ is indeed a simplicial map; this is completely analogous to the argument for the usual last vertex map, but we nevertheless include it for completeness. For this we let $(A_0,S_0,m^{(0)}_\bullet,f_0)\xrightarrow{\alpha_1}\cdots\xrightarrow{\alpha_k}(A_k,S_k,m^{(k)}_\bullet,f_k)$ be any $k$-simplex of $\nerve\textup{C}_X$, and we let $\varphi\colon[\ell]\to[k]$ be any map in $\Delta$. Then $\varphi^*(\sigma_{\alpha_\bullet})$ is the unique $\ell$-simplex of $E\Inj(S_k,\omega)\times\prod_{a\in A_k}\Delta^{m^{(k)}_a}$ with $i$-th vertex $\alpha_k\cdots\alpha_{\varphi(i)+1}(\iota_{S_{\varphi(i)}},{*})$. On the other hand, $\sigma_{\varphi^*\alpha_\bullet}$ is by definition the unique $\ell$-simplex of $E\Inj(S_{\varphi(\ell)},\omega)\times\prod_{a\in A_{\varphi(\ell)}}\Delta^{m^{(\varphi(\ell))}_a}$ with $i$-th vertex given by
\begin{equation*}
\varphi^*(\alpha_\bullet)_\ell\cdots\varphi^*(\alpha_\bullet)_{i+1}(\iota_{S_{\varphi(i)}},{*})
=\alpha_{\varphi(\ell)}\cdots\alpha_{\varphi(i)+1}(\iota_{S_{\varphi(i)}},{*}).
\end{equation*}
Thus, $\alpha_k\cdots\alpha_{\varphi(\ell)+1}(\sigma_{\varphi^*(\alpha_\bullet)})=\varphi^*\sigma_{\alpha_\bullet}$ and hence
\begin{align*}
\epsilon(\varphi^*(\alpha_\bullet))&=f_{\varphi(\ell)}(\sigma_{\varphi^*(\alpha_\bullet)})=
f_{k}\alpha_k\cdots\alpha_{\varphi(\ell)+1}(\sigma_{\varphi^*(\alpha_\bullet)})\\
&=f_k(\varphi^*\sigma_{\alpha_\bullet})=\varphi^*f_k(\sigma_{\alpha_\bullet})=\varphi^*\epsilon(\alpha_\bullet),
\end{align*}
i.e.~$\epsilon$ is indeed a simplicial map.

Next, we have to show that $\epsilon$ is $E\mathcal M$-equivariant, for which we let $(u_0,\dots,u_k)\in\mathcal M^{k+1}$ arbitrary. Then we have a commutative diagram
\begin{equation*}
\begin{tikzcd}[column sep=small]
(A_0,S_0,m^{(0)}_\bullet,f_0)\arrow[r, "\alpha_1"]\arrow[d, "(u_0)_\circ"'] & (A_1,S_1,m^{(1)}_\bullet,f_1)\arrow[d, "(u_1)_\circ"] \arrow[r, "\alpha_2"] & \cdots \arrow[r, "\alpha_k"] & (A_k,S_k,m^{(k)}_\bullet,f_k)\arrow[d, "(u_k)_\circ"]\\
u_0.(A_0,S_0,m^{(0)}_\bullet,f_0)\arrow[r] & u_1.(A_1,S_1,m^{(1)}_\bullet,f_1)\arrow[r] & \cdots \arrow[r] & u_k.(A_k,S_k,m^{(k)}_\bullet,f_k)
\end{tikzcd}
\end{equation*}
in $\textup{C}_X$, where the lower row is given by $(u_0,\dots,u_k).\alpha_\bullet$. Thus, $\sigma_{(u_0,\dots,u_k).\alpha_\bullet}$ is the unique $k$-simplex with $i$-th vertex given by
\begin{equation}\label{eq:vertex-action}
(u_k)_\circ\alpha_k\cdots\alpha_{i+1}(u_i)_\circ^{-1}(\iota_{u(A_i)},{*}).
\end{equation}
By definition, $(u_i)_\circ^{-1}(\iota_{u(A_i)},{*})=(u_i|_A,{*})=u_i.(\iota_A,{*})$; $E\mathcal M$-equivariance of $\alpha_k,\dots,\allowbreak\alpha_{i+1}$ therefore implies that $(\ref{eq:vertex-action})$ equals $(u_k)_\circ(u_i.(\alpha_k\cdots\alpha_{i+1}(\iota_{A_i},{*})))$. Comparing vertices, we conclude that $\sigma_{(u_0,\dots,u_k).\alpha_\bullet}=(u_k^*\times u_k^*)^{-1}((u_0,\dots,u_k).\sigma_{\alpha_\bullet})$, hence
\begin{align*}
\epsilon((u_0,\dots,u_k).\alpha_\bullet)&=f_k\circ(u_k^*\times u_k^*)(\sigma_{(u_0,\dots,u_k).\alpha_\bullet})=f_k((u_0,\dots,u_k).\sigma_{\alpha_\bullet})\\
&=(u_0,\dots,u_k).f_k(\sigma_{\alpha_\bullet})=(u_0,\dots,u_k).\epsilon(\alpha_\bullet),
\end{align*}
i.e.~$\epsilon$ is $E\mathcal M$-equivariant.

Finally, let us show that $\epsilon$ is natural. If $\varphi\colon X\to Y$ is any $E\mathcal M$-equivariant map, then $\nerve(\textup{C}_\varphi)(\alpha_\bullet)=(A_0,S_0,m^{(0)}_\bullet,\varphi\circ f_0)\xrightarrow{\alpha_1}\cdots\xrightarrow{\alpha_k}(A_k,S_k,m^{(k)},\varphi\circ f_k)$. Thus, $\sigma_{\nerve(\textup{C}_\varphi)(\alpha_\bullet)}=\sigma_{\alpha_\bullet}$ and $\epsilon(\nerve(\textup{C}_\varphi)(\alpha_\bullet))=\varphi f_k(\sigma_{\nerve(\textup{C}_\varphi)(\alpha_\bullet)})=\varphi f_k(\sigma_{\alpha_\bullet})=\varphi(\epsilon(\alpha_\bullet))$. This completes the proof of the proposition.
\end{proof}

\begin{rk}\label{rk:tilde-epsilon}
Let $\mathcal C$ be a small $E\mathcal M$-category. Then applying Construction~\ref{constr:epsilon} to the $E\mathcal M$-simplicial set $\nerve(\mathcal C)$ yields an $E\mathcal M$-equivariant map $\nerve(\textup{C}_{\nerve\mathcal C})\to\nerve(\mathcal C)$. As the nerve is fully faithful, this is induced by a unique functor $\tilde\epsilon\colon\textup{C}_{\nerve\mathcal C}\to\mathcal C$, which is then automatically $E\mathcal M$-equivariant again. This way, we get a (unique) natural transformation $\tilde\epsilon\colon\textup{C}_\bullet\circ\nerve\Rightarrow\id$ of endofunctors of $\cat{$\bmEM$-Cat}$ with $\nerve(\tilde\epsilon_{\mathcal C})=\epsilon_{\nerve\mathcal C}$.

Explicitly, $\tilde\epsilon_{\mathcal C}$ is the functor sending an object $(A,S,m_\bullet,f)$ to the object corresponding to the image of $(\iota_S,{*})$ under $f$, and a morphism $\alpha\colon (A,S,m_\bullet,f)\to (B,T,n_\bullet,g)$ to the morphism corresponding to the image under $g$ of the unique edge $\alpha(\iota_S,{*})\to(\iota_T,{*})$ of $E\Inj(T,\omega)\times\prod_{b\in B}\Delta^{n_b}$.
\end{rk}

So far we have only considered $\textup{C}_\bullet$ as a functor $\cat{$\bmEM$-SSet}\to\cat{$\bmEM$-Cat}^\tau$. However, we can formally lift this to $\cat{$\bmEM$-$\bm G$-SSet}\to\cat{$\bmEM$-$\bm G$-Cat}^\tau$ by pulling through the $G$-action via functoriality. Explicitly, if $X$ is an $E\mathcal M$-$G$-simplicial set, then $g\in G$ acts on $(A,S,m_\bullet,f)$ via $g.(A,S,m_\bullet,f)=(A,S,m_\bullet,(g.\blank)\circ f)$, and if $\alpha\colon (A,S,m_\bullet,f)\to (A',S',m'_\bullet,f')$, then $g.\alpha$ is the same morphism of $E\mathcal M$-simplicial sets, but this time considered as a map $(A,S,m_\bullet,(g.\blank)\circ f)\to(A',S',m'_\bullet,(g.\blank)\circ f')$. It follows formally that $\epsilon$ and $\tilde\epsilon$ are $G$-equivariant, and that they define natural transformations of endofunctors of $\cat{$\bmEM$-$\bm G$-SSet}$ and $\cat{$\bmEM$-$\bm G$-Cat}$, respectively.

\section{An unstable comparison}\label{sec:unstable}
In this section we will prove the following predecessor to Theorem~\ref{thm:par-sum-cat-vs-sset}:

\begin{thm}\label{thm:nerve-vs-c-bullet-unstable}
The nerve and the functor $\textup{C}_\bullet$ from Lemma~\ref{lemma:c-bullet-functor} restrict to mutually inverse homotopy equivalences
\begin{equation*}
\textup{C}_\bullet\colon\cat{$\bmEM$-$\bm{G}$-SSet}^\tau\rightleftarrows\cat{$\bmEM$-$\bm{G}$-Cat}^\tau \colon\nerve
\end{equation*}
with respect to the $G$-global weak equivalences on both sides. More precisely, the natural transformations $\epsilon$ from Construction~\ref{constr:epsilon} and $\tilde\epsilon$ from Remark~\ref{rk:tilde-epsilon} restrict to natural levelwise $G$-global weak equivalences between the composites $\nerve\circ\textup{C}_\bullet$ and $\textup{C}_\bullet\circ\nerve$ and the respective identities.
\end{thm}

The proof will be given later in this section after some preparations.

\begin{constr}\label{constr:pre-post-actions}
Let $H\subset\mathcal M$ be any subgroup, let $A,S\subset\omega$ be $H$-subsets, and let $m_\bullet\colon A\to\mathbb N$ be constant on $H$-orbits (i.e.~$m_{h.a}=m_a$ for all $a\in A$ and $h\in H$). Then we have an $H$-action on $E\Inj(S,\omega)\times\prod_{a\in A}\Delta^{m_a}$ given by $h.\blank=(h^*\times h^*)^{-1}$, i.e.~$H$ acts by the diagonal of the $H$-action on $S$ and the `shuffling' action on $\prod_{a\in A}\Delta^{m_a}$; we call this the \emph{preaction} as it is essentially given by precomposition.

The $E\mathcal M$-action on $E\Inj(S,\omega)\times\prod_{a\in A}\Delta^{m_a}$ commutes with the preaction, and in particular restricting it to $H$ gives another $H$-action commuting with the preaction, and that we denote by `$*$' instead of the usual `.' in order to avoid confusion. We will refer to this action as the \emph{postaction} as it is given by postcomposition. The pre- and postaction together make $E\Inj(S,\omega)\times\prod_{a\in A}\Delta^{m_a}$ into an $(H\times H)$-simplicial set.
\end{constr}

Let now $X$ be an $E\mathcal M$-$G$-simplicial set and $\varphi\colon H\to G$ a homomorphism. For reasons that will become apparent below, we also refer to the $H$-action on $X$ given by restricting the $G$-action along $\varphi$ as \emph{preaction}. The $E\mathcal M$-action on $X$ then gives rise to another $H$-action (again denoted by $*$ and again called the \emph{postaction}) commuting with the preaction, making it into another $(H\times H)$-simplicial set.

\begin{lemma}\label{lemma:fixed-point-translation}
Let $H\subset\mathcal M$ be a subgroup, let $\varphi\colon H\to G$ be a homomorphism, let $X\in\cat{$\bmEM$-$\bm G$-SSet}$, and let $(A,S,m_\bullet,f)\in \textup C_X$. Then the following are equivalent\/{\rm:}
\begin{enumerate}
    \item $(A,S,m_\bullet,f)$ is $\varphi$-fixed.
    \item $A$ and $S$ are $H$-subsets of $\omega$, $m_\bullet$ is constant on $H$-orbits, and the map $f\colon E\Inj(S,\omega)\times \prod_{a\in A}\Delta^{m_a}\to X$ is $H$-equivariant with respect to the \emph{preactions}.
    \item $A$ and $S$ are $H$-subsets of $\omega$, $m_\bullet$ is constant on $H$-orbits, and $f$ is $H$-equivariant with respect to the \emph{diagonal} of the preaction and the postaction.
\end{enumerate}
Moreover, if $\alpha\colon (A,S,m_\bullet,f)\to (A',S',m'_\bullet,f')$ is a map to another $\varphi$-fixed object, then the following are equivalent:
\begin{enumerate}
    \item $\alpha$ is $\varphi$-fixed.
    \item $\alpha\colon E\Inj(S,\omega)\times\prod_{a\in A}\Delta^{m_a}\to E\Inj(S',\omega)\times\prod_{a\in A'}\Delta^{m'_a}$ is $H$-equivariant with respect to the \emph{preactions}.
    \item $\alpha$ is equivariant with respect to the \emph{diagonal} $H$-actions.
\end{enumerate}
\begin{proof}
For the statement about objects we first observe that $f$ is by construction of $\textup C_X$ always $E\mathcal M$-equivariant, hence $H$-equivariant with respect to the postactions. Thus, $(2)\Leftrightarrow(3)$ follows immediately. For $(1)\Leftrightarrow(2)$ we compute for any $h\in H$
\begin{equation*}
(h,\varphi(h)).(A,S,m_\bullet,f)=(h(A), h(S), m_{h^{-1}(\bullet)}, (\varphi(h).\blank)\circ f\circ (h^*\times h^*)).
\end{equation*}
In particular, the first three components are fixed if and only if $A,S\subset\omega$ are $H$-subsets, and $m_\bullet$ is constant on $H$-orbits. On the other hand, the condition $(\varphi(h).\blank)\circ f\circ (h^*\times h^*)=f$ for all $h\in H$ by construction precisely means that $f$ is an $H$-equivariant map with respect to the preactions.

For the characterization of $\varphi$-fixed morphisms, we note that $(h,\varphi(h)).\alpha$ is again a map $(A,S,m_\bullet,f)\to (A',S',m_\bullet',f')$ for all $h\in H$, so $\alpha=(h,\varphi(h)).\alpha$ as morphisms of $\textup{C}_X$ if and only if both sides agree in $\cat{$\bmEM$-SSet}$. But acting with $\varphi(h)$ does not affect $\alpha$ as a morphism of $\cat{$\bmEM$-SSet}$, and the action of $h$ is given by conjugating with $(h^*\times h^*)$. Thus, $\alpha$ is a $\varphi$-fixed point if and only if it is $H$-equivariant with respect to the preactions constructed above. As before this is equivalent to $\alpha$ being equivariant with respect to the diagonal action as $\alpha$ is $E\mathcal M$-equivariant.
\end{proof}
\end{lemma}

The following technical lemma provides the necessary equivariant information about the above $H$-objects:

\begin{lemma}\label{lemma:fixed-point-lemma}
Let $S\subset\omega$ be finite, let $Y\in\cat{SSet}$ be isomorphic to the nerve of a poset $P$ with a maximum element, and let ${*}$ denote the corresponding vertex of $Y$. Write $X\mathrel{:=}E\Inj(S,\omega)\times Y$ and let $H$ be any group.
\begin{enumerate}
\item Any $H$-action on $X$ restricts to an $H$-action on $E\Inj(S,\omega)\times\{{*}\}$.
\item If the $H$-action on $X$ is through $E\mathcal M$-equivariant maps, then its restriction to $E\Inj(S,\omega)\times\{{*}\}$ is induced by a unique $H$-action on $S$.\label{item:fpl-S-action}
\item Assume again that $H$ acts on $X$ through $E\mathcal M$-equivariant maps, so that $X$ is an $(E\mathcal M\times H)$-simplicial set, but assume moreover that $H$ is a subgroup of $\mathcal M$. Let $\Delta$ be the diagonal subgroup of $\mathcal M\times H$, let $T$ be any $H$-subset of $\omega$, and consider $E\Inj(S,T)\times Y$ with the restriction of the $\Delta$-action on $E\Inj(S,\omega)\times Y$.

Then $\big(E\Inj(S,T)\times Y\big)^\Delta$ is contractible provided that there exists an $H$-equivariant injection $S\to T$ with respect to the $H$-action on $S$ from $(\ref{item:fpl-S-action})$.
\end{enumerate}
\end{lemma}
\begin{proof}
For the first statement we observe that $X$ is canonically identified with the nerve of $\mathcal C\mathrel{:=}E\Inj(S,\omega)\times P$, and as $\nerve$ is fully faithful, it suffices to prove the analogous statement for $\mathcal C$. For this it is then enough to observe that any isomorphism of categories preserves the full subcategory spanned by the terminal objects, and that this is precisely given by $E\Inj(S,\omega)\times\{{*}\}$ in our case.

For the second statement we observe that evaluation at $\iota_S$ provides a bijection $\Hom_{E\mathcal M}(E\Inj(S,\omega),E\Inj(S,\omega))\to (E\Inj(S,\omega)_{[S]})_0=\Inj(S,\omega)_{[S]}$ by Proposition~\ref{prop:corep}. On the other hand, we have a map
\begin{equation}\label{eq:permutation-action-universal}
\Sigma_S\to\Hom_{E\mathcal M}(E\Inj(S,\omega),E\Inj(S,\omega))
\end{equation}
from the symmetric group on the set $S$ sending a permutation $\sigma$ to the map given by precomposition with $\sigma^{-1}$. The composition $\Sigma_S\to\Inj(S,\omega)_{[S]}$ is then given by $\sigma\mapsto\iota_S\sigma^{-1}$, which is obviously bijective. We conclude that also $(\ref{eq:permutation-action-universal})$ is bijective. In particular, there exists for each $h\in H$ a unique $\sigma(h)$ such that $h.\blank\colon E\Inj(S,\omega)\to E\Inj(S,\omega)$ agrees with $\blank\circ\sigma(h)^{-1}$. It only remains to show that this defines an action on $S$, i.e.~that $\sigma$ is a group homomorphism, for which it is enough to observe that $(\ref{eq:permutation-action-universal})$ is a monoid homomorphism.

For the final statement, we again switch to the categorical perspective. As the nerve is continuous, it then suffices to show that $\mathcal C^\Delta$ is contractible, for which it is enough that it has a terminal object. For this we observe that there is at most one map $x\to y$ for any $x,y\in\mathcal C= E\Inj(S,\omega)\times P$. Thus, a morphism in $\mathcal C$ is fixed by $\Delta$ if and only if its two endpoints are, i.e.~$\mathcal C^\Delta$ is a full subcategory of $\mathcal C$. It is therefore enough to show that one of the terminal objects of $\mathcal C$ is fixed by $\Delta$. But by the previous steps, the $\Delta$-action restricts to a $\Delta$-action on $E\Inj(S,T)\times\{{*}\}$, where $(h,h)$ for $h\in H$ acts by $(f,*)\mapsto (h\circ f\circ(h^{-1}.\blank),{*})$. Obviously, a terminal object $(f,{*})$ is fixed under this action if and only if the injection $f$ is $H$-equivariant, and such an $f$ exists by assumption.
\end{proof}

\begin{proof}[Proof of Theorem~\ref{thm:nerve-vs-c-bullet-unstable}]
We will show that $\epsilon_X$ is a $G$-global weak equivalence for each tame $E\mathcal M$-$G$-simplicial set $X$. If $\mathcal C$ is any tame $E\mathcal M$-$G$-category, then applying this to $\nerve(\mathcal C)$ will also show that $\nerve(\tilde\epsilon_{\mathcal C})=\epsilon_{\nerve\mathcal C}$ is a $G$-global weak equivalence, and as the nerve creates the $G$-global weak equivalences in $\cat{$\bmEM$-$\bm G$-Cat}^\tau$, this will then imply that also $\tilde\epsilon$ is a levelwise $G$-global weak equivalence. Moreover, we can conclude from this by $2$-out-of-$3$ that $\nerve\circ\textup{C}_\bullet$ is homotopical, and hence so is $\textup{C}_\bullet$, which then altogether implies the theorem.

Therefore let us fix a tame $E\mathcal M$-$G$-simplicial set $X$, a universal subgroup $H\subset\mathcal M$, and a homomorphism $\varphi\colon H\to G$. We claim that the restriction of $\epsilon$ to $\nerve(\textup{C}_X)_{[T]}\to X_{[T]}$ induces a weak equivalence on $\varphi$-fixed points for each finite $H$-subset $T\subset\omega$ containing an $H$-fixed point. For varying $T$, these exhaust $\nerve\textup{C}_X$ and $X$, as both are tame by assumption and since any finite set $T'\subset\omega$ is contained in a finite $H$-subset containing an $H$-fixed point (the latter uses that $\omega^H\not=\varnothing$ by universality). Now $(\blank)^\varphi=(\blank)^{\Gamma_{H,\varphi}}$ is given by a finite limit in $\cat{SSet}$ (as $\Gamma_{H,\varphi}\cong H$ is a finite group), and in $\cat{SSet}$ filtered colimits commute with finite limits. Thus, the claim implies the theorem by passing to the filtered colimit of the weak equivalences $\nerve(\textup{C}_X)_{[T]}^\varphi\to X_{[T]}^\varphi$ over all such $T$ as filtered colimits in $\cat{SSet}$ are homotopical.

To prove the claim, we fix $t\in T^H$ and we consider the functor $i\colon\Delta\downarrow X_{[T]}^\varphi\to(\textup{C}_X)_{[T]}$ sending an object $k\colon\Delta^n\to X$ to $(\{t\},T,n,\tilde k)$, where $\tilde k$ is the unique $E\mathcal M$-equivariant map $E\Inj(T,\omega)\times\Delta^n\to X$ with $\tilde k(\iota_T,\blank)=k$, and a morphism $\alpha\colon k\to\ell$ to $E\Inj(T,\omega)\times\alpha$. We omit the easy verification that $i$ is well-defined.

We now claim that $i$ actually lands in the $\varphi$-fixed points. Let us first check this on objects: if $k\colon\Delta^n\to X$ is any object of $\Delta\downarrow X_{[T]}^\varphi$, then we have to show that $(\{t\},T,n,\tilde k)$ is $\varphi$-fixed. But indeed, $T\subset\omega$ is an $H$-subset by assumption, $\{t\}\subset\omega$ is an $H$-subset as $t\in T^H$, and any family on $\{t\}$ is constant on orbits for trivial reasons, so it only remains to show by Lemma~\ref{lemma:fixed-point-translation} that $\tilde k$ is equivariant with respect to the preactions. But since for any $h\in H$ both $\tilde k\circ (h.\blank)$ and $(h.\blank)\circ\tilde k$ are $E\mathcal M$-equivariant, it suffices to show that they agree on $\{\iota_T\}\times\Delta^n$, for which we let $\sigma$ denote any simplex of $\Delta^n$. Then
\begin{equation*}
\tilde k(h.\iota_T,\sigma)=\tilde k(\iota_T\circ h^{-1}|_T,\sigma)=\tilde k(h^{-1}*\iota_T,\sigma)=h^{-1}*(\tilde k(\iota_T,\sigma))=\varphi(h).(\tilde k(\iota_T,\sigma))
\end{equation*}
as desired, where the last equation uses that $\tilde k(\iota_T,\sigma)=k(\sigma)$ is $\varphi$-fixed.

Now let $\ell\colon\Delta^{n'}\to X_{[T]}^\varphi$ be another object of $\Delta\downarrow X_{[T]}^\varphi$ and let $\alpha\colon k\to\ell$ be any morphism. Then $E\Inj(T,\omega)\times\alpha\colon E\Inj(T,\omega)\times\Delta^n\to E\Inj(T,\omega)\times\Delta^{n'}$ is obviously equivariant in the preactions, hence $\varphi$-fixed by Lemma~\ref{lemma:fixed-point-translation}. This completes the proof that $i$ lands in $(\textup{C}_X)_{[T]}^\varphi$.

Next, we consider the composite
\begin{equation}\label{eq:composite-last-vertex}
\nerve(\Delta\downarrow X_{[T]}^\varphi)\xrightarrow{\nerve(i)}\nerve\big((\textup{C}_X)_{[T]}^\varphi\big)\cong(\nerve\textup{C}_X)_{[T]}^\varphi\xrightarrow{\epsilon_{[T]}^\varphi} X_{[T]}^\varphi,
\end{equation}
where the unlabelled isomorphism comes from Example~\ref{ex:nerve-supported-on} together with the fact that $\nerve$ is a right adjoint.

If $g_0\xrightarrow{\alpha_1} g_1\to\cdots\xrightarrow{\alpha_k}g_k$ is a $k$-simplex of the left hand side (where each $g_\ell$ is a map $\Delta^{n_\ell}\to X_{[T]}^\varphi)$, then the above composite sends this to the image of $\sigma$ under $\tilde g_k$, where $\sigma$ is the unique $\ell$-simplex of $E\Inj(T,\omega)\times \Delta^{m_k}$ whose $\ell$-th vertex is
\begin{align*}
i(\alpha_k)\cdots i(\alpha_{\ell+1})(\iota_T,{*})&=(E\Inj(T,\omega)\times\alpha_k)\cdots(E\Inj(T,\omega)\times\alpha_{\ell+1})(\iota_T,{*})\\
&=(\iota_T,\alpha_{k}\cdots\alpha_{\ell+1}({*}))=(\iota_T,\alpha_{k}\cdots\alpha_{\ell+1}(\Delta^{\{n_\ell\}}))
\end{align*}
Thus, if $\tau$ is the unique $k$-simplex of $\Delta^{n_k}$ with $\ell$-th vertex $\alpha_{k}\cdots\alpha_{\ell+1}(\Delta^{\{n_\ell\}})$, then $\sigma=(\iota_T,\tau)$, and hence $\tilde g_k(\sigma)=\tilde g_k(\iota_T,\tau)=g_k(\tau)=\epsilon(\alpha_\bullet)$. We therefore conclude that the composite $(\ref{eq:composite-last-vertex})$ agrees with the last vertex map $\nerve(\Delta\downarrow X_{[T]}^\varphi)\to X_{[T]}^\varphi$, so that it is a weak homotopy equivalence by Proposition~\ref{prop:last-vertex}.

It is therefore enough to show that $i$ is a weak homotopy equivalence. By Quillen's Theorem A \cite[§1]{quillen-theorem-A}, it suffices for this that the slice $i\downarrow(A,S,m_\bullet,f)$ has weakly contractible nerve for each $(A,S,m_\bullet,f)\in(\textup{C}_X)_{[T]}^\varphi$.

So let $(A,S,m_\bullet,f)$ be any $\varphi$-fixed point supported on $T$. Then \begin{equation*}K\mathrel{:=}E\Inj(S,T)\times\prod_{a\in A}\Delta^{m_a}\end{equation*} is canonically identified with the subcomplex of $E\Inj(S,\omega)\times\prod_{a\in A}\Delta^{m_a}$ consisting of the simplices supported on $T$, and from this it inherits the two commuting $H$-actions considered before: the postaction given by restriction of the $E\mathcal M$-action on $E\Inj(S,\omega)$ (i.e.~induced by the $H$-action on $T$) and the preaction given by the $H$-actions on $A$ and $S$. We will be interested in the fixed points $K^\Delta$ for the diagonal of these two actions. Namely, let us define a functor $c\colon i\downarrow(A,S,m_\bullet,f)\to\Delta\downarrow K^\Delta$ as follows: an object of the left hand side consists by definition of a map $g\colon\Delta^n\to X_{[T]}^\varphi$ together with a $\varphi$-fixed morphism $\alpha\colon i(g)\to (A,S,m_\bullet,f)$, i.e.~an $(E\mathcal M\times H)$-equivariant map $\alpha\colon E\Inj(T,\omega)\times\Delta^n\to E\Inj(S,\omega)\times\prod_{a\in A}\Delta^{m_a}$ such that $\tilde g=f\alpha$. We now claim that the composition
\begin{equation*}
\Delta^n\xrightarrow{(\iota_T,\blank)}E\Inj(T,\omega)\times\Delta^n\xrightarrow{\alpha} E\Inj(S,\omega)\times\prod_{a\in A}\Delta^{m_a}
\end{equation*}
actually lands in $K^\Delta$. Indeed, it is clear that it lands in $K$, so we only have to show that $\alpha(\iota_T,\sigma)$ is $\Delta$-fixed for each simplex $\sigma$ of $\Delta^n$. But indeed, as $\alpha$ is $\Delta$-equivariant, it suffices that $\iota_T$ is a $\Delta$-fixed point of $E\Inj(T,\omega)$, which is immediate from the definition. With this established, we now define $c(g,\alpha)$ as $\alpha(\iota_T,\blank)$ considered as a map $\Delta^n\to K^\Delta$.

If \begin{equation*}(g'\colon\Delta^{n'}\to X_{[T]}^\varphi,\alpha'\colon i(g')\to (A,S,m_\bullet,f))\end{equation*} is another object of $i\downarrow(A,S,m_\bullet,f)$, then a morphism $(g,\alpha)\to (g',\alpha')$ is given by a map $\mathfrak a\colon\Delta^n\to\Delta^{n'}$ such that $g=g'\circ\mathfrak a$ (i.e.~$\mathfrak a$ is a map $g\to g'$ in $\Delta\downarrow X_{[T]}^\varphi$) and $\alpha=i(\mathfrak a)\circ\alpha'$. As $i(\mathfrak a)=E\Inj(T,\omega)\times\mathfrak a$, restricting to $\{\iota_T\}\times\Delta^n$ shows that $\alpha(\iota_T,\blank)=\alpha'(\iota_T,\blank)\circ\mathfrak a$, i.e.~$\mathfrak a$ also defines a morphism $c(g,\alpha)\to c(g',\alpha')$ in $\Delta\downarrow K^\Delta$, which we take as the definition of $c(\mathfrak a)$. It is clear that $c$ is a functor.

\begin{claim*}
$c$ is an equivalence of categories $i\downarrow(A,S,m_\bullet,f)\simeq\Delta\downarrow K^\Delta$.

\begin{proof}
We will show that $c$ is fully faithful and surjective on objects; in fact, it is not hard to show that $c$ is also injective on objects, hence an isomorphism of categories, but we will not need this.

It is clear from the definition that $c$ is faithful. To see that it is full we let $(g\colon\Delta^n\to X_{[T]}^\varphi,\alpha), (g'\colon\Delta^{n'}\to X_{[T]}^\varphi,\alpha')$ be objects of the left hand side, and we let $\mathfrak a\colon\Delta^n\to\Delta^{n'}$ be a morphism $c(g,\alpha)\to c(g',\alpha')$, i.e.~
\begin{equation}\label{eq:morphism-right-slice}
\alpha(\iota_T,\blank)=\alpha'(\iota_T,\blank)\circ\mathfrak a.
\end{equation}
We want to show that $\mathfrak a$ also defines a morphism $(g,\alpha)\to (g',\alpha')$, i.e.~that the two triangles
\begin{equation*}
\begin{tikzcd}[column sep=tiny, cramped]
\Delta^{n}\arrow[rr, "\mathfrak a"]\arrow[rd, bend right=10pt, "g"'] && \Delta^{n'}\arrow[ld, "g'", bend left=10pt]\\
& X_{[T]}^\varphi\vphantom{\prod\limits_{a\in A}}
\end{tikzcd}
\qquad\hskip-.2em\text{and}\hskip.2em\quad
\begin{tikzcd}[column sep=-10pt,cramped]
E\Inj(T,\omega)\times\Delta^n\arrow[rr, "{E\Inj(T,\omega)\times\mathfrak a}"]\arrow[rd, "\alpha"', bend right=10pt] && E\Inj(T,\omega)\times\Delta^{n'}\arrow[ld, "\alpha'", bend left=10pt]\\
&E\Inj(S,\omega)\times\prod\limits_{a\in A}\Delta^{m_a}
\end{tikzcd}
\end{equation*}
commute. For the second one we observe that both paths through the diagram are $E\mathcal M$-equivariant, so that it suffices to show this after restricting to $\{\iota_T\}\times\Delta^n$, where this is precisely the identity $(\ref{eq:morphism-right-slice})$. On the other hand, the commutativity of the first diagram follows once we observe that $\tilde g=f\circ\alpha$ as $\alpha$ is a morphism $i(g)\to(A,S,m_\bullet,f)$, hence $g=f\circ\alpha(\iota_T,\blank)$ and analogously $g'=f\circ\alpha'(\iota_T,\blank)$. Thus, $\mathfrak a$ also defines a morphism $(g,\alpha)\to(g',\alpha')$ which is then obviously the desired preimage.

Finally, let us show that $c$ is surjective on objects. We let $\hat\alpha\colon\Delta^n\to K^\Delta$ be any map; then the composition
\begin{equation*}
\Delta^n\xrightarrow{\hat\alpha}K^\Delta=\left(E\Inj(S,T)\times\prod_{a\in A}\Delta^{m_a}\right)^\Delta\hookrightarrow E\Inj(S,\omega)\times\prod_{a\in A}\Delta^{m_a}
\end{equation*}
by construction lands in the subcomplex of those simplices that are supported on $T$, so it extends to a unique $E\mathcal M$-equivariant map $\alpha\colon E\Inj(T,\omega)\times\Delta^n\to E\Inj(S,\omega)\times\prod_{a\in A}\Delta^{m_a}$.

We claim that $(f\hat\alpha,\alpha)$ defines an element of $i\downarrow (A,S,m_\bullet,f)$, which amounts to saying that $f\hat\alpha\colon\Delta^n\to X$ factors through $X_{[T]}^\varphi$, that $\alpha$ is $H$-equivariant with respect to the preactions, and that the diagram
\begin{equation*}
\begin{tikzcd}[column sep=tiny]
E\Inj(T,\omega)\times\Delta^n\arrow[rr, "\alpha"]\arrow[rd, "\widetilde{(f\hat\alpha)}"', bend right=10pt] && E\Inj(S,\omega)\times\smash{\prod\limits_{a\in A}\Delta^{m_a}}\arrow[ld, bend left=10pt, "f"]\\
&X
\end{tikzcd}
\end{equation*}
commutes. For the first statement, we observe that $f\hat\alpha$ lands in $X_{[T]}$ as $\hat\alpha$ lands in $(E\Inj(S,\omega)\times\prod_{a\in A}\Delta^{m_a})_{[T]}$ and because $f$ is $E\mathcal M$-equivariant. To see that it also lands in the $\varphi$-fixed points, it suffices to observe that $f$ restricts to \begin{equation*}(E\Inj(S,\omega)\times\prod_{a\in A}\Delta^{m_a})^\Delta\to X^\varphi\end{equation*} by Lemma~\ref{lemma:fixed-point-translation}.

For the second statement it is again enough that $h.(\hat\alpha(\sigma))=\alpha(h.\iota_T,\sigma)$ for all simplices $\sigma$ of $\Delta^n$ and all $h\in H$. But indeed, $h.\iota_T=h^{-1}*\iota_T$ as before, so $\alpha(h.\iota_T,\sigma)=h^{-1}*(\hat\alpha(\sigma))$, and this in turn agrees with $h.(\hat\alpha(\sigma))$ because $\hat\alpha(\sigma)$ is $\Delta$-fixed.

Finally, for the third statement it suffices again to check this on $\{\iota_T\}\times\Delta^n$, where it holds tautologically. Altogether we have shown that $(f\hat\alpha,\alpha)$ defines an element of $i\downarrow(A,S,m_\bullet,f)$. It is then immediate from the definition that $c(f\hat\alpha,\alpha)=\hat\alpha$, which completes the proof of the claim.
\end{proof}
\end{claim*}

We conclude that in particular $\nerve(i\downarrow(A,S,m_\bullet,f))\simeq\nerve(\Delta\downarrow K^\Delta)$. By Proposition~\ref{prop:last-vertex} we further see that $\nerve(\Delta\downarrow K^\Delta)$ is weakly equivalent to $K^\Delta$, so it only remains to prove that the latter is (weakly) contractible. This is a direct application of Lemma~\ref{lemma:fixed-point-lemma}: the restriction of the preaction on $E\Inj(S,\omega)\times\prod_{a\in A}\Delta^{m_a}$ to $E\Inj(S,\omega)\times\{{*}\}$ is by construction induced by the $H$-action on $S$ coming from the $H$-action on $\omega$. On the other hand, $S$ is a subset of $\supp(A,S,m_\bullet,f)\subset T$ by  Lemma~\ref{lemma:C-bullet-supp}, so the inclusion $S\hookrightarrow T$ is the desired $H$-equivariant injection. Altogether, this completes the proof of the theorem.
\end{proof}

So far we have only used the third part of Lemma~\ref{lemma:fixed-point-lemma}. However, we will now need its full strength for the proof of the following result:

\begin{prop}\label{prop:c-bullet-ws}
Let $X$ be a tame $E\mathcal M$-$G$-simplicial set. Then $\textup{C}_X$ is weakly saturated.
\end{prop}

\begin{rk}
Since the above proposition will have non-trivial consequences later, let us give some intuition why one should expect this to be true: as the inclusion $\mathcal C^\varphi\hookrightarrow\mathcal C^{\myh\varphi}$ is fully faithful for any $E\mathcal M$-$G$-category $\mathcal C$, we can heuristically interpret the failure of $\mathcal C$ to be saturated either as $\mathcal C$ not having enough fixed points or as $\mathcal C$ having too many (categorical) homotopy fixed points. While our previous saturation construction solves this issue by potentially introducing additional fixed points, we will see below that the categories of the form $\textup{C}_X$ instead have very few homotopy fixed points. This is in turn to be expected as $\textup{C}_X$ contains only few non-trivial isomorphisms; more specifically, all automorphisms of objects in $\textup{C}_X$ come from automorphisms of the $E\mathcal M$-equivariant simplicial sets $E\Inj(S,\omega)\times\prod_{a\in A}\Delta^{m_a}$, over which we have good control by Lemma~\ref{lemma:fixed-point-lemma}. In particular, if $X=\nerve\mathcal C$, then the isomorphisms in $\textup{C}_{\nerve\mathcal C}$ are detached from the isomorphisms in $\mathcal C$.

(In fact, $\textup{C}_X$ has so few isomorphisms, that usually not every homotopy fixed point will be \emph{isomorphic} to an honest fixed point, i.e.~$\textup{C}_X^\varphi\hookrightarrow\textup{C}_X^{\myh\varphi}$ won't be an equivalence, see Remark~\ref{rk:CX-not-saturated}. What we will see below, however, is that for each homotopy fixed point the space of maps from honest fixed points to it is contractible, which is enough to imply that the inclusion is a weak homotopy equivalence.)
\end{rk}

In order to turn the above intuition into a rigorous proof, we need to understand the categories $\textup{C}_X^{\myh\varphi}$ better:

\begin{constr}
Let us first assume that $X={*}$, and let $\Phi\in\textup{C}_*^{\myh\varphi}$ arbitrary. As there are no non-trivial actions on $*$, this means that $\Phi\colon EH\to\textup{C}_*$ is $H$-equivariant with respect to the restriction of the $\mathcal M$-action on $\textup{C}_*$ to $H$.

Let us write $(A,S,m_\bullet,*)\mathrel{:=}\Phi(1)$ (where ${*}$ will always denote the unique map from an implicitly understood object to the fixed terminal simplicial set ${*}$). Then we have for each $h\in H$ an $E\mathcal M$-equivariant self-map $h.\blank$ of $E\Inj(S,\omega)\times\prod_{a\in A}\Delta^{m_a}$ given by the composition
\begin{equation*}
\Phi(1)\xrightarrow{h_\circ=(h^*\times h^*)^{-1}} h.\Phi(1)=\Phi(h)\xrightarrow{\Phi(1,h)}\Phi(1);
\end{equation*}
it is not hard to check that this defines an $H$-action, also cf.~\cite[Construction~7.4]{schwede-k-theory}. In analogy with Construction~\ref{constr:pre-post-actions} we call this the \emph{preaction} induced by $\Phi$.
\end{constr}

\begin{rk}\label{rk:hfp-preaction}
More generally, if $\Phi\in \textup{C}_X^{\myh\varphi}$, $(A,S,m_\bullet,f)\mathrel{:=}\Phi(1)$, then we get an \emph{$H$-preaction} on $E\Inj(S,\omega)\times\prod_{a\in A}\Delta^{m_a}$ by pushing $\Phi$ forward along the unique map $X\to {*}$. In explicit terms, this action is given by the composites
\begin{align*}
h.\blank\colon E\Inj(S,\omega)\times\prod_{a\in A}\Delta^{m_a}&\xrightarrow{(h^*\times h^*)^{-1}} E\Inj(h(S),\omega)\times\prod_{b\in h(A)} \Delta^{m_{h^{-1}(b)}}\\
&\xrightarrow{\Phi(1,h)} E\Inj(S,\omega)\times\prod_{a\in A}\Delta^{m_a}.
\end{align*}
Beware however, that this will typically not be an $H$-action on $\Phi(1)\in \textup{C}_X$; instead, the identities $f\circ\Phi(1,h)=(\varphi(h).\blank)\circ f\circ (h^*\times h^*)$ coming from the requirement that $\Phi(1,h)$ be a map $(h,\varphi(h)).\Phi(1)=\Phi(h)\to\Phi(1)$ translate to the condition that $f\colon E\Inj(S,\omega)\times\prod_{a\in A}\Delta^{m_a}\to\varphi^*X$ is $H$-equivariant with respect to the above $H$-action on the source, i.e.~$f$ is equivariant in the preactions. As before, this is equivalent for the $E\mathcal M$-equivariant map $f$ to being equivariant with respect to the diagonal actions.
\end{rk}

\begin{lemma}\label{lemma:homotopy-fixed-points}
Let $\Phi,\Psi\in\textup{C}_X^{\myh\varphi}$ arbitrary, and write $\Phi(1)\mathrel{=:}(A,S,m_\bullet,f)$, $\Psi(1)\mathrel{=:}(B,T,n_\bullet,g)$. Then the following are equivalent for a map $\alpha_1\colon\Phi(1)\to\Psi(1)$ in $\textup{C}_X$\/{\rm:}
\begin{enumerate}
    \item $\alpha_1$ extends to a map $\alpha\colon\Phi\to\Psi$ in $\textup{C}_X^{\myh\varphi}$.
    \item $\alpha_1\colon E\Inj(S,\omega)\times\prod_{a\in A}\Delta^{m_a}\to E\Inj(T,\omega)\times\prod_{b\in B}\Delta^{n_b}$ is $H$-equivariant with respect to the \emph{preactions} coming from $\Phi$ and $\Psi$.
    \item $\alpha_1$ is $H$-equivariant with respect to the diagonal of the preaction and the postaction (coming from the $E\mathcal M$-action).
\end{enumerate}
Moreover, we have a bijection
\begin{equation}\label{eq:hfp}
\begin{aligned}
    \Hom_{\textup{C}_X^{\myh\varphi}}(\Phi,\Psi)&\to\Hom_{\cat{$\bm H$-SSet}\downarrow X}(f,g)\\
    \alpha&\mapsto\alpha_1
\end{aligned}
\end{equation}
where we view $f\colon E\Inj(S,\omega)\times\prod_{a\in A}\Delta^{m_a}\to X$ and $g\colon E\Inj(T,\omega)\times\prod_{b\in B}\Delta^{n_b}\to X$ as maps in $\cat{$\bm H$-SSet}$ either by equipping all objects with the preactions or by equipping all objects with the diagonal $H$-action.
\begin{proof}
We begin by proving the equivalence of $(1)$--$(3)$.
As before, $(2)\Leftrightarrow(3)$ is clear. For the proof of $(1)\Rightarrow(2)$, we may assume that $X={*}$, and we let $\alpha\colon\Phi\to\Psi$ be any map in $\textup{C}_*^{\myh\varphi}$ extending $\alpha_1$. Then in the diagram
\begin{equation}\label{diag:h-phi-morphism-equivariant}
\begin{tikzcd}
\Phi(1)\arrow[d, "\alpha_1"']\arrow[r, "h_\circ"] & h.\Phi(1)\arrow[d, "h.\alpha_1"']\arrow[r, equal] & \Phi(h)\arrow[d, "\alpha_h"]\arrow[r, "{\Phi(1,h)}"] & \Phi(1)\arrow[d, "\alpha_1"]\\
\Psi(1)\arrow[r, "h_\circ"'] & h.\Psi(1)\arrow[r,equal] & \Psi(h)\arrow[r, "{\Psi(1,h)}"'] & \Psi(1)
\end{tikzcd}
\end{equation}
the left hand square commutes by naturality of $h_\circ$, the middle square commutes by equivariance of $\alpha$, and the right hand square commutes by naturality of $\alpha$. Thus, the total rectangle commutes, i.e.~$\alpha_1$ is equivariant in the preactions as claimed.

To prove $(2)\Rightarrow(1)$, we first consider the special case $X={*}$. We set $\alpha_h\mathrel{:=}h.\alpha_1=(h^*\times h^*)^{-1}\circ\alpha_1\circ(h^*\times h^*)$ and observe that the outer rectangle in $(\ref{diag:h-phi-morphism-equivariant})$ commutes by assumption on $\alpha_1$, and so do the left hand and middle square by the same arguments as above. As all horizontal morphisms are isomorphisms, we conclude that also the right hand square commutes, i.e.~$\alpha$ is compatible with the edges $(1,h)$ in $EH$. Since these generate $EH$ as a groupoid, we conclude that $\alpha$ is natural. As it is $H$-equivariant by construction, it is therefore a morphism $\Phi\to\Psi$ in $\textup{C}_*^{\myh\varphi}$.

In the case of a general $E\mathcal M$-$G$-simplicial set $X$, we apply the above to the pushforwards of $\Phi$ and $\Psi$ to $\textup{C}_*$; it then only remains to show that each $\alpha_h=h.\alpha_1$ is a map $\Phi(h)\to\Psi(h)$ in $\textup{C}_X$, i.e.~that \begin{equation*}\big((\varphi(h).\blank)\circ g\circ(h^*\times h^*)\big)\circ\alpha_h= (\varphi(h).\blank)\circ f\circ(h^*\times h^*).\end{equation*} This is however immediate from the explicit description of $\alpha_h$ and the fact that $\alpha_1$ is a map in $\textup{C}_X$.

With this at hand, we can easily show that $(\ref{eq:hfp})$ is bijective. Namely, we first observe that it is indeed well-defined by the implications $(1)\Rightarrow(2)$ and $(1)\Rightarrow(3)$, and that it is injective as any $\alpha\colon\Phi\to\Psi$ is determined by $\alpha_1$ by $H$-equivariance. Finally, surjectivity precisely amounts to the converse implications $(2)\Rightarrow(1)$ and $(3)\Rightarrow(1)$.
\end{proof}
\end{lemma}

\begin{proof}[Proof of Proposition~\ref{prop:c-bullet-ws}]
Let $H\subset\mathcal M$ be a universal subgroup and $\varphi\colon H\to G$ a group homomorphism. We have to show that the canonical map $c\colon\textup{C}_X^\varphi\to\textup{C}_X^{\myh\varphi}$ is a weak homotopy equivalence. To this end we consider for each (finite) $H$-subset $T\subset\omega$ the following full subcategory $(\textup{C}_X^{\myh\varphi})_{\langle T\rangle}\subset\textup{C}_X^{\myh\varphi}$: if $\Phi\in\textup{C}_X^{\myh\varphi}$, $(A,S,m_\bullet,f)\mathrel{:=}\Phi(1)$, then Remark~\ref{rk:hfp-preaction} describes a canonical $H$-action on $E\Inj(S,\omega)\times\prod_{a\in A}\Delta^{m_a}$, whose restriction to $E\Inj(S,\omega)\times\{*\}$ is induced by a unique $H$-action on $S$ according to Lemma~\ref{lemma:fixed-point-lemma}; we now declare that $\Phi$ should belong to $(\textup{C}_X^{\myh\varphi})_{\langle T\rangle}$ if and only if the $H$-set $S$ admits an $H$-equivariant injection into $T$.

If $(A,S,m_\bullet,f)$ is $\varphi$-fixed, then $A,S\subset\omega$ are $H$-subsets, and the above $H$-action on $E\Inj(S,\omega)\times\prod_{a\in A}\Delta^{m_a}$ is simply the preaction from Construction~\ref{constr:pre-post-actions}. In particular, its restriction to $E\Inj(S,\omega)\times\{{*}\}$ is induced by the tautological $H$-action on $S\subset\omega$. Thus, if $(A,S,m_\bullet,f)$ is supported on $T$, then the inclusion $S\hookrightarrow T$ is $H$-equivariant with respect to the above action, so that $c\colon\textup{C}_X^\varphi\to\textup{C}_X^{\myh\varphi}$ restricts to $(\textup{C}_X)_{[T]}^\varphi\to(\textup{C}_X^{\myh\varphi})_{\langle T\rangle}$.

Next, we observe that the $(\textup{C}_{X}^{\myh\varphi})_{\langle T\rangle}$ exhaust $\textup{C}_X^{\myh\varphi}$ when we let $T$ run through all finite $H$-subsets of $\omega$ with $T^H\not=\varnothing$: indeed, if $\Phi$ is arbitrary, $(A,S,m_\bullet,f)\mathrel{:=}\Phi(1)$, then we consider the finite $H$-set $S\amalg\{{*}\}$ where $H$ acts on $S$ as above and trivially on ${*}$. As $\omega$ is a complete $H$-set universe, there exists an $H$-equivariant injection $S\amalg\{*\}\rightarrowtail\omega$, whose image is then the desired $T$. Thus, the inclusions express $\textup{C}_X^{\myh\varphi}$ as a filtered colimit along the inclusions of the $(\textup{C}_X^{\myh\varphi})_{\langle T\rangle}$ over all finite $H$-subsets $T\subset\omega$ with $T^H\not=\varnothing$.

Altogether, we are reduced to showing that $(\textup{C}_X)_{[T]}^\varphi\to(\textup{C}_X^{\myh\varphi})_{\langle T\rangle}$ is a weak homotopy equivalence for all such $T$, for which it is enough by $2$-out-of-$3$ that the composition
\begin{equation*}
j\colon\Delta\downarrow X_{[T]}^\varphi\xrightarrow{i}(\textup{C}_X)_{[T]}^\varphi\xrightarrow{c}(\textup{C}_X^{\myh\varphi})_{\langle T\rangle}
\end{equation*}
is a weak homotopy equivalence, where $i$ is the weak homotopy equivalence  from the proof of Theorem~\ref{thm:nerve-vs-c-bullet-unstable}.

For this it is again enough by Quillen's Theorem~A that the slice $j\downarrow\Phi$ has weakly contractible nerve for each $\Phi\in(\textup{C}_X^{\myh\varphi})_{\langle T\rangle}$. To prove this, let $\Phi(1)\mathrel{=:}(A,S,m_\bullet,f)$ and define $K\mathrel{:=}E\Inj(S,T)\times\prod_{a\in A}\Delta^{m_a}$ with $H$-action via the $H$-action on $T$ and the restriction of the preaction on $E\Inj(S,\omega)\times\prod_{a\in A}\Delta^{m_a}$ induced by $\Phi$. Using Lemma~\ref{lemma:homotopy-fixed-points} one can show precisely as in the proof of Theorem~\ref{thm:nerve-vs-c-bullet-unstable} that we have an equivalence of categories $d\colon j\downarrow\Phi\to\Delta\downarrow K^\Delta$ sending $\alpha\colon j(g\colon\Delta^n\to X_{[T]}^\varphi)\to\Phi$ to $\alpha_1(\iota_T,\blank)$ and a map $(g,\alpha)\to(g',\alpha')$ given by $\mathfrak a\colon\Delta^n\to\Delta^{n'}$ to the map $d(g,\alpha)\to d(g',\alpha')$ given by the same $\mathfrak a$. In particular, we conclude together with Proposition~\ref{prop:last-vertex} that $\nerve(j\downarrow\Phi)\simeq\nerve(\Delta\downarrow K^\Delta)\simeq K^\Delta$. By definition of $(\textup{C}_X^{\myh\varphi})_{\langle T\rangle}$ there exists an $H$-equivariant injection $S\to T$ with respect to the $H$-action on $S$ induced by $\Phi$. Thus, Lemma~\ref{lemma:fixed-point-lemma} implies that $K^\Delta$ is contractible, which completes the proof of the proposition.
\end{proof}

\begin{rk}\label{rk:CX-not-saturated}
Let $X\not=\varnothing$ be a tame $E\mathcal M$-$G$-simplicial set. Then $\textup{C}_X$ is \emph{not} saturated:

Indeed, let $x\in X_0$ be arbitrary and write $S\mathrel{:=}\supp(x)$. Then there exists a (unique) $E\mathcal M$-equivariant map $\tilde x\colon E\Inj(S,\omega)\to X$ sending $\iota_S$ to $x$. We pick a finite set $T\subset\omega\smallsetminus S$ with at least two elements and we define $f$ as the composition
\begin{equation*}
E\Inj(S\cup T,\omega)\times\Delta^1\xrightarrow{\pr} E\Inj(S\cup T,\omega)\xrightarrow{\text{res}} E\Inj(S)\xrightarrow{\tilde x} X.
\end{equation*}
We moreover choose a universal subgroup $H$ of $\mathcal M$ together with an isomorphism $\psi\colon H\to\Sigma_T$, and we write $\varphi\colon H\to 1$ for the unique homomorphism. For any $h\in H$ we define its action on $E\Inj(S\cup T,\omega)\times\Delta^1$ as the unique self-map $\tau_h$ sending $(u,0)$ to $(u,0)$ and $(u,1)$ to $(u\circ(S\cup\psi(h^{-1})),1)$ for each $u\in\mathcal M$. We omit the easy verification that this is a well-defined $H$-action.

Let now $a\in\omega$ be any $H$-fixed point. It is then not hard to check that $\Phi\colon EH\to\textup{C}_X$ with $\Phi(h)=(\{a\},h(S\cup T),1,f\circ(h^*\times h^*))$ and structure maps $\Phi(h_2,h_1)=(h_2)_\circ \tau_{h_2^{-1}h_1} (h_1)_\circ^{-1}$ defines an element of $\textup{C}_X^{\myh\varphi}$. The induced $H$-action on $E\Inj(S\cup T,\omega)\times\Delta^1$ is then simply the one given above. By the description of the morphisms in $\textup{C}_X^{\myh\varphi}$ given in Lemma~\ref{lemma:homotopy-fixed-points} it is then enough to show that this is not $H$-equivariantly isomorphic to a simplicial set of the form $E\Inj(U,\omega)\times\prod_{b\in B}\Delta^{n_b}$ for some finite $H$-subsets $B,U\subset\omega$, with $H$ acting via its tautological actions on $B$ and $U$.

Indeed, if there were such an isomorphism $\alpha$, then it would restrict to $H$-equivariant isomorphisms $E\Inj(U,\omega)\times\prod_{b\in B}\Delta^{\{0\}}\cong E\Inj(S\cup T,\omega)\times\Delta^{\{0\}}$ and $E\Inj(U,\omega)\times\prod_{b\in B}\Delta^{\{n_b\}}\cong E\Inj(S\cup T,\omega)\times\Delta^{\{1\}}$. In particular, the two $H$-simplicial sets $E\Inj(S\cup T,\omega)\times\Delta^{\{0\}}$ and $E\Inj(S\cup T,\omega)\times\Delta^{\{1\}}$ would be $H$-equivariantly isomorphic. But this is obviously not the case as precisely one of them has trivial $H$-action, yielding the desired contradiction.
\end{rk}

With Proposition~\ref{prop:c-bullet-ws} at hand we can now prove:

\begin{cor}\label{cor:sat-ge}
All the inclusions in
\begin{equation*}
\cat{$\bmEM$-$\bm G$-Cat}^{\tau,s}\hookrightarrow\cat{$\bmEM$-$\bm G$-Cat}^{\tau,ws}\hookrightarrow\cat{$\bmEM$-$\bm G$-Cat}^\tau
\end{equation*}
are homotopy equivalences with respect to the $G$-global weak equivalences.
\end{cor}
\begin{proof}
We already know this for the left hand inclusion by Corollary~\ref{cor:ws-vs-s}, so it suffices to consider the right hand inclusion. We claim that $\textup{C}_\bullet\circ\nerve$ defines a homotopy inverse. Indeed, this lands in $\cat{$\bmEM$-$\bm G$-Cat}^{\tau,ws}$ by Proposition~\ref{prop:c-bullet-ws}; moreover, the natural map $\tilde\epsilon\colon\textup{C}_{\nerve\mathcal C}\to\mathcal C$ is a $G$-global weak equivalence for every tame $E\mathcal M$-$G$-category $\mathcal C$ by Theorem~\ref{thm:nerve-vs-c-bullet-unstable}. As $\cat{$\bmEM$-$\bm G$-Cat}^{\tau,ws}$ is a full subcategory of $\cat{$\bmEM$-$\bm G$-Cat}^\tau$, this immediately implies the claim.
\end{proof}

\section{Lifting the parsummable structure}\label{sec:comp-parsum}
In this section we will prove the parsummable analogues of Theorem~\ref{thm:nerve-vs-c-bullet-unstable} and Corollary~\ref{cor:sat-ge} by lifting $\textup{C}_\bullet$ to a functor $\cat{$\bm G$-ParSumSSet}\to\cat{$\bm G$-ParSumCat}$ and showing that the natural transformations $\epsilon,\tilde\epsilon$ are compatible with the resulting structure.

\begin{constr}
Let $A',S'$ be finite sets and let $A\subset A', S\subset S'$. Let moreover $(m_a)_{a\in A'}$ be a family of non-negative integers. Then we define
\begin{equation*}
\rho^{A',S'}_{A,S}\colon E\Inj(S',\omega)\times\prod_{a\in A'}\Delta^{m_a}\to E\Inj(S,\omega)\times\prod_{a\in A}\Delta^{m_a}
\end{equation*}
as the product of the restriction $E\Inj(S',\omega)\to E\Inj(S,\omega)$ and the projection $\prod_{a\in A'}\Delta^{m_a}\to\prod_{a\in A}\Delta^{m_a}$.
\end{constr}

\begin{lemma}\label{lemma:rho-basic-properties}
Throughout, let $m_\bullet$ be an appropriately indexed family of non-negative integers.
\begin{enumerate}
\item $\rho^{A',S'}_{A,S}$ is $E\mathcal M$-equivariant for all $A\subset A'$, $S\subset S'$.\label{item:rbp-em}
\item If $A\subset A'\subset A''$, $S\subset S'\subset S''$, then $\rho^{A',S'}_{A,S}\rho^{A'',S''}_{A',S'}=\rho^{A'',S''}_{A,S}$.\label{item:rbp-associative}
\item $\rho_{A,S}^{A,S}=\id$ for all $A,S$\label{item:rbp-unital}
\item If $A\cap B=\varnothing=S\cap T$, then $(\rho^{A\cup B,S\cup T}_{A,S}, \rho^{A\cup B,S\cup T}_{B,T})$ restricts to an isomorphism\label{item:rbp-isomorphism}
\begin{equation*}
E\Inj(S\cup T,\omega)\times\prod_{i\in A\cup B}\Delta^{m_i}
\cong
\left(E\Inj(S,\omega)\times\prod_{a\in A}\Delta^{m_a}\right)\boxtimes
\left(E\Inj(T,\omega)\times\prod_{b\in B}\Delta^{m_b}\right).
\end{equation*}
\item If $A',S'\subset\omega$, and $u\in\mathcal M$, then $\rho_{A,S}^{A',S'}\circ(u^*\times u^*)=(u^*\times u^*)\circ\rho_{u(A),u(S)}^{u(A'),u(S')}$ for all $A\subset A', S\subset S'$.\label{item:rbp-equivariant}
\end{enumerate}
\end{lemma}
\begin{proof}
The first three statements are obvious. For the fourth statement let us first show that $(\rho^{A\cup B,S\cup T}_{A,S}, \rho^{A\cup B,S\cup T}_{B,T})$ lands in the box product. Indeed, it sends an $n$-simplex $(u_0,\dots,u_n,(\sigma_{i})_{i\in A\cup B})$ to $\big((u_0|_S,\dots,u_n|_S;\sigma|_A),(u_0|_T,\dots,u_n|_T;\sigma|_B)$. Obviously, $\supp_k(u_0|_S,\dots,u_n|_S;\sigma|_A)=u_k(S)$ and $\supp_k(u_0|_T,\dots,u_n|_T;\sigma|_B)=u_k(T)$; as $u_k$ is injective and $S\cap T=\varnothing$, these are disjoint, i.e.~this is indeed an $n$-simplex of the box product.

Conversely, given any $n$-simplex $\big((u_0,\dots,u_n;\sigma), (v_0,\dots,v_n;\tau)\big)$ of the box product, $(u_0\cup v_0,\dots,u_n\cup v_n;\sigma\cup\tau)$ with
\begin{equation*}
(u_k\cup v_k)(x)=\begin{cases}
\,u_k(x) & \text{if }x\in S\\
\,v_k(x) & \text{if }x\in T
\end{cases}
\qquad\text{and}\qquad
(\sigma\cup\tau)_i=\begin{cases}
\,\sigma_i & \text{if }i\in A\\
\,\tau_i & \text{if }i\in B
\end{cases}
\end{equation*}
is well-defined because $S\cap T=\varnothing$ and $A\cap B=\varnothing$, respectively. Moreover, this is an $n$-simplex of the left hand side: $u_k\cup v_k$ is injective, since its restrictions to $S$ and $T$ are, and since $(u_k\cup v_k)(S)=u_k(S)$ is disjoint from $(u_k\cup v_k)(T)=v_k(T)$ by the same support calculation as above. It is then trivial to check that this is inverse to the restriction of $(\rho^{A\cup B,S\cup T}_{A,S}, \rho^{A\cup B,S\cup T}_{B,T})$, which completes the proof of the fourth statement.

For the final statement, it suffices to observe that the diagram
\begin{equation*}
\begin{tikzcd}
E\Inj(u(S'),\omega)\arrow[r, "u^*"]\arrow[d, "\textup{res}"'] & E\Inj(S',\omega)\arrow[d, "\textup{res}"]\\
E\Inj(u(S),\omega)\arrow[r, "u^*"'] & E\Inj(S,\omega)
\end{tikzcd}
\end{equation*}
commutes as both paths through it are given by restricting along $S\to u(S'),s\mapsto u(s)$, and that
\begin{equation*}
\begin{tikzcd}
\prod_{b\in u(A')}\Delta^{m_{u^{-1}(b)}}\arrow[r, "u^*"]\arrow[d, "\pr"'] & \prod_{a\in A'}\Delta^{m_a}\arrow[d,"\pr"]\\
\prod_{b\in u(A)}\Delta^{m_{u^{-1}(b)}}\arrow[r, "u^*"'] & \prod_{a\in A}\Delta^{m_a}
\end{tikzcd}
\end{equation*}
commutes because after postcomposition with $\pr_a$, $a\in A$, both paths agree with the projection $\prod_{b\in u(A')}\Delta^{m_{u^{-1}(b)}}\to \Delta^{m_a}$ onto the $u(a)$-th factor.
\end{proof}

\begin{constr}
Let $X,Y\in\cat{$\bmEM$-SSet}^\tau$. We define
\begin{equation*}
\nabla\colon\textup{C}_X\boxtimes\textup{C}_Y\to\textup{C}_{X\boxtimes Y}
\end{equation*}
as follows: an object $\big((A,S,m_\bullet,f),(B,T,n_\bullet,g)\big)$ is sent to $(A\cup B,S\cup T,(m\cup n)_\bullet,f\cup g)$ where
\begin{equation*}
(m\cup n)_i=\begin{cases}
\,m_i & \text{if }i\in A\\
\,n_i & \text{if }i\in B
\end{cases}
\end{equation*}
and $f\cup g=(f\boxtimes g)\circ(\rho^{A\cup B,S\cup T}_{A,S},\rho^{A\cup B,S\cup T}_{B,T})=(f\circ\rho^{A\cup B,S\cup T}_{A,S},g\circ\rho^{A\cup B,S\cup T}_{B,T})$. Moreover, a morphism $\big((A,S,m_\bullet,f),(B,T,n_\bullet,g)\big)\to\big((A',S',m'_\bullet,f'),(B',T',n'_\bullet,g')\big)$ given by a pair
\begin{align*}
\alpha\colon E\Inj(S,\omega)\times\prod_{a\in A}\Delta^{m_a}&\to E\Inj(S',\omega)\times\prod_{a'\in A'}\Delta^{m'_{a'}}\\
\beta\colon E\Inj(T,\omega)\times\prod_{b\in B}\Delta^{n_b}&\to E\Inj(T',\omega)\times\prod_{b'\in B'}\Delta^{n'_{b'}}
\end{align*}
is sent to the morphism $(A\cup B,S\cup T, (m\cup n)_\bullet,f\cup g)\to (A'\cup B',S'\cup T', (m'\cup n')_\bullet,f'\cup g')$ given by the composition
\begin{align*}
&(\rho^{A'\cup B',S'\cup T'}_{A',S'},\rho^{A'\cup B',S'\cup T'}_{B',T'})^{-1}\circ (\alpha\boxtimes\beta) \circ (\rho^{A\cup B,S\cup T}_{A,S},\rho^{A\cup B,S\cup T}_{B,T})\\
&\qquad=
(\rho^{A'\cup B',S'\cup T'}_{A',S'},\rho^{A'\cup B',S'\cup T'}_{B',T'})^{-1}\circ (\alpha\rho^{A\cup B,S\cup T}_{A,S},\beta\rho^{A\cup B,S\cup T}_{B,T}).
\end{align*}
Finally, we define $\iota\colon {*}\to\textup{C}_*$ as the functor sending the unique object of the left hand side to $(\varnothing,\varnothing,\varnothing,0)$, where $0$ denotes the unique map $E\Inj(\varnothing,\omega)\times\prod_\varnothing\to {*}$.
\end{constr}

\begin{prop}
The above functors are well-defined and $E\mathcal M$-equivariant.
\end{prop}
\begin{proof}
Let us first show that $\nabla$ is well-defined. For this we observe that $A\cup S=\supp(A,S,m_\bullet,f)$ and $B\cup T=\supp(B,T,n_\bullet,g)$ by Lemma~\ref{lemma:C-bullet-supp}, so $(A\cup S)\cap (B\cup T)=\varnothing$ by definition of the box product, and in particular $A\cap B=\varnothing$. Thus, $m\cup n$ is well-defined. Moreover, $\rho^{A\cup B,S\cup T}_{A,S}$ and $\rho^{A\cup B,S\cup T}_{B,T}$ are $E\mathcal M$-equivariant by Lemma~\ref{lemma:rho-basic-properties}-$(\ref{item:rbp-em})$, so $f\cup g=(f\boxtimes g)\circ(\rho^{A\cup B,S\cup T}_{A,S},\rho^{A\cup B,S\cup T}_{B,T})$ is again $E\mathcal M$-equivariant. This shows that $\nabla$ is well-defined on objects. To prove that it is well-defined on morphisms, we observe that as above $A'\cap B'=\varnothing,S'\cap T'=\varnothing$, so that $(\rho^{A'\cup B',S'\cup T'}_{A',S'},\rho^{A'\cup B',S'\cup T'}_{B',T'})$ is indeed invertible by Lemma~\ref{lemma:rho-basic-properties}-$(\ref{item:rbp-isomorphism})$. By another application of Lemma~\ref{lemma:rho-basic-properties}-$(\ref{item:rbp-em})$ we then see that $\nabla(\alpha,\beta)$ is $E\mathcal M$-equivariant. Finally,
\begin{align*}
(f'\cup g')\nabla(\alpha,\beta)&=
(f'\boxtimes g')(\alpha\boxtimes\beta)(\rho_{A,S}^{A\cup B,S\cup T},\rho_{B,T}^{A\cup B,S\cup T})\\
&=\big((f'\alpha)\boxtimes(g'\beta)\big)(\rho_{A,S}^{A\cup B,S\cup T},\rho_{B,T}^{A\cup B,S\cup T})\\
&=(f\boxtimes g)(\rho_{A,S}^{A\cup B,S\cup T},\rho_{B,T}^{A\cup B,S\cup T})=f\cup g,
\end{align*}
i.e.~$\nabla(\alpha,\beta)$ is indeed a morphism in $\textup{C}_{X\boxtimes Y}$ from $\nabla((A,S,m_\bullet,f),(B,T,n_\bullet,g))$ to $\nabla((A',S',m'_\bullet,f'),(B',T',n'_\bullet,g'))$.

It is trivial to check that $\nabla$ is a functor. Let us now prove that it is $E\mathcal M$-equivariant, for which we let $u\in\mathcal M$ be arbitrary. Then
\begin{align*}
&\nabla(u.(A,S,m_\bullet,f),u.(B,T,n_\bullet,g))\\
&\quad=
\nabla\big((u(A),u(S),m_{u^{-1}(\bullet)},f\circ(u^*\times u^*)),
(u(B),u(T),n_{u^{-1}(\bullet)},g\circ(u^*\times u^*))\big)\\
&\quad=\big(u(A\cup B),u(S\cup T),(m_{u^{-1}(\bullet)}\cup n_{u^{-1}(\bullet)})_\bullet, (f\circ(u^*\times u^*))\cup(g\circ(u^*\times u^*))\big).
\end{align*}
It is clear that $(m_{u^{-1}(\bullet)}\cup n_{u^{-1}(\bullet)})_\bullet=(m\cup n)_{u^{-1}(\bullet)}$, so for $\mathcal M$-equivariance on objects it only remains to show that $(f\circ(u^*\times u^*))\cup(g\circ(u^*\times u^*))=(f\cup g)\circ(u^*\times u^*)$. But indeed,
Lemma~\ref{lemma:rho-basic-properties}-$(\ref{item:rbp-equivariant})$ implies that
\begin{equation}\label{eq:rho-u-times-u}
\begin{aligned}
&(\rho^{A\cup B,S\cup T}_{A,S},\rho^{A\cup B,S\cup T}_{B,T})\circ(u^*\times u^*)\\
&\quad=(\rho^{A\cup B,S\cup T}_{A,S}\circ(u^*\times u^*),\rho^{A\cup B,S\cup T}_{B,T}\circ(u^*\times u^*))\\
&\quad=((u^*\times u^*)\circ\rho^{u(A\cup B),u(S\cup T)}_{u(A),u(S)},(u^*\times u^*)\circ\rho^{u(A\cup B),u(S\cup T)}_{u(B),u(T)}),
\end{aligned}
\end{equation}
hence
\begin{align*}
&(f\circ(u^*\times u^*))\cup(g\circ(u^*\times u^*))\\
&\quad=
(f\circ(u^*\times u^*)\circ\rho^{u(A\cup B),u(S\cup T)}_{u(A),u(S)},
g\circ(u^*\times u^*)\circ\rho^{u(A\cup B),u(S\cup T)}_{u(B),u(T)})\\
&\quad=(f\boxtimes g)\circ((u^*\times u^*)\circ\rho^{u(A\cup B),u(S\cup T)}_{u(A),u(S)},(u^*\times u^*)\circ\rho^{u(A\cup B),u(S\cup T)}_{u(B),u(T)})\\
&\quad=(f\boxtimes g)\circ(\rho^{A\cup B,S\cup T}_{A,S},\rho^{A\cup B,S\cup T}_{B,T})\circ(u^*\times u^*)\\
&\quad=(f\cup g)\circ(u^*\times u^*).
\end{align*}

Next, we have to show that $\nabla(u_\circ^{(A,S,m_\bullet,f)},u_\circ^{(B,T,n_\bullet,g)})=u_\circ^{\nabla((A,S,m_\bullet,f), (B,T,n_\bullet,g))}$. As we already know that both sides are maps between the same two objects in $\textup{C}_{X\boxtimes Y}$, it suffices to show this as maps in $\cat{$\bmEM$-SSet}$, for which it is in turn enough that their inverses agree. But indeed,
\begin{align*}
&\nabla(u_\circ^{(A,S,m_\bullet,f)},u_\circ^{(B,T,n_\bullet,g)})^{-1}\\
&\quad=\nabla\big((u_\circ^{(A,S,m_\bullet,f)})^{-1},(u_\circ^{(B,T,n_\bullet,g)})^{-1}\big)\\
&\quad=(\rho^{A\cup B,S\cup T}_{A,S},\rho^{A\cup B,S\cup T}_{B,T})^{-1}\circ\big((u^*\times u^*)\rho^{u(A\cup B),u(S\cup T)}_{u(A),u(S)},(u^*\times u^*)\rho^{u(A\cup B),u(S\cup T)}_{u(B),u(T)}\big)\\
&\quad=(\rho^{A\cup B,S\cup T}_{A,S},\rho^{A\cup B,S\cup T}_{B,T})^{-1}\circ(\rho^{A\cup B,S\cup T}_{A,S},\rho^{A\cup B,S\cup T}_{B,T})\circ(u^*\times u^*)\\
&\quad=u^*\times u^*=(u_\circ^{\nabla((A,S,m_\bullet,f), (B,T,n_\bullet,g))})^{-1}
\end{align*}
where we used $(\ref{eq:rho-u-times-u})$. This completes the argument for $\nabla$.

Finally, $E\mathcal M$-equivariance of $\iota$ amounts to saying that $\iota({*})=(\varnothing,\varnothing,\varnothing,0)$ has empty support, which is immediate from Lemma~\ref{lemma:C-bullet-supp}.
\end{proof}

\begin{prop}
The functors $\iota$ and $\nabla$ define a lax symmetric monoidal structure on $\textup{C}_\bullet\colon\cat{$\bmEM$-SSet}^\tau\to\cat{$\bmEM$-Cat}^\tau$.
\end{prop}
\begin{proof}
It is trivial to check that $\nabla$ is natural; it remains to show the compatibility of $\nabla$ and $\iota$ with the unitality, symmetry, and associativity isomorphisms.

\textit{Unitality.} We will only prove left unitality, the argument for right unitality being analogous (in fact, right unitality will also follow from left unitality together with symmetry). For this we have to show that the composition
\begin{equation*}
{*}\boxtimes\textup{C}_X\xrightarrow{\iota\boxtimes\textup{C}_X}\textup{C}_*\boxtimes\textup{C}_X\xrightarrow{\nabla}\textup{C}_{{*}\boxtimes X}\xrightarrow{\textup{C}_\lambda}\textup{C}_X
\end{equation*}
agrees with the left unitality isomorphism of $(\cat{$\bmEM$-Cat}^\tau,\boxtimes)$, i.e.~projection to the second factor.

Let us first check this on objects: if $(A,S,m_\bullet,f)\in\textup{C}_X$ is arbitrary, then the above sends $({*},(A,S,m_\bullet,f))$ by definition to $(\varnothing\cup A, \varnothing\cup A, (\varnothing\cup m)_\bullet, \lambda\circ(0\cup f))$, so the only non-trivial statement is that $\lambda\circ(0\cup f)=f$. Indeed, by definition $0\cup f=(0\circ\rho^{A,S}_{\varnothing,\varnothing},f\circ\rho^{A,S}_{A,S})$. As $\lambda\colon {*}\boxtimes X\to X$ is given by projection to the second factor, we conclude $\lambda\circ(0\cup f)=f\circ\rho^{A,S}_{A,S}$, so the claim follows from Lemma~\ref{lemma:rho-basic-properties}-$(\ref{item:rbp-unital})$.

Next, let $\alpha\colon(A,S,m_\bullet,f)\to(B,T,n_\bullet,g)$; we have to show that the above composite sends $(\id_*,\alpha)$ to $\alpha$. As we already know that this has the correct source and target, it suffices to show this as morphism in $\cat{$\bmEM$-SSet}$. But indeed, plugging in the definition we see that $\alpha$ is sent to
\begin{equation}\label{eq:unital-morphisms}
(\rho_{\varnothing,\varnothing}^{A,S},\rho^{A,S}_{A,S})^{-1}(\id\circ\rho_{\varnothing,\varnothing}^{A,S},\alpha\circ\rho^{A,S}_{A,S}).
\end{equation}
As $\rho^{A,S}_{A,S}=\id$ by Lemma~\ref{lemma:rho-basic-properties}-$(\ref{item:rbp-unital})$, we see that projecting onto the second factor is left inverse to $(\rho_{\varnothing,\varnothing}^{A,S},\rho^{A,S}_{A,S})$; as the latter is an isomorphism by Lemma~\ref{lemma:rho-basic-properties}-$(\ref{item:rbp-isomorphism})$ (or alternatively using that the projection is an isomorphism for obvious reasons), it is then also right inverse, and $(\ref{eq:unital-morphisms})$ equals $\alpha\rho^{A,S}_{A,S}=\alpha$ as desired.

\textit{Associativity.} We have to show that the diagram
\begin{equation*}
\begin{tikzcd}
(\textup{C}_X\boxtimes\textup{C}_Y)\boxtimes\textup{C}_Z\arrow[r, "a", "\cong"']\arrow[d, "\nabla\boxtimes\textup{C}_Z"'] & \textup{C}_X\boxtimes(\textup{C}_Y\boxtimes\textup{C}_Z)\arrow[d, "\textup{C}_X\boxtimes\nabla"]\\
\textup{C}_{X\boxtimes Y}\boxtimes\textup{C}_Z\arrow[d,"\nabla"'] & \textup{C}_X\boxtimes\textup{C}_{Y\boxtimes Z}\arrow[d,"\nabla"]\\
\textup{C}_{(X\boxtimes Y)\boxtimes Z}\arrow[r, "\cong", "\textup{C}_a"'] & \textup{C}_{X\boxtimes(Y\boxtimes Z)}
\end{tikzcd}
\end{equation*}
commutes for all $X,Y,Z\in\cat{$\bmEM$-SSet}^\tau$; here we denote the associativity isomorphism by `$a$' instead of the usual `$\alpha$' in order to avoid confusion with our notation for a generic morphism in $\textup{C}_\bullet$.

To check this on objects we let $\big(((A,S,m_\bullet,f), (B,T,n_\bullet,g)), (C,U,o_\bullet,h)\big)$ be any object of the top left corner. Then the upper right path through the diagram sends this to $(A\cup(B\cup C), S\cup(T\cup U),(m\cup(n\cup o))_\bullet, f\cup(g\cup h))$ while the lower left path sends it to $((A\cup B)\cup C, (S\cup T)\cup U, ((m\cup n)\cup o)_\bullet, a\circ ((f\cup g)\cup h)$. It is clear that the first three components agree, so it only remains to show that $f\cup(g\cup h)=a\circ((f\cup g)\cup h)$ as maps $E\Inj(S\cup T\cup U,\omega)\times\prod_{i\in A\cup B\cup C}\Delta^{(m\cup n\cup o)_i}\to X\boxtimes(Y\boxtimes Z)$. But indeed,
\begin{align*}
f\cup(g\cup h)&=(f\rho_{A,S}^{A\cup B\cup C,S\cup T\cup U},(g\cup h)\rho^{A\cup B\cup C,S\cup T\cup U}_{B\cup C,T\cup U})\\
&=(f\rho^{A\cup B\cup C,S\cup T\cup U}_{A,S},(g\rho^{B\cup C,T\cup U}_{B,T}, h\rho^{B\cup C,T\cup U}_{C,U})\rho^{A\cup B\cup C,S\cup T\cup U}_{B\cup C,T\cup U})\\
&=(f\rho^{A\cup B\cup C,S\cup T\cup U}_{A,S},(g\rho^{A\cup B\cup C,S\cup T\cup U}_{B,T},h\rho^{A\cup B\cup C,S\cup T\cup U}_{C,U}))
\end{align*}
where the final equality follows from Lemma~\ref{lemma:rho-basic-properties}-$(\ref{item:rbp-associative})$. Analogously, one shows that
\begin{equation*}
a\circ\big((f\cup g)\cup h\big)=a\circ((f\rho^{A\cup B\cup C,S\cup T\cup U}_{A,S},g\rho^{A\cup B\cup C,S\cup T\cup U}_{B,T}),h\rho^{A\cup B\cup C,S\cup T\cup U}_{C,U})
\end{equation*}
and this is obviously equal to the above.

Next, we let $\big(((A',S',m'_\bullet,f'), (B',T',n'_\bullet,g')), (C',U',o'_\bullet,h')\big)$ be another such object, and we let $((\alpha,\beta),\gamma)$ be a morphism. We have to show that both paths through the diagram send this to the same morphism in $\textup{C}_{X\boxtimes(Y\boxtimes Z)}$, for which it is then enough to show equality as morphisms in $\cat{$\bmEM$-SSet}^\tau$. For this we first observe that on the one hand by Lemma~\ref{lemma:rho-basic-properties}-$(\ref{item:rbp-associative})$
\begin{align*}
&(\rho_{A',S'}^{A'\cup B'\cup C',S'\cup T'\cup U'},(\rho^{A'\cup B'\cup C',S'\cup T'\cup U'}_{B',T'},\rho^{A'\cup B'\cup C',S'\cup T'\cup U'}_{C',U'}))\\
&\quad=\big(\id\boxtimes(\rho^{B'\cup C',T'\cup U'}_{B',T'},\rho^{B'\cup C',T'\cup U'}_{C',U'})\big)\circ(\rho^{A'\cup B'\cup C',S'\cup T'\cup U'}_{A',S'},\rho^{A'\cup B'\cup C',S'\cup T'\cup U'}_{B'\cup C',T'\cup U'})
\end{align*}
(in particular this is an isomorphism), and on the other hand obviously
\begin{align*}
&(\rho_{A',S'}^{A'\cup B'\cup C',S'\cup T'\cup U'},(\rho^{A'\cup B'\cup C',S'\cup T'\cup U'}_{B',T'},\rho^{A'\cup B'\cup C',S'\cup T'\cup U'}_{C',U'}))\\
&\quad=a\circ\big((\rho^{A'\cup B'\cup C',S'\cup T'\cup U'}_{A',S'}, \rho^{A'\cup B'\cup C',S'\cup T'\cup U'}_{B',T'}),\rho^{A'\cup B'\cup C',S'\cup T'\cup U'}_{C',U'}\big).
\end{align*}
We now calculate
\begin{equation}\label{eq:associativity-lhs}
\begin{aligned}
&(\rho_{A',S'},(\rho_{B',T'},\rho_{C',U'}))\nabla(\alpha,\nabla(\beta,\gamma))\\
&\quad=
(\id\boxtimes(\rho_{B',T'},\rho_{C',U'}))(\alpha\rho_{A,S},\nabla(\beta,\gamma)\rho_{B\cup C,T\cup U})\\
&\quad=
(\alpha\rho_{A,S},(\rho_{B',T'},\rho_{C',U'})\nabla(\beta,\gamma)\rho_{B\cup C,T\cup U})\\
&\quad=(\alpha\rho_{A,S},(\beta\rho_{B,T},\gamma\rho_{C,U}))
\end{aligned}
\end{equation}
where we omitted the superscripts for legibility. Analogously,
\begin{align*}
(\rho_{A',S'},(\rho_{B',T'},\rho_{C',U'}))\nabla(\nabla(\alpha,\beta),\gamma)
&=a\circ((\rho_{A',S'},\rho_{B',T'}),\rho_{C',U'})\nabla(\nabla(\alpha,\beta),\gamma)\\
&=a\circ ((\alpha\rho_{A,S},\beta\rho_{B,T}),\gamma\rho_{C,U})
\end{align*}
which equals $(\ref{eq:associativity-lhs})$. We conclude that $\nabla(\alpha,\nabla(\beta,\gamma))=\nabla(\nabla(\alpha,\beta),\gamma)$ as morphisms in $\cat{$\bmEM$-SSet}^\tau$ as they agree after postcomposing with an isomorphism. This completes the proof of associtativity.

\textit{Symmetry.} Finally, we have to show that the diagram
\begin{equation*}
\begin{tikzcd}
\textup{C}_X\boxtimes\textup{C}_Y\arrow[r, "\tau", "\cong"']\arrow[d, "\nabla"'] & \textup{C}_Y\boxtimes\textup{C}_X\arrow[d, "\nabla"]\\
\textup{C}_{X\boxtimes Y}\arrow[r, "\cong", "\textup{C}_\tau"'] & \textup{C}_{Y\boxtimes X}
\end{tikzcd}
\end{equation*}
commutes for all tame $E\mathcal M$-simplicial sets $X,Y$, where $\tau$ denotes the symmetry isomorphism of $\boxtimes$ on $\cat{$\bmEM$-Cat}^\tau$ and $\cat{$\bmEM$-SSet}^\tau$, respectively; in both cases it is given by restriction of the flip map $K\times L\cong L\times K$.

Again, let us first check this on objects. If $\big((A,S,m_\bullet,f),(B,T,n_\bullet,g)\big)$ is an object of the top left corner, then the upper right path through this diagram sends this to $(B\cup A, T\cup S, (n\cup m)_\bullet, g\cup f)$, while the lower left path sends it to $(A\cup B,S\cup T,(m\cup n)_\bullet,\tau\circ(f\cup g))$. The first three components agree trivially, while for the fourth components we simply calculate
\begin{equation*}
\tau\circ(f\cup g)=\tau\circ(f\rho_{A,S},g\rho_{B,T})=(g\rho_{B,T},f\rho_{A,S})=g\cup f.
\end{equation*}
This proves commutativity on objects. If now $\big((A',S',m'_\bullet,f'),(B',T',n'_\bullet,g')\big)$ is another such object and $(\alpha,\beta)$ is a morphism, then in order to show that both paths through the diagram send $(\alpha,\beta)$ to the same morphism of $\textup{C}_{Y\boxtimes X}$ it is again enough to check this as morphisms in $\cat{$\bmEM$-SSet}$. But indeed, the top right path through the diagram sends $(\alpha,\beta)$ to $(\rho_{B',T'},\rho_{A',S'})^{-1}(\beta\rho_{B,T},\alpha\rho_{A,S})$. Using that $(\rho_{B',T'},\rho_{A',S'})=\tau\circ(\rho_{A',S'},\rho_{B',T'})$, this equals
\begin{equation*}
(\rho_{A',S'},\rho_{B',T'})^{-1}\circ\tau\circ(\beta\rho_{B,T},\alpha\rho_{A,S})=(\rho_{A',S'},\rho_{B',T'})^{-1}(\alpha\rho_{A,S},\beta\rho_{B,T})
\end{equation*}
which is by definition the image of $(\alpha,\beta)$ under the lower left composition. This completes the proof of symmetry and hence of the proposition.
\end{proof}

As before, the corresponding result for $\textup{C}_\bullet\colon\cat{$\bmEM$-$\bm G$-SSet}^\tau\to\cat{$\bmEM$-$\bm G$-Cat}^\tau$ follows formally. In particular, $\textup{C}_\bullet$ canonically lifts to a functor $\cat{$\bm G$-ParSumSSet}\to\cat{$\bm G$-ParSumCat}$. Explicitly, if $X$ is a parsummable simplicial set, then $\textup{C}_X$ has the same underlying $E\mathcal M$-$G$-category as before. The sum of two disjointly supported objects $(A,S,m_\bullet,f),(B,T,n_\bullet,g)$ is $(A\cup B,S\cup T, (m\cup n)_\bullet,f+g)$, with $f+g$ given by the composition
\begin{equation*}
E\Inj(S\cup T,\omega)\times\prod_{i\in A\cup B}\Delta^{(m\cup n)_i}\xrightarrow{f\cup g} X\boxtimes X\xrightarrow{+}X
\end{equation*}
where $+$ denotes the sum operation of the $G$-parsummable simplicial set $X$. Moreover, the sum of two morphisms $\alpha$, $\beta$ having disjointly supported sources and disjointly supported targets agrees as a map of $E\mathcal M$-simplicial sets with $\nabla(\alpha,\beta)$ as defined above. Finally, the unit is given by $(\varnothing,\varnothing,\varnothing,0)$ where $0$ denotes the map $E\Inj(\varnothing,\omega)\times\prod_\varnothing\to X$ with image the zero vertex of $X$.

Next, we will show that the natural maps $\epsilon$ and $\tilde\epsilon$ also define natural transformations between these lifts.

\begin{prop}\label{prop:epsilon-monoidal}
The natural transformation $\epsilon\colon\nerve\circ\textup{C}_\bullet\Rightarrow\id_{\cat{$\bmEM$-SSet}^\tau}$ is {\rm(}symmetric{\rm)} monoidal.
\end{prop}
\begin{proof}
We have to show that the diagrams
\begin{equation*}
\begin{tikzcd}
(\nerve\textup{C}_X)\boxtimes(\nerve\textup{C}_Y)\arrow[r,"\epsilon\boxtimes\epsilon"]\arrow[d, "\nabla"'] & X\boxtimes Y\\
\nerve\textup{C}_{X\boxtimes Y}\arrow[ur, bend right=10pt, "\epsilon"']
\end{tikzcd}
\qquad
\begin{tikzcd}
{*}\arrow[r, "="]\arrow[d, "\iota"'] & {*}\\
\nerve\textup{C}_*\arrow[ru, bend right=10pt, "\epsilon"']
\end{tikzcd}
\end{equation*}
commute, where $\nabla$ and $\iota$ denote the compositions of the structure maps of $\nerve$ and $\textup{C}_\bullet$ of the same name.

The commutativity of the right hand triangle is trivial as the target is terminal. For the left hand triangle, we consider any $k$-simplex of $(\nerve\textup{C}_X)\boxtimes(\nerve\textup{C}_Y)$. This is by definition and Example~\ref{ex:nerve-kth-support} given by a pair of a $k$-simplex
\begin{equation*}
(A_0,S_0,m^{(0)}_\bullet,f_0)\xrightarrow{\alpha_1}\cdots\xrightarrow{\alpha_k}(A_k,S_k,m^{(k)}_\bullet,f_k)
\end{equation*}
of $\nerve(\textup{C}_X)$ and a $k$-simplex
\begin{equation*}
(B_0,T_0,n^{(0)}_\bullet,g_0)\xrightarrow{\beta_1}\cdots\xrightarrow{\beta_k}(B_k,T_k,n^{(k)}_\bullet,g_k)
\end{equation*}
of $\nerve(\textup{C}_Y)$ such that $\supp(A_i,S_i,m^{(i)}_\bullet,f_i)\cap\supp(B_i,T_i,n^{(i)}_\bullet,g_i)=\varnothing$ for $i=0,\dots,k$.

If $\sigma_{\alpha_\bullet}$, $\sigma_{\beta_\bullet}$ are defined as before, then the top arrow in this diagram sends $(\alpha_\bullet,\beta_\bullet)$ to $(f_k(\sigma_{\alpha_\bullet}),g_k(\sigma_{\beta_\bullet}))$. On the other hand, the lower path sends $(\alpha_\bullet,\beta_\bullet)$ to $(f_k\cup g_k)(\sigma_{\nabla(\alpha_\bullet,\beta_\bullet)})$. Here $\sigma_{\nabla(\alpha_\bullet,\beta_\bullet)}$ is uniquely characterized by demanding that its $\ell$-th vertex be given by
\begin{equation*}
\nabla(\alpha_k,\beta_k)\cdots\nabla(\alpha_{\ell+1},\beta_{\ell+1})(\iota_{S_\ell\cup T_\ell},{*})\in E\Inj(S_k\cup T_k,\omega)\times\prod_{i\in A_k\cup B_k}\Delta^{(m\cup n)_i}.
\end{equation*}
By functoriality of $\nabla$ and its definition, this is equal to
\begin{equation*}
(\rho_{A_k,S_k},\rho_{B_k,T_k})^{-1}
\big((\alpha_k\cdots\alpha_{\ell+1})\boxtimes(\beta_k\cdots\beta_{\ell+1})\big)(\rho_{A_\ell,S_\ell},\rho_{B_\ell,T_\ell})(\iota_{S_\ell\cup T_\ell,{*}}),
\end{equation*}
and as obviously $(\rho_{A_\ell,S_\ell},\rho_{B_\ell,T_\ell})(\iota_{S_\ell\cup T_\ell},*)=\big((\iota_{S_\ell},{*}),(\iota_{T_\ell},{*})\big)$, we conclude that $\sigma_{\nabla(\alpha_\bullet,\beta_\bullet)}=(\rho_{A_k,S_k},\rho_{B_k,T_k})^{-1}(\sigma_{\alpha_\bullet},\sigma_{\beta_\bullet})$. Thus,
\begin{equation*}
\epsilon(\nabla(\alpha_\bullet,\beta_\bullet))=
(f_k\cup g_k)(\sigma_{\nabla(\alpha_\bullet,\beta_\bullet)})=(f_k\boxtimes g_k)(\sigma_{\alpha_\bullet},\sigma_{\beta_\bullet})=(f_k(\sigma_{\alpha_\bullet}),g_k(\sigma_{\beta_\bullet}))
\end{equation*}
as claimed.
\end{proof}

\begin{prop}\label{prop:tilde-epsilon-monoidal}
The natural transformation $\tilde\epsilon\colon \textup{C}_\bullet\circ\nerve\Rightarrow\id_{\cat{$\bmEM$-Cat}^\tau}$ is {\rm(}symmetric{\rm)} monoidal.
\end{prop}
\begin{proof}
We have to prove commutativity of the diagrams
\begin{equation*}
\begin{tikzcd}
\textup{C}_{\nerve\mathcal C}\boxtimes\textup{C}_{\nerve\mathcal D}\arrow[r,"\tilde\epsilon\boxtimes\tilde\epsilon"]\arrow[d, "\nabla"'] & \mathcal C\boxtimes\mathcal D\\
\textup{C}_{\nerve(\mathcal C\boxtimes\mathcal D)}\arrow[ur, bend right=10pt, "\tilde\epsilon"']
\end{tikzcd}
\qquad
\begin{tikzcd}
{*}\arrow[r, "="]\arrow[d, "\iota"'] & {*}\\
\textup{C}_{\nerve({*})}\arrow[ru, bend right=10pt, "\tilde\epsilon"']
\end{tikzcd}
\end{equation*}
which in the case of the right hand triangle is trivial again. For the left hand diagram, it suffices to prove this after applying $\nerve$ as the latter is fully faithful. The resulting diagram is
\begin{equation}\label{diag:to-check-nerve}
\begin{tikzcd}
\nerve(\textup{C}_{\nerve\mathcal C}\boxtimes\textup{C}_{\nerve\mathcal D})\arrow[r,"\nerve(\tilde\epsilon\boxtimes\tilde\epsilon)"]\arrow[d, "\nerve(\nabla)"'] & \nerve(\mathcal C\boxtimes\mathcal D)\\
\nerve(\textup{C}_{\nerve(\mathcal C\boxtimes\mathcal D)})\arrow[ur, bend right=10pt, "\epsilon"']
\end{tikzcd}
\end{equation}
where we have applied the definition of $\tilde\epsilon$. We now consider the 3-dimensional diagram
\begin{equation*}
\begin{tikzcd}[column sep=small, row sep=small]
& \nerve(\textup{C}_{\nerve\mathcal C}\boxtimes\textup{C}_{\nerve\mathcal D})\arrow[rr]\arrow[dd] && \nerve(\mathcal C\boxtimes\mathcal D)\\
(\nerve\textup{C}_{\nerve\mathcal C})\boxtimes(\nerve\textup{C}_{\nerve\mathcal D})\arrow[ur,"\cong"', "\nabla"]\arrow[rr, crossing over]\arrow[dd] && (\nerve\mathcal C)\boxtimes (\nerve\mathcal D)\arrow[ur, "\cong", "\nabla"']\\
& \nerve\textup{C}_{\nerve(\mathcal C\boxtimes\mathcal D)}\arrow[uurr, bend right=25pt]\\
\nerve\textup{C}_{(\nerve\mathcal C)\boxtimes(\nerve\mathcal D)}\arrow[ur, "\nerve\textup{C}_\nabla", "\cong"']\arrow[uurr, bend right=25pt, crossing over]
\end{tikzcd}
\end{equation*}
where the back face is $(\ref{diag:to-check-nerve})$, the front face is the coherence diagram for the symmetric monoidal transformation $\epsilon\colon\nerve\circ\textup{C}_\bullet\Rightarrow\id$, and the front-to-back maps are induced by the structure isomorphisms of the strong symmetric monoidal functor $\nerve$ as indicated. Then the front face commutes by the previous proposition, the left face commutes by the definition of the structure maps of a composition of lax symmetric monoidal functors, the top face commutes by naturality of $\nabla$, and the lower right face commutes by naturality of $\epsilon$. As all the front-to-back maps are isomorphisms, it follows that also the back face commutes, which then completes the proof of the proposition.
\end{proof}

As before, we automatically get the corresponding statements for the lifts of $\epsilon$ and $\tilde\epsilon$ to $\cat{$\bmEM$-$\bm G$-SSet}^\tau$ and $\cat{$\bmEM$-$\bm G$-Cat}^\tau$, respectively. We can now immediately prove the following precise form of Theorem~\ref{thm:par-sum-cat-vs-sset} from the introduction:

\begin{thm}\label{thm:nerve-vs-c-bullet}
The lifts of $\nerve$ and $\textup{C}_\bullet$ constructed above define mutually inverse homotopy equivalences
\begin{equation*}
\textup{C}_\bullet\colon\cat{$\bm G$-ParSumSSet}\rightleftarrows
\cat{$\bm G$-ParSumCat} :\nerve
\end{equation*}
with respect to the $G$-global weak equivalences on both sides. More precisely, the natural maps $\epsilon\colon\nerve(\textup{C}_X)\to X$ and $\tilde\epsilon\colon\textup{C}_{\nerve\mathcal C}\to\mathcal C$ define levelwise $G$-global weak equivalences between the two composites and the respective identities.
\end{thm}
\begin{proof}
We{\hskip0pt minus .75pt} first{\hskip0pt minus .75pt} observe{\hskip0pt minus .75pt} that{\hskip0pt minus .75pt} these{\hskip0pt minus .75pt} indeed{\hskip0pt minus .75pt} assemble{\hskip0pt minus .75pt} into{\hskip0pt minus .75pt} natural{\hskip0pt minus .75pt} transformations $\nerve\circ\textup{C}_\bullet\Rightarrow\id_{\cat{$\bm G$-ParSumSSet}}${\hskip0pt minus .75pt} and{\hskip0pt minus .75pt} $\textup{C}_\bullet\circ\nerve\Rightarrow\id_{\cat{$\bm G$-ParSumCat}}${\hskip0pt minus .75pt} by{\hskip0pt minus .75pt} Propositions~\ref{prop:epsilon-monoidal} and~\ref{prop:tilde-epsilon-monoidal}, respectively. As the weak equivalences of $\cat{$\bm G$-ParSumSSet}$ and $\cat{$\bm G$-ParSumCat}$ are created in $\cat{$\bmEM$-$\bm G$-SSet}^\tau$ and $\cat{$\bmEM$-$\bm G$-Cat}^\tau$, respectively, the claim now follows from Theorem~\ref{thm:nerve-vs-c-bullet-unstable}.
\end{proof}

Moreover, we can now prove:

\begin{thm}\label{thm:parsum-cat-sat-global}
All the inclusions in
\begin{equation*}
\cat{$\bm G$-ParSumCat}^{s}\hookrightarrow\cat{$\bm G$-ParSumCat}^{ws}\hookrightarrow\cat{$\bm G$-ParSumCat}
\end{equation*}
are homotopy equivalences with respect to the $G$-global weak equivalences.
\end{thm}
\begin{proof}
For the left hand inclusion we have already shown this as Corollary~\ref{cor:ws-vs-s-parsummable}. For the right hand inclusion it suffices to observe again that $\textup{C}_\bullet\circ\nerve$ is homotopy inverse, which follows from Theorem~\ref{thm:nerve-vs-c-bullet} by the same arguments as in Corollary~\ref{cor:sat-ge}.
\end{proof}

\section{$G$-global homotopy theory of symmetric monoidal $G$-categories}\label{sec:sym-mon}
We will write $\cat{SymMonCat}$ for the $1$-category of small symmetric monoidal categories and strong symmetric monoidal functors, and we will denote by $\cat{$\bm G$-SymMonCat}$ the corresponding category of $G$-objects. Explicitly, an object of $\cat{$\bm G$-SymMonCat}$ is a small symmetric monoidal category equipped with a strict $G$-action through strong symmetric monoidal functors (a small \emph{symmetric monoidal $G$-category}, for short), and the morphisms are given by strong symmetric monoidal functors that strictly preserve the actions.

We want to study $\cat{$\bm G$-SymMonCat}$ from a $G$-global perspective, for which we introduce the following notion of weak equivalence:

\begin{defi}
A $G$-equivariant functor $f\colon\mathscr C\to\mathscr D$ of small $G$-categories is a \emph{$G$-global weak equivalence} if \begin{equation*}\Fun(EH,\varphi^*f)^H\colon\Fun(EH,\varphi^*\mathscr C)^H\to\Fun(EH,\varphi^*\mathscr D)^H\end{equation*} is a weak homotopy equivalence for every finite group $H$ and every homomorphism $\varphi\colon H\to G$. A morphism in $\cat{$\bm G$-SymMonCat}$ is called a $G$-global weak equivalence if and only if its underlying $G$-equivariant functor is.
\end{defi}

By Lemma~\ref{lemma:universal-abundant} we may restrict ourselves to those $H$ that are universal subgroups of $\mathcal M$ in the above definition without changing the notion of $G$-global weak equivalence.

\begin{ex}
If $G=1$, then the $1$-global weak equivalences of small categories are precisely the global equivalences in the sense of \cite[Definition~3.2]{schwede-cat}.
\end{ex}

\begin{ex}\label{ex:categorical-are-g-global}
Any underlying equivalence of categories induces equivalences on categorical homotopy fixed points, so it is in particular a $G$-global weak equivalence.
\end{ex}

In what follows, we will again abbreviate $\Fun(EH,\varphi^*(\blank))^H\mathrel{=:}(\blank)^{\myh\varphi}$, corresponding to the fact that the above agrees with the previous definition of $(\blank)^{\myh\varphi}$ for the case of a trivial $E\mathcal M$-action. As this creates potential ambiguity for $E\mathcal M$-$G$-categories, we will always distinguish between an $E\mathcal M$-$G$-category and its underlying $G$-category below; homotopically, however, this ambiguity is inconsequential anyhow by the following easy observation:

\begin{lemma}\label{lemma:comparison-homotopy-fp}
Let $\mathcal C$ be an $E\mathcal M$-$G$-category, and let $\forget\colon\cat{$\bmEM$-$\bm G$-Cat}\to\cat{$\bm G$-Cat}$ be the forgetful functor. Then there is a natural zig-zag of equivalences between $\mathcal C^{\myh\varphi}$ and $(\forget\mathcal C)^{\myh\varphi}$ for any subgroup $H\subset\mathcal M$ and any $\varphi\colon H\to G$.
\begin{proof}
While the claim could be proven analogously to \cite[Proposition~7.6]{schwede-k-theory}, we will give a slightly different argument: let us consider the zig-zag
\begin{equation}\label{eq:action-vs-no-action}
\mathcal C\xleftarrow{\textup{action}} E\mathcal M\times \mathcal C^{\textup{triv}}\xrightarrow{\pr} \mathcal C^{\textup{triv}},
\end{equation}
where $\mathcal C^{\textup{triv}}$ denotes $\mathcal C$ with trivial $E\mathcal M$-action. There is an evident way to make the middle term functorial in $\mathcal C$, and with respect to this the above two maps are clearly natural. Moreover, one easily checks that they are both $(E\mathcal M\times G)$-equivariant.

We claim that they are also underlying equivalences of categories. Indeed, this is obvious for the projection as $E\mathcal M$ is contractible. The non-equivariant functor $(1,\blank)\colon\mathcal C\to E\mathcal M\times\mathcal C$ is right-inverse to it, hence again an equivalence of categories. But it is also right inverse to the action map $E\mathcal M\times\mathcal C\to\mathcal C$, hence also the latter is an equivalence of categories as desired.

The claim now simply follows by applying $(\blank)^{\myh\varphi}$ to $(\ref{eq:action-vs-no-action})$.
\end{proof}
\end{lemma}

\begin{rk}
The $G$-global weak equivalences might look a bit counterintuitive at first, so let us explain the connection to classical equivariant $K$-theory, for which we assume that $G$ is finite.

If $\mathscr C$ is a small symmetric monoidal category, then the Shimada-Shimakawa construction \cite[Definition~2.1]{shimada-shimakawa} associates to this a special $\Gamma$-space $\Gamma(\mathscr C)$. If we now let $G$ act suitably on $\mathscr C$, then $\Gamma(\mathscr C)$ acquires a $G$-action through functoriality, making it into a $\Gamma$-$G$-space; more precisely, $\Gamma$ is functorial in the category $\cat{SymMonCat}^0$ of small symmetric monoidal categories and \emph{strictly unital} strong symmetric monoidal functors, so it induces a functor $\cat{$\bm G$-SymMonCat}^0\to\cat{$\bm\Gamma$-$\bm G$-SSet}$ where the left hand side again denotes the category of $G$-objects in $\cat{SymMonCat}^0$; in particular, its objects have $G$-actions through \emph{strictly unital} strong symmetric monoidal functors.

Unfortunately (or interestingly), $\Gamma(\mathscr C)$ is usually \emph{not} special in the correct $G$-equivariant sense. However, it is an observation going back to Shimakawa \cite[discussion before Theorem~$\text{A}'$]{shimakawa} and later extensively used by \cite{merling} that this defect can be cured by replacing $\mathscr C$ with $\Fun(EG,\mathscr C)$ equipped with the conjugation action.

Thus, the natural way to obtain a \emph{special} $\Gamma$-$G$-space is via the composition
\begin{equation}\label{eq:shimakawa-kg}
\cat{$\bm G$-SymMonCat}^0\xrightarrow{\Fun(EG,\blank)}\cat{$\bm G$-SymMonCat}^0\xrightarrow{\Gamma}\cat{$\bm\Gamma$-$\bm G$-SSet},
\end{equation}
and this is also the basis for the usual definition of the equivariant algebraic $K$-theory of $\mathscr C$.

There are several useful notions of $G$-equivariant weak equivalences on the right hand side, the simplest (and strongest) of which are the \emph{level equivalences}, see e.g.~\cite[4.2.1]{equivariant-gamma}. It is then not hard to check from the definitions that a map $f\colon\mathscr C\to\mathscr D$ in $\cat{$\bm G$-SymMonCat}^0$ induces a level equivalence under $(\ref{eq:shimakawa-kg})$ if and only if the induced functor $\Fun(EG,f)^H\colon\Fun(EG,\mathscr C)^H\to\Fun(EG,\mathscr D)^H$ is a weak homotopy equivalence for every subgroup $H\subset G$. Using that the inclusion $H\hookrightarrow G$ induces an $H$-equivariant equivalence $EH\to EG$, we conclude in particular that any $G$-global weak equivalence induces a level equivalence under $(\ref{eq:shimakawa-kg})$.
\end{rk}

Finally, we will compare the $G$-global homotopy theory of $\cat{$\bm G$-SymMonCat}$ (and $\cat{$\bm G$-SymMonCat}^0$) to the models considered so far, for which it will be useful to introduce an intermediate step. We therefore recall:

\begin{defi}
A \emph{permutative category} is a symmetric monoidal category in which the associativity and unitality isomorphisms are the respective identities. We write $\cat{PermCat}$ for the category of small permutative categories and \emph{strict} symmetric monoidal functors.
\end{defi}

\begin{prop}
The inclusions
\begin{equation*}
\cat{$\bm G$-PermCat}\hookrightarrow\cat{$\bm G$-SymMonCat}^0\hookrightarrow\cat{$\bm G$-SymMonCat}
\end{equation*}
are homotopy equivalences with respect to the underlying equivalences of categories {\rm(}and hence also with respect to the $G$-global weak equivalences{\rm).}
\begin{proof}
It suffices to prove the first statement, for which it is in turn enough to consider the case $G=1$.

The composition $\cat{PermCat}\hookrightarrow\cat{SymMonCat}$ is a homotopy equivalence as a consequence of MacLane's strictification construction \cite[Section~XI.3]{cat-working}, see e.g.~\cite[Theorem~1.19]{perm-parsum-categorical} for an elaboration on this argument.

We are therefore reduced to showing that $\cat{SymMonCat}^0\hookrightarrow\cat{SymMonCat}$ is a homotopy equivalence; this is again well-known, but I do not know of an explicit reference, so let me sketch the construction of a homotopy inverse:

For a small symmetric monoidal category $\mathscr C$, consider the map $\pi\colon\Ob(\mathscr C)\amalg\{{\bm1}'\}\to\Ob(\mathscr C)$ that is given by the identity on $\Ob(\mathscr C)$ and that sends ${\bm 1}'$ to $\bm1$. We then define $\mathscr C^0$ as the category with set of objects $\Ob(\mathscr C)\amalg\{{\bm1}'\}$ and hom sets $\Hom_{\mathscr C^0}(X,Y)\mathrel{:=}\Hom_{\mathscr C}(\pi(X),\pi(Y))$; then $\pi$ tautologically extends to a functor $\mathscr C^0\to\mathscr C$, and this is clearly an equivalence of categories.

We extend the tensor product from $\Ob(\mathscr C)$ to $\Ob(\mathscr C^0)$ by demanding that $\textbf{1}'$ be a strict unit. Then $\pi$ commutes with the tensor products on objects, so there is by full faithfulness a unique way to extend the tensor product on $\Ob(\mathscr C^0)$ to morphisms in such a way that $\pi$ strictly preserves the tensor products. By the same argument, we can uniquely lift the associativity, unitality, and symmetry isomorphism from $\mathscr C$ to $\mathscr C^0$ through $\pi$, making $\mathscr C^0$ a symmetric monoidal category with unit ${\bm 1}'$, and $\pi$ a strict symmetric monoidal functor.

On the other hand, the inclusion $\eta\colon\mathscr C\hookrightarrow\mathscr C^0$ is right inverse to $\pi$, so there is a unique strong symmetric monoidal structure on $\eta$ such that the composition $\pi\eta$ agrees as a strong symmetric monoidal functor with the identity. We now claim that $\eta\colon\mathscr C\to\mathscr C^0$ has the following universal property: for any strong symmetric monoidal $f\colon\mathscr C\to\mathscr D$ there exists a unique strictly unital strong monoidal functor $\tilde f\colon\mathscr C^0\to\mathscr D$ with $f=\tilde f\circ\eta$. With this established, it will then follow formally that the assignment $\mathscr C\mapsto\mathscr C^0$ extends to a functor $(\blank)^0\colon\cat{SymMonCat}\to\cat{SymMonCat}^0$ left adjoint to the inclusion and with unit $\eta$. As $\eta$ is an underlying equivalence and the inclusion $\cat{SymMonCat}^0\hookrightarrow\cat{SymMonCat}$ creates underlying equivalences, it follows by $2$-out-of-$3$ that also $(\blank)^0$ is homotopical. Together with the triangle identity we moreover see that also the counit $\epsilon$ is an underlying equivalence, finishing the proof (in fact, one can also easily check that the counit is simply the functor $\pi$ considered above).

It remains to prove the claim. Let us first show that the underlying functor of $\tilde f$ is unique. Indeed, as $\tilde f$ is strictly unital, $\tilde f({\bm1}')=\bm1$; together with $\tilde f\circ\eta=f$, this means that $\tilde f$ is uniquely prescribed on objects. On the other hand, this also prescribes $\tilde f$ on all morphisms between objects of $\mathscr C$, whereas the equality of the unit isomorphisms for $\tilde f\circ\eta$ and $f$ prescribes $\tilde f$ on the unit isomorphism $\iota\colon{\bm 1}'\to\bm1$ of $\eta$ (i.e.~the map corresponding to the identity of $\bm1$). The claim follows as any morphism in $\mathscr C^0$ can be expressed as a composition of a morphism in $\mathscr C$ and (possibly) $\iota$ and $\iota^{-1}$. To see that $\tilde f$ is also unique as a strictly unital strong symmetric monoidal functor, it suffices to show that there is at most one choice of the multiplicativity isomorphisms $\nabla_{X,Y}\colon\tilde f(X)\otimes\tilde f(Y)\to\tilde f(X\otimes Y)$. But indeed, as $\nabla$ is natural, it suffices to show this after precomposing with the equivalence $\eta\times\eta$, where this follows from the equality $\tilde f\circ\eta=f$ of strong symmetric monoidal functors.

Finally, we construct a strictly unital strong symmetric monoidal $\tilde f\colon\mathscr C^0\to\mathscr D$ as follows: we define $\tilde f(X)=f(X)$ for any $X\in\mathscr C$ and $\tilde f({\bm1}')=\bm1$. We now consider the isomorphisms $\theta_X=\id_{f(X)}\colon\tilde f(X)\to f(\pi(X))$ for $X\in\mathscr C$, $\theta_{{\bm1}'}=\iota\colon\bm1\to f(\bm1)=f(\pi({\bm1}'))$. There is then a unique way to extend $\tilde f$ to a functor such that $\theta$ becomes a natural transformation $\tilde f\Rightarrow f\circ\pi$, and this then acquires a unique strong symmetric monoidal structure such that $\theta$ is (symmetric) monoidal. It is easy to check that $\tilde f$ is strictly unital and satifies $\tilde f\circ\eta=f$, finishing the proof.
\end{proof}
\end{prop}

Thus, roughly speaking, $\cat{PermCat}$ is just as good as $\cat{SymMonCat}$ from a purely formal point of view. In practice, however, working with permutative categories is often easier than working with general symmetric monoidal categories as there are less coherence data to keep track of.

As a concrete manifestation of this, Schwede constructs in \cite[Construction~11.1]{schwede-k-theory} an explicit functor $\Phi\colon\cat{PermCat}\to\cat{ParSumCat}$ (partially recalled in Example~\ref{ex:Phi} above), and while it is plausible that his construction could be extended to all small symmetric monoidal categories, working out the details would probably become quite technical and cumbersome. Accordingly, the parsummable categories associated to general small symmetric monoidal categories are only defined indirectly by applying $\Phi$ to a permutative replacement.

If we look at parsummable categories from a categorical angle, then $\Phi$ is very well-behaved. Namely, we proved as the main result of \cite{perm-parsum-categorical}:

\begin{thm}\label{thm:categorical-comparison}
The functor $\Phi\colon\cat{PermCat}\to\cat{ParSumCat}$ is a homotopy equivalence with respect to the \emph{underlying equivalences of categories} on both sides.
\end{thm}
\begin{proof}
This is \cite[Theorem~3.1]{perm-parsum-categorical}.
\end{proof}

However, \emph{from a global perspective} $\Phi(\mathscr C)$ is not yet the `correct' parsummable category associated to $\mathscr C$. For example, \cite[Proposition~11.9]{schwede-k-theory} implies that if $\mathscr C$ is any small permutative replacement of the symmetric monoidal category of finite dimensional $\mathbb C$-vector spaces and $\mathbb C$-linear isomorphisms under $\oplus$, then the global $K$-theory of $\Phi(\mathscr C)$ is different from the usual definition of the global algebraic $K$-theory $\textbf{K}_{\textup{gl}}(\mathbb C)$ of the complex numbers. In order to avoid this issue, one applies the saturation construction first, so that the global $K$-theory of $\mathscr C$ is obtained by feeding $\Phi(\mathscr C)^\sat$ into Schwede's machinery.

Thus, if we write $\cat{$\bm G$-PermCat}$ for the category of $G$-objects in $\cat{PermCat}$, then it is actually the composition
\begin{equation}\label{eq:perm-cat-vs-parsumcat}
\cat{$\bm G$-PermCat}\xrightarrow{\Phi}\cat{$\bm G$-ParSumCat}\xrightarrow{(\blank)^{\sat}}\cat{$\bm G$-ParSumCat}
\end{equation}
that is the natural way to associate a $G$-parsummable category to a small $G$-permutative category, at least from a $G$-global point of view. We therefore want to prove:

\begin{thm}\label{thm:perm-cat-vs-parsumcat}
The composition $(\ref{eq:perm-cat-vs-parsumcat})$
is a homotopy equivalence with respect to the $G$-global weak equivalences on both sides.
\end{thm}

\begin{rk}
Since the inclusion defines a homotopy equivalence between $\cat{PermCat}$ and $\cat{SymMonCat}$, the above theorem is the appropriate $G$-global generalization of Theorem~\ref{thm:perm-vs-par-sum-cat} from the introduction. We can moreover conclude from it that if $\widehat\Phi\colon\cat{$\bm G$-SymMonCat}\to\cat{$\bm G$-ParSumCat}$ is any (hypothetical) extension of $\Phi$ respecting underlying equivalences of categories, then the homotopical functor $(\blank)^{\sat}\circ\widehat\Phi\colon\cat{$\bm G$-SymMonCat}\to\cat{$\bm G$-ParSumCat}$ is a homotopy equivalence with respect to the $G$-global weak equivalences on both sides, and analogously for $\cat{$\bm G$-SymMonCat}^0$.
\end{rk}

\begin{rk}
Elaborating on Schwede's argument cited above, we showed in \cite[Proposition~4.4]{perm-parsum-categorical} that there is no small permutative category $\mathscr C$ at all such that the global algebraic $K$-theory of $\Phi(\mathscr C)$ is equivalent to $\textbf{K}_{\textup{gl}}(\mathbb C)$, which in particular implies that there is no notion of weak equivalence on $\cat{PermCat}$ such that $\Phi$ becomes an equivalence of homotopy theories with respect to the global weak equivalences on $\cat{ParSumCat}$, and this even remains impossible when we pass to the larger class of those morphisms that induce global weak equivalences on $K$-theory. Thus, the passage to saturations is not a mere artifact of our proof (or our prejudices against the parsummable categories $\Phi(\mathscr C)$ and their global $K$-theory), but actually necessary.
\end{rk}

For the proof of the theorem we need the following lemma:

\begin{lemma}
The functor $(\blank)^{\textup{sat}}\circ\Phi\colon\cat{$\bm G$-PermCat}\to\cat{$\bm G$-ParSumCat}$ preserves and reflects $G$-global weak equivalences.
\begin{proof}
We fix a universal subgroup $H\subset\mathcal M$ together with a homomorphism $\varphi\colon H\to G$. If now $f\colon\mathscr C\to\mathscr D$ is a $G$-equivariant strict symmetric monoidal functor, then Theorem~\ref{thm:sat-cat} implies that $\Phi(f)^{\sat}$ induces a weak homotopy equivalence on $\varphi$-fixed points if and only if $\Phi(\mathscr C)^{\myh\varphi}\to\Phi(\mathscr D)^{\myh\varphi}$ is a weak homotopy equivalence, which is in turn equivalent by Lemma~\ref{lemma:comparison-homotopy-fp} to $(\forget\Phi(f))^{\myh\varphi}$ being a weak equivalence. Finally, we have natural equivalences of categories $\forget\Phi(\mathscr C)\simeq\mathscr C$, $\forget\Phi(\mathscr D)\simeq\mathscr D$ by \cite[Remark~11.4]{schwede-k-theory}, and these are automatically $G$-equivariant as the $G$-actions on the left hand sides are induced by functoriality of $\Phi$.

Thus, we altogether see that $\Phi(f)^\sat$ induces a weak homotopy equivalence on $\varphi$-fixed points if and only if $f^{\myh\varphi}$ is a weak homotopy equivalence. Letting $\varphi$ vary, this precisely yields the definitions of the $G$-global weak equivalences on $\cat{$\bm G$-ParSumCat}$ and $\cat{$\bm G$-PermCat}$, respectively, which completes the proof of the theorem.
\end{proof}
\end{lemma}

\begin{proof}[Proof of Theorem~\ref{thm:perm-cat-vs-parsumcat}]
By Theorem~\ref{thm:sat-cat}, the functor $(\ref{eq:perm-cat-vs-parsumcat})$ factors through the inclusion of the full subcategory $\cat{$\bm G$-ParSumCat}^s$; as the latter is a homotopy equivalence with respect to the $G$-global weak equivalences by Theorem~\ref{thm:parsum-cat-sat-global}, it is then enough to show that $(\ref{eq:perm-cat-vs-parsumcat})$ is a homotopy equivalence when viewed as a functor into $\cat{$\bm G$-ParSumCat}^s$. This is true with respect to the \emph{categorical equivalences} by Theorem~\ref{thm:categorical-comparison} together with Corollary~\ref{cor:sat-ps-cat}; moreover, the $G$-global weak equivalences on $\cat{$\bm G$-PermCat}$ are coarser than the categorical ones by Example~\ref{ex:categorical-are-g-global}, and so are the $G$-global weak equivalences on $\cat{$\bm G$-ParSumCat}^s$ by Lemma~\ref{lemma:cat-between-ws}. The claim follows as $(\blank)^\sat\circ\Phi$ preserves and reflects $G$-global weak equivalences by the previous lemma.
\end{proof}

\end{document}